\documentclass[preprint,11pt]{elsarticle}

\makeatletter
\def\ps@pprintTitle{%
 \let\@oddhead\@empty
 \let\@evenhead\@empty
 \def\@oddfoot{}%
 \let\@evenfoot\@oddfoot}
\makeatother

\date{}

\usepackage[english]{babel}
\usepackage{comment}
\usepackage{arydshln}
\usepackage{multirow}
\usepackage{subcaption}
\usepackage[letterpaper,top=2.5cm,bottom=2.5cm,left=2.5cm,right=2.5cm]{geometry}
\usepackage{float}
\usepackage{tikz}
\usepackage{tkz-graph}

\usepackage{amsmath,amssymb,amsthm,mathtools}
\usepackage{graphicx}
\usepackage[colorlinks=true, allcolors=blue]{hyperref}
\usepackage{tikz}
\usepackage{natbib}

\newtheorem{defi}{Definition}[section]
\theoremstyle{definition}

\newtheorem*{conj*}{Conjecture}
\newtheorem{coro}[defi]{Corollary}
\newtheorem{prop}[defi]{Proposition}
\newtheorem{lemma}[defi]{Lemma}
\newtheorem{theorem}[defi]{Theorem}

\newtheorem{example}[defi]{Example}
\newtheorem{examples}[defi]{Examples}
\newtheorem{remark}[defi]{Remark}

\newtheorem*{note1}{Note}
\newtheorem*{note2}{Notes}

\newcommand{\Gins}[1]{%
	\mathcal{G}\left[#1\right]%
}
\newcommand{\Cins}[1]{%
	\mathcal{C}^B{[#1]}%
}

\newcommand{\Chain}[1]{\mathcal{C}_*\left(#1\right)}
\newcommand{\ChainSub}[2]{\mathcal{C}_{#1}\left(#2\right)}

\newcommand{\ChainBlock}[1]{\mathcal{C}^B_*\left[#1\right]}
\newcommand{\ChainBlockSub}[2]{\mathcal{C}^B_{#1}\left[#2\right]}

\newcommand{\R}{\mathbb{R}}
\renewcommand{\S}{\mathbb{S}}
\newcommand{\Z}{\mathbb{Z}}
\newcommand{\N}{\mathbb{N}}
\newcommand{\E}{\mathcal{E}}
\newcommand{\Ev}{\mathcal{E}^\square}

\newcommand{\C}{\mathcal{C}}
\newcommand{\calF}{\mathcal{F}}

\newcommand{\sub}[2]{\mathcal{S}\!ub_{#1}\left(#2\right)}

\renewcommand{\star}{\text{star}}
\newcommand{\link}{\text{link}}

\newcommand{\SolVertices}[4]{\left\{
		u=#1, \quad v_1=#2, \quad  v_2=#3, \quad w=#4\right.}

\title{The Insertion Chain Complex: \\A Topological Approach to the Structure of Word Sets }
\author{Nata\v sa Jonoska, Francisco Martinez-Figueroa, Masahico Saito}

\begin{document}
\tikzset{
  cube/.pic={
    \tikzset{every edge quotes/.append style={midway,auto},/cube/.cd,#1}
    \def\csize{2*\cubesize}
    \draw [every edge/.append style={pic actions,densely dashed,opacity=.5},pic actions]
    (0,0,0) coordinate (o)
            edge coordinate [pos=1] (g) ++(0,0,-\csize)
        --++(0,\csize,0) coordinate (a)
        --++(\csize,0,0) coordinate (b)
        --++(0,-\csize,0) coordinate (c) -- cycle
    (a) --++(0,0,-\csize) coordinate (d) edge (g)
        --++(\csize,0,0) coordinate (e) -- (b) -- cycle
    (c) -- (b) -- (e)
        --++(0,-\csize,0) coordinate (f) edge (g) -- cycle;
    \path [every edge/.append style={pic actions, |-|}];
  },
  z={(-0.334,-0.334)},
  /cube/.search also={/tikz},
  /cube/.cd,
  size/.store in=\cubesize, size=1,
}

\newcommand{\Solutions}[6]{\begin{cases}
    x=#1, y=#2, z=#3\\
    x'=#4, y'=#5, z'=#6
\end{cases}}

\newcommand{\lineword}[9]{\mathord{%
\begin{tikzpicture}[baseline=-0.65ex,scale=1]
\node[circle,draw=white, fill=white, inner sep=0pt] (A) at (0,0) {$#1$};
\node[circle,draw=white, fill=white, inner sep=0pt] (B) at (#2/2,0) {$#3$};
\node[circle,draw=white, fill=white, inner sep=0pt] (F) at (#2/2+#4/2,0) {};

\node[circle,draw=white, fill=white, inner sep=0pt] (O1) at (0,-1em) {};
\node[circle,draw=white, fill=white, inner sep=0pt] (A1) at (#5/2,-1em) {$#6$};
\node[circle,draw=white, fill=white, inner sep=0pt] (B1) at (#5/2+#7/2,-1em) {$#8$};
\node[circle,draw=white, fill=white, inner sep=0pt] (F1) at (#5/2+#7/2+#9/2,-1em) {};

\draw (A)--(B);
\draw (B)--(F);
\draw (O1)--(A1);
\draw (A1)--(B1);
\draw (B1)--(F1);

\end{tikzpicture}}}

\newcommand{\CaseIa}{\lineword{a}{3}{b}{1}{1}{a}{1}{b}{2}}
\newcommand{\CaseIb}{\lineword{a}{3}{b}{0}{1}{a}{1}{b}{1}}
\newcommand{\CaseIc}{\mathord{%
\begin{tikzpicture}[baseline=-0.65ex,scale=1]
\node[circle,draw=white, fill=white, inner sep=0pt] (AA) at (0.5em,0) {$a$};
\node[circle,draw=white, fill=white, inner sep=0pt] (BB) at (10.5em,0) {$b$};
\node[inner sep=0pt] (A) at (1em,0) {};
\node[inner sep=0pt] (B) at (4em,0) {};
\node[inner sep=0pt] (C) at (4em,0) {};
\node[inner sep=0pt] (D) at (7em,0) {};
\node[inner sep=0pt] (E) at (7em,0) {};
\node[inner sep=0pt] (F) at (10em,0) {};

\node[circle,draw=white, fill=white, inner sep=0pt] (AA1) at (3.5em,-1em) {$a$};
\node[circle,draw=white, fill=white, inner sep=0pt] (BB1) at (7.5em,-1em) {$b$};
\node[inner sep=0pt] (A1) at (0em,-1em) {};
\node[inner sep=0pt] (B1) at (3em,-1em) {};
\node[inner sep=0pt] (C1) at (4em,-1em) {};
\node[inner sep=0pt] (D1) at (7em,-1em) {};
\node[inner sep=0pt] (E1) at (8em,-1em) {};
\node[inner sep=0pt] (F1) at (11em,-1em) {};

\draw (A) to node [above] {} (B);
\draw (C) to node [above] {} (D);
\draw (E) to node [above] {} (F);

\draw (A1) to node [below] {} (B1);
\draw (C1) to node [below] {} (D1);
\draw (E1) to node [below] {} (F1);

\draw (A)--(A1);
\draw (B)--(B1);
\draw (C)--(C1);
\draw (D)--(D1);
\draw (E)--(E1);
\draw (F)--(F1);
\end{tikzpicture}}}
\newcommand{\CaseId}{\mathord{%
\begin{tikzpicture}[baseline=-0.65ex,scale=1]
\node[circle,draw=white, fill=white, inner sep=0pt] (AA) at (0.5em,0) {$a$};
\node[circle,draw=white, fill=white, inner sep=0pt] (BB) at (10.5em,0) {$b$};
\node[inner sep=0pt] (A) at (1em,0) {};
\node[inner sep=0pt] (B) at (4em,0) {};
\node[inner sep=0pt] (C) at (4em,0) {};
\node[inner sep=0pt] (D) at (7em,0) {};
\node[inner sep=0pt] (E) at (7em,0) {};
\node[inner sep=0pt] (F) at (10em,0) {};

\node[circle,draw=white, fill=white, inner sep=0pt] (AA1) at (3.5em,-1em) {$a$};
\node[circle,draw=white, fill=white, inner sep=0pt] (BB1) at (7.5em,-1em) {$b$};
\node[inner sep=0pt] (A1) at (0em,-1em) {};
\node[inner sep=0pt] (B1) at (3em,-1em) {};
\node[inner sep=0pt] (C1) at (4em,-1em) {};
\node[inner sep=0pt] (D1) at (7em,-1em) {};
\node[inner sep=0pt] (E1) at (8em,-1em) {};
\node[inner sep=0pt] (F1) at (11em,-1em) {};

\draw (A) to node [above] {$a^r$} (B);
\draw (C) to node [above] {$\gamma$} (D);
\draw (E) to node [above] {$b^s$} (F);

\draw (A1) to node [below] {$a^r$} (B1);
\draw (C1) to node [below] {$\gamma$} (D1);
\draw (E1) to node [below] {$b^s$} (F1);

\draw (A)--(A1);
\draw (B)--(B1);
\draw (C)--(C1);
\draw (D)--(D1);
\draw (E)--(E1);
\draw (F)--(F1);
\end{tikzpicture}}}

\newcommand{\CaseIIa}{\lineword{a}{2}{b}{5}{3}{a}{2}{b}{2}}
\newcommand{\CaseIIb}{\lineword{a}{2}{b}{3}{3}{a}{2}{b}{0}}
\newcommand{\CaseIIc}{\mathord{%
\begin{tikzpicture}[baseline=-0.65ex,scale=1]
\node[circle,draw=white, fill=white, inner sep=0pt] (AA) at (0.5em,0.1em) {$a$};
\node[circle,draw=white, fill=white, inner sep=0pt] (BB) at (4.25em,0.1em) {$b$};
\node[inner sep=0pt] (A) at (1em,0) {};
\node[inner sep=0pt] (B) at (4em,0) {};
\node[inner sep=0pt] (C) at (5em,0) {};
\node[inner sep=0pt] (D) at (8em,0) {};
\node[inner sep=0pt] (E) at (8em,0) {};
\node[inner sep=0pt] (F) at (11em,0) {};

\node[circle,draw=white, fill=white, inner sep=0pt] (AA1) at (6.5em,-1.1em) {$a$};
\node[circle,draw=white, fill=white, inner sep=0pt] (BB1) at (10.5em,-1.1em) {$b$};
\node[inner sep=0pt] (A1) at (0em,-1em) {};
\node[inner sep=0pt] (B1) at (3em,-1em) {};
\node[inner sep=0pt] (C1) at (3em,-1em) {};
\node[inner sep=0pt] (D1) at (6em,-1em) {};
\node[inner sep=0pt] (E1) at (7em,-1em) {};
\node[inner sep=0pt] (F1) at (10em,-1em) {};

\draw (A) to node [above] {} (B);
\draw (C) to node [above] {} (D);
\draw (E) to node [above] {} (F);

\draw (A1) to node [below] {} (B1);
\draw (C1) to node [below] {} (D1);
\draw (E1) to node [below] {} (F1);

\draw (A)--(A1);
\draw (B)--(B1);
\draw (C)--(C1);
\draw (D)--(D1);
\draw (E)--(E1);
\draw (F)--(F1);
\end{tikzpicture}}}
\newcommand{\CaseIId}{\mathord{%
\begin{tikzpicture}[baseline=-0.65ex,scale=1]
\node[circle,draw=white, fill=white, inner sep=0pt] (AA) at (0.5em,0.1em) {$a$};
\node[circle,draw=white, fill=white, inner sep=0pt] (BB) at (4.25em,0.1em) {$b$};
\node[inner sep=0pt] (A) at (1em,0) {};
\node[inner sep=0pt] (B) at (4em,0) {};
\node[inner sep=0pt] (C) at (5em,0) {};
\node[inner sep=0pt] (D) at (8em,0) {};
\node[inner sep=0pt] (E) at (8em,0) {};
\node[inner sep=0pt] (F) at (11em,0) {};

\node[circle,draw=white, fill=white, inner sep=0pt] (AA1) at (6.5em,-1.1em) {$a$};
\node[circle,draw=white, fill=white, inner sep=0pt] (BB1) at (10.5em,-1.1em) {$b$};
\node[inner sep=0pt] (A1) at (0em,-1em) {};
\node[inner sep=0pt] (B1) at (3em,-1em) {};
\node[inner sep=0pt] (C1) at (3em,-1em) {};
\node[inner sep=0pt] (D1) at (6em,-1em) {};
\node[inner sep=0pt] (E1) at (7em,-1em) {};
\node[inner sep=0pt] (F1) at (10em,-1em) {};

\draw (A) to node [above] {$a^r$} (B);
\draw (C) to node [above] {} (D);
\draw (E) to node [above] {$b^s$} (F);

\draw (A1) to node [below] {$a^r$} (B1);
\draw (C1) to node [below] {} (D1);
\draw (E1) to node [below] {$b^s$} (F1);

\draw (A)--(A1);
\draw (B)--(B1);
\draw (C)--(C1);
\draw (D)--(D1);
\draw (E)--(E1);
\draw (F)--(F1);
\end{tikzpicture}}}
\newcommand{\CaseIIe}{\mathord{%
\begin{tikzpicture}[baseline=-0.65ex,scale=1]
\node[circle,draw=white, fill=white, inner sep=0pt] (AA) at (3.75em,0.1em) {$ab$};
\node[inner sep=0pt] (A) at (0em,0) {};
\node[inner sep=0pt] (B) at (3em,0) {};
\node[inner sep=0pt] (C) at (5em,0) {};
\node[inner sep=0pt] (D) at (8em,0) {};
\node[inner sep=0pt] (E) at (8em,0) {};
\node[inner sep=0pt] (F) at (11em,0) {};

\node[circle,draw=white, fill=white, inner sep=0pt] (AA1) at (7.25em,-1.1em) {$ab$};
\node[inner sep=0pt] (A1) at (0em,-1em) {};
\node[inner sep=0pt] (B1) at (3em,-1em) {};
\node[inner sep=0pt] (C1) at (3em,-1em) {};
\node[inner sep=0pt] (D1) at (6em,-1em) {};
\node[inner sep=0pt] (E1) at (8em,-1em) {};
\node[inner sep=0pt] (F1) at (11em,-1em) {};

\draw (A) to node [above] {$a^r$} (B);
\draw (C) to node [above] {} (D);
\draw (E) to node [above] {$b^s$} (F);

\draw (A1) to node [below] {$a^r$} (B1);
\draw (C1) to node [below] {} (D1);
\draw (E1) to node [below] {$b^s$} (F1);

\draw (A)--(A1);
\draw (B)--(B1);
\draw (C)--(C1);
\draw (D)--(D1);
\draw (E)--(E1);
\draw (F)--(F1);
\end{tikzpicture}}}
\newcommand{\CaseIIf}{\mathord{%
\begin{tikzpicture}[baseline=-0.65ex,scale=1]
\node[circle,draw=white, fill=white, inner sep=0pt] (AA) at (3.75em,0.1em) {$ab$};
\node[inner sep=0pt] (A) at (0em,0) {};
\node[inner sep=0pt] (B) at (3em,0) {};
\node[inner sep=0pt] (C) at (5em,0) {};
\node[inner sep=0pt] (D) at (8em,0) {};
\node[inner sep=0pt] (E) at (8em,0) {};
\node[inner sep=0pt] (F) at (11em,0) {};

\node[circle,draw=white, fill=white, inner sep=0pt] (AA1) at (7.25em,-1.1em) {$ab$};
\node[inner sep=0pt] (A1) at (0em,-1em) {};
\node[inner sep=0pt] (B1) at (3em,-1em) {};
\node[inner sep=0pt] (C1) at (3em,-1em) {};
\node[inner sep=0pt] (D1) at (6em,-1em) {};
\node[inner sep=0pt] (E1) at (8em,-1em) {};
\node[inner sep=0pt] (F1) at (11em,-1em) {};

\draw (A) to node [above] {$a^r$} (B);
\draw (C) to node [above] {$(ab)^t$} (D);
\draw (E) to node [above] {$b^s$} (F);

\draw (A1) to node [below] {$a^r$} (B1);
\draw (C1) to node [below] {$(ab)^t$} (D1);
\draw (E1) to node [below] {$b^s$} (F1);

\draw (A)--(A1);
\draw (B)--(B1);
\draw (C)--(C1);
\draw (D)--(D1);
\draw (E)--(E1);
\draw (F)--(F1);
\end{tikzpicture}}}

\newcommand{\CaseIIIa}{\lineword{a}{2}{b}{2}{1}{a}{2}{b}{1}}
\newcommand{\CaseIIIb}{\lineword{a}{2}{b}{1}{1}{a}{2}{b}{0}}
\newcommand{\CaseIIIc}{\mathord{%
\begin{tikzpicture}[baseline=-0.65ex,scale=1]
\node[circle,draw=white, fill=white, inner sep=0pt] (AA) at (0.5em,0.1em) {$a$};
\node[circle,draw=white, fill=white, inner sep=0pt] (BB) at (7.5em,0.1em) {$b$};
\node[inner sep=0pt] (A) at (1em,0) {};
\node[inner sep=0pt] (B) at (4em,0) {};
\node[inner sep=0pt] (C) at (4em,0) {};
\node[inner sep=0pt] (D) at (7em,0) {};
\node[inner sep=0pt] (E) at (8em,0) {};
\node[inner sep=0pt] (F) at (11em,0) {};

\node[circle,draw=white, fill=white, inner sep=0pt] (AA1) at (3.5em,-1.1em) {$a$};
\node[circle,draw=white, fill=white, inner sep=0pt] (BB1) at (10.5em,-1.1em) {$b$};
\node[inner sep=0pt] (A1) at (0em,-1em) {};
\node[inner sep=0pt] (B1) at (3em,-1em) {};
\node[inner sep=0pt] (C1) at (4em,-1em) {};
\node[inner sep=0pt] (D1) at (7em,-1em) {};
\node[inner sep=0pt] (E1) at (7em,-1em) {};
\node[inner sep=0pt] (F1) at (10em,-1em) {};

\draw (A) to node [above] {} (B);
\draw (C) to node [above] {} (D);
\draw (E) to node [above] {} (F);

\draw (A1) to node [below] {} (B1);
\draw (C1) to node [below] {} (D1);
\draw (E1) to node [below] {} (F1);

\draw (A)--(A1);
\draw (B)--(B1);
\draw (C)--(C1);
\draw (D)--(D1);
\draw (E)--(E1);
\draw (F)--(F1);
\end{tikzpicture}}}
\newcommand{\CaseIIId}{\mathord{%
\begin{tikzpicture}[baseline=-0.65ex,scale=1]
\node[circle,draw=white, fill=white, inner sep=0pt] (AA) at (0.5em,0.1em) {$a$};
\node[circle,draw=white, fill=white, inner sep=0pt] (BB) at (7.5em,0.1em) {$b$};
\node[inner sep=0pt] (A) at (1em,0) {};
\node[inner sep=0pt] (B) at (4em,0) {};
\node[inner sep=0pt] (C) at (4em,0) {};
\node[inner sep=0pt] (D) at (7em,0) {};
\node[inner sep=0pt] (E) at (8em,0) {};
\node[inner sep=0pt] (F) at (11em,0) {};

\node[circle,draw=white, fill=white, inner sep=0pt] (AA1) at (3.5em,-1.1em) {$a$};
\node[circle,draw=white, fill=white, inner sep=0pt] (BB1) at (10.5em,-1.1em) {$b$};
\node[inner sep=0pt] (A1) at (0em,-1em) {};
\node[inner sep=0pt] (B1) at (3em,-1em) {};
\node[inner sep=0pt] (C1) at (4em,-1em) {};
\node[inner sep=0pt] (D1) at (7em,-1em) {};
\node[inner sep=0pt] (E1) at (7em,-1em) {};
\node[inner sep=0pt] (F1) at (10em,-1em) {};

\draw (A) to node [above] {$a^r$} (B);
\draw (C) to node [above] {$\gamma$} (D);
\draw (E) to node [above] {$b^s$} (F);

\draw (A1) to node [below] {$a^r$} (B1);
\draw (C1) to node [below] {$\gamma$} (D1);
\draw (E1) to node [below] {$b^s$} (F1);

\draw (A)--(A1);
\draw (B)--(B1);
\draw (C)--(C1);
\draw (D)--(D1);
\draw (E)--(E1);
\draw (F)--(F1);
\end{tikzpicture}}}

\newcommand{\CaseIVa}{\lineword{a}{3}{b}{0}{1}{b}{1}{a}{1}}
\newcommand{\CaseIVb}{\mathord{%
\begin{tikzpicture}[baseline=-0.65ex,scale=1]
\node[circle,draw=white, fill=white, inner sep=0pt] (AA) at (0.5em,0.1em) {$a$};
\node[circle,draw=white, fill=white, inner sep=0pt] (BB) at (10.5em,0.1em) {$b$};
\node[inner sep=0pt] (A) at (1em,0) {};
\node[inner sep=0pt] (B) at (4em,0) {};
\node[inner sep=0pt] (C) at (4em,0) {};
\node[inner sep=0pt] (D) at (7em,0) {};
\node[inner sep=0pt] (E) at (7em,0) {};
\node[inner sep=0pt] (F) at (10em,0) {};

\node[circle,draw=white, fill=white, inner sep=0pt] (AA1) at (3.5em,-1.1em) {$b$};
\node[circle,draw=white, fill=white, inner sep=0pt] (BB1) at (7.5em,-1.1em) {$a$};
\node[inner sep=0pt] (A1) at (0em,-1em) {};
\node[inner sep=0pt] (B1) at (3em,-1em) {};
\node[inner sep=0pt] (C1) at (4em,-1em) {};
\node[inner sep=0pt] (D1) at (7em,-1em) {};
\node[inner sep=0pt] (E1) at (8em,-1em) {};
\node[inner sep=0pt] (F1) at (11em,-1em) {};

\draw (A) to node [above] {} (B);
\draw (C) to node [above] {} (D);
\draw (E) to node [above] {} (F);

\draw (A1) to node [below] {} (B1);
\draw (C1) to node [below] {} (D1);
\draw (E1) to node [below] {} (F1);

\draw (A)--(A1);
\draw (B)--(B1);
\draw (C)--(C1);
\draw (D)--(D1);
\draw (E)--(E1);
\draw (F)--(F1);
\end{tikzpicture}}}
\newcommand{\CaseIVc}{\mathord{%
\begin{tikzpicture}[baseline=-0.65ex,scale=1]
\node[circle,draw=white, fill=white, inner sep=0pt] (AA) at (0.5em,0.1em) {$a$};
\node[circle,draw=white, fill=white, inner sep=0pt] (BB) at (10.5em,0.1em) {$a$};
\node[inner sep=0pt] (A) at (1em,0) {};
\node[inner sep=0pt] (B) at (4em,0) {};
\node[inner sep=0pt] (C) at (4em,0) {};
\node[inner sep=0pt] (D) at (7em,0) {};
\node[inner sep=0pt] (E) at (7em,0) {};
\node[inner sep=0pt] (F) at (10em,0) {};

\node[circle,draw=white, fill=white, inner sep=0pt] (AA1) at (3.5em,-1.1em) {$a$};
\node[circle,draw=white, fill=white, inner sep=0pt] (BB1) at (7.5em,-1.1em) {$a$};
\node[inner sep=0pt] (A1) at (0em,-1em) {};
\node[inner sep=0pt] (B1) at (3em,-1em) {};
\node[inner sep=0pt] (C1) at (4em,-1em) {};
\node[inner sep=0pt] (D1) at (7em,-1em) {};
\node[inner sep=0pt] (E1) at (8em,-1em) {};
\node[inner sep=0pt] (F1) at (11em,-1em) {};

\draw (A) to node [above] {$a^r$} (B);
\draw (C) to node [above] {$\gamma$} (D);
\draw (E) to node [above] {$a^s$} (F);

\draw (A1) to node [below] {$a^r$} (B1);
\draw (C1) to node [below] {$\gamma$} (D1);
\draw (E1) to node [below] {$a^s$} (F1);

\draw (A)--(A1);
\draw (B)--(B1);
\draw (C)--(C1);
\draw (D)--(D1);
\draw (E)--(E1);
\draw (F)--(F1);
\end{tikzpicture}}}

\newcommand{\CaseVa}{\lineword{a}{1}{b}{2}{2}{b}{1}{a}{0}}
\newcommand{\CaseVb}{\mathord{%
\begin{tikzpicture}[baseline=-0.65ex,scale=1]
\node[circle,draw=white, fill=white, inner sep=0pt] (AA) at (0.5em,0.1em) {$a$};
\node[circle,draw=white, fill=white, inner sep=0pt] (BB) at (4.25em,0.1em) {$b$};
\node[inner sep=0pt] (A) at (1em,0) {};
\node[inner sep=0pt] (B) at (4em,0) {};
\node[inner sep=0pt] (C) at (5em,0) {};
\node[inner sep=0pt] (D) at (8em,0) {};
\node[inner sep=0pt] (E) at (8em,0) {};
\node[inner sep=0pt] (F) at (11em,0) {};

\node[circle,draw=white, fill=white, inner sep=0pt] (AA1) at (6.75em,-1.1em) {$b$};
\node[circle,draw=white, fill=white, inner sep=0pt] (BB1) at (10.25em,-1.1em) {$a$};
\node[inner sep=0pt] (A1) at (0em,-1em) {};
\node[inner sep=0pt] (B1) at (3em,-1em) {};
\node[inner sep=0pt] (C1) at (3em,-1em) {};
\node[inner sep=0pt] (D1) at (6em,-1em) {};
\node[inner sep=0pt] (E1) at (7em,-1em) {};
\node[inner sep=0pt] (F1) at (10em,-1em) {};

\draw (A) to node [above] {} (B);
\draw (C) to node [above] {} (D);
\draw (E) to node [above] {} (F);

\draw (A1) to node [below] {} (B1);
\draw (C1) to node [below] {} (D1);
\draw (E1) to node [below] {} (F1);

\draw (A)--(A1);
\draw (B)--(B1);
\draw (C)--(C1);
\draw (D)--(D1);
\draw (E)--(E1);
\draw (F)--(F1);
\end{tikzpicture}}}
\newcommand{\CaseVc}{\mathord{%
\begin{tikzpicture}[baseline=-0.65ex,scale=1]
\node[circle,draw=white, fill=white, inner sep=0pt] (AA) at (0.5em,0.1em) {$a$};
\node[circle,draw=white, fill=white, inner sep=0pt] (BB) at (4.25em,0.1em) {$b$};
\node[inner sep=0pt] (A) at (1em,0) {};
\node[inner sep=0pt] (B) at (4em,0) {};
\node[inner sep=0pt] (C) at (5em,0) {};
\node[inner sep=0pt] (D) at (8em,0) {};
\node[inner sep=0pt] (E) at (8em,0) {};
\node[inner sep=0pt] (F) at (11em,0) {};

\node[circle,draw=white, fill=white, inner sep=0pt] (AA1) at (6.75em,-1.1em) {$b$};
\node[circle,draw=white, fill=white, inner sep=0pt] (BB1) at (10.25em,-1.1em) {$a$};
\node[inner sep=0pt] (A1) at (0em,-1em) {};
\node[inner sep=0pt] (B1) at (3em,-1em) {};
\node[inner sep=0pt] (C1) at (3em,-1em) {};
\node[inner sep=0pt] (D1) at (6em,-1em) {};
\node[inner sep=0pt] (E1) at (7em,-1em) {};
\node[inner sep=0pt] (F1) at (10em,-1em) {};

\draw (A) to node [above] {$a^r$} (B);
\draw (C) to node [above] {} (D);
\draw (E) to node [above] {$a^s$} (F);

\draw (A1) to node [below] {$a^r$} (B1);
\draw (C1) to node [below] {} (D1);
\draw (E1) to node [below] {$a^s$} (F1);

\draw (A)--(A1);
\draw (B)--(B1);
\draw (C)--(C1);
\draw (D)--(D1);
\draw (E)--(E1);
\draw (F)--(F1);
\end{tikzpicture}}}
\newcommand{\CaseVd}{\mathord{%
\begin{tikzpicture}[baseline=-0.65ex,scale=1]
\node[circle,draw=white, fill=white, inner sep=0pt] (AA) at (3.7em,0.1em) {$ab$};
\node[inner sep=0pt] (A) at (0em,0) {};
\node[inner sep=0pt] (B) at (3em,0) {};
\node[inner sep=0pt] (C) at (4.5em,0) {};
\node[inner sep=0pt] (D) at (7em,0) {};
\node[inner sep=0pt] (E) at (7em,0) {};
\node[inner sep=0pt] (F) at (10em,0) {};

\node[circle,draw=white, fill=white, inner sep=0pt] (AA1) at (6.5em,-1.1em) {$ba$};
\node[inner sep=0pt] (A1) at (0em,-1em) {};
\node[inner sep=0pt] (B1) at (3em,-1em) {};
\node[inner sep=0pt] (C1) at (3em,-1em) {};
\node[inner sep=0pt] (D1) at (5.5em,-1em) {};
\node[inner sep=0pt] (E1) at (7em,-1em) {};
\node[inner sep=0pt] (F1) at (10em,-1em) {};

\draw (A) to node [above] {$a^r$} (B);
\draw (C) to node [above] {} (D);
\draw (E) to node [above] {$a^s$} (F);

\draw (A1) to node [below] {$a^r$} (B1);
\draw (C1) to node [below] {} (D1);
\draw (E1) to node [below] {$a^s$} (F1);

\draw (A)--(A1);
\draw (B)--(B1);
\draw (C)--(C1);
\draw (D)--(D1);
\draw (E)--(E1);
\draw (F)--(F1);
\end{tikzpicture}}}
\newcommand{\CaseVe}{\mathord{%
\begin{tikzpicture}[baseline=-0.65ex,scale=1]
\node[circle,draw=white, fill=white, inner sep=0pt] (AA) at (3.7em,0.1em) {$ab$};
\node[inner sep=0pt] (A) at (0em,0) {};
\node[inner sep=0pt] (B) at (3em,0) {};
\node[inner sep=0pt] (C) at (4.5em,0) {};
\node[inner sep=0pt] (D) at (7em,0) {};
\node[inner sep=0pt] (E) at (7em,0) {};
\node[inner sep=0pt] (F) at (10em,0) {};

\node[circle,draw=white, fill=white, inner sep=0pt] (AA1) at (6.5em,-1.1em) {$ba$};
\node[inner sep=0pt] (A1) at (0em,-1em) {};
\node[inner sep=0pt] (B1) at (3em,-1em) {};
\node[inner sep=0pt] (C1) at (3em,-1em) {};
\node[inner sep=0pt] (D1) at (5.5em,-1em) {};
\node[inner sep=0pt] (E1) at (7em,-1em) {};
\node[inner sep=0pt] (F1) at (10em,-1em) {};

\draw (A) to node [above] {$a^r$} (B);
\draw (C) to node [above] {$(ab)^ta$} (D);
\draw (E) to node [above] {$a^s$} (F);

\draw (A1) to node [below] {$a^r$} (B1);
\draw (C1) to node [below] {$(ab)^ta$} (D1);
\draw (E1) to node [below] {$a^s$} (F1);

\draw (A)--(A1);
\draw (B)--(B1);
\draw (C)--(C1);
\draw (D)--(D1);
\draw (E)--(E1);
\draw (F)--(F1);
\end{tikzpicture}}}

\newcommand{\CaseVIa}{\lineword{a}{2}{b}{1}{1}{b}{2}{a}{0}}
\newcommand{\CaseVIb}{\mathord{%
\begin{tikzpicture}[baseline=-0.65ex,scale=1]
\node[circle,draw=white, fill=white, inner sep=0pt] (AA) at (0.5em,0.1em) {$a$};
\node[circle,draw=white, fill=white, inner sep=0pt] (BB) at (7.5em,0.1em) {$b$};
\node[inner sep=0pt] (A) at (1em,0) {};
\node[inner sep=0pt] (B) at (4em,0) {};
\node[inner sep=0pt] (C) at (4em,0) {};
\node[inner sep=0pt] (D) at (7em,0) {};
\node[inner sep=0pt] (E) at (8em,0) {};
\node[inner sep=0pt] (F) at (11em,0) {};

\node[circle,draw=white, fill=white, inner sep=0pt] (AA1) at (3.5em,-1.1em) {$b$};
\node[circle,draw=white, fill=white, inner sep=0pt] (BB1) at (10.5em,-1.1em) {$a$};
\node[inner sep=0pt] (A1) at (0em,-1em) {};
\node[inner sep=0pt] (B1) at (3em,-1em) {};
\node[inner sep=0pt] (C1) at (4em,-1em) {};
\node[inner sep=0pt] (D1) at (7em,-1em) {};
\node[inner sep=0pt] (E1) at (7em,-1em) {};
\node[inner sep=0pt] (F1) at (10em,-1em) {};

\draw (A) to node [above] {} (B);
\draw (C) to node [above] {} (D);
\draw (E) to node [above] {} (F);

\draw (A1) to node [below] {} (B1);
\draw (C1) to node [below] {} (D1);
\draw (E1) to node [below] {} (F1);

\draw (A)--(A1);
\draw (B)--(B1);
\draw (C)--(C1);
\draw (D)--(D1);
\draw (E)--(E1);
\draw (F)--(F1);
\end{tikzpicture}}}
\newcommand{\CaseVIc}{\mathord{%
\begin{tikzpicture}[baseline=-0.65ex,scale=1]
\node[circle,draw=white, fill=white, inner sep=0pt] (AA) at (0.5em,0.1em) {$a$};
\node[circle,draw=white, fill=white, inner sep=0pt] (BB) at (7.5em,0.1em) {$a$};
\node[inner sep=0pt] (A) at (1em,0) {};
\node[inner sep=0pt] (B) at (4em,0) {};
\node[inner sep=0pt] (C) at (4em,0) {};
\node[inner sep=0pt] (D) at (7em,0) {};
\node[inner sep=0pt] (E) at (8em,0) {};
\node[inner sep=0pt] (F) at (11em,0) {};

\node[circle,draw=white, fill=white, inner sep=0pt] (AA1) at (3.5em,-1.1em) {$a$};
\node[circle,draw=white, fill=white, inner sep=0pt] (BB1) at (10.5em,-1.1em) {$a$};
\node[inner sep=0pt] (A1) at (0em,-1em) {};
\node[inner sep=0pt] (B1) at (3em,-1em) {};
\node[inner sep=0pt] (C1) at (4em,-1em) {};
\node[inner sep=0pt] (D1) at (7em,-1em) {};
\node[inner sep=0pt] (E1) at (7em,-1em) {};
\node[inner sep=0pt] (F1) at (10em,-1em) {};

\draw (A) to node [above] {$a^r$} (B);
\draw (C) to node [above] {$\gamma$} (D);
\draw (E) to node [above] {$a^s$} (F);

\draw (A1) to node [below] {$a^r$} (B1);
\draw (C1) to node [below] {$\gamma$} (D1);
\draw (E1) to node [below] {$a^s$} (F1);

\draw (A)--(A1);
\draw (B)--(B1);
\draw (C)--(C1);
\draw (D)--(D1);
\draw (E)--(E1);
\draw (F)--(F1);
\end{tikzpicture}}}

\newcommand{\CasesItoIII}{%
\begin{align*}
\textbf{Case I}\\
    \CaseIb &= \CaseIc =\CaseId\\
    &\Rightarrow \Solutions{1}{a^r\gamma b^s}{1}{a^r}{\gamma}{b^s}\\
\textbf{Case II}\\
   \CaseIIb&=\CaseIIc=\CaseIId\\
   &=\CaseIIe=\CaseIIf\\
   &\Rightarrow \Solutions{1}{a^r}{(ab)^tb^s}{a^r(ab)^t}{b^s}{1}\\
\textbf{Case III}\\
    \CaseIIIb &= \CaseIIIc =\CaseIIId\\
    &\Rightarrow \Solutions{1}{a^r\gamma}{b^s}{a^r}{\gamma b^s}{1}
\end{align*}}

\newcommand{\CasesIVtoVI}{%
\begin{align*}
\textbf{Case IV}\\
    \CaseIVa &= \CaseIVb =\CaseIVc\\
    &\Rightarrow a=b, \text{ and } \Solutions{1}{a^r\gamma a^s}{1}{a^r}{\gamma}{a^s}\\
\textbf{Case V}\\
    \CaseVa &=\CaseVb =\CaseVc\\
    &=\CaseVd =\CaseVe\\
    &\Rightarrow \Solutions{1}{a^r}{(ab)^ta^{s+1}}{a^{r+1}(ba)^t}{a^s}{1}\\
\textbf{Case VI}\\
    \CaseVIa &=\CaseVIb =\CaseVIc\\
    &\Rightarrow a=b, \text{ and } \Solutions{1}{a^r\gamma}{a^s}{a^r}{\gamma a^s}{1}
\end{align*}}

\newcommand{\TableA}{%
\begin{tabular}{|l|c|l|l|}
\hline
\multicolumn{1}{|c|}{\textbf{Insertion Pattern}} &
  \textbf{Solution  Case} &
  \multicolumn{1}{c|}{\textbf{Vertices}} &
  \multicolumn{1}{c|}{\textbf{Vertices w/o suffixes or prefixes}} \\ \hline
\multicolumn{1}{|c|}{}  & I   &  &  Not a solution since $v_1=v_2$. \\ \cline{2-4} 
                        & II  &  &  \\ \cline{2-4} 
\multicolumn{1}{|c|}{A} & III &  &  \\ \cline{2-4} 
                        & IV  &  &  \\ \cline{2-4} 
                        & V   &  &  \\ \cline{2-4} 
                        & VI  &  &  \\ \hline
\end{tabular}}

\newcommand{\FigureFilledSquare}{%
\begin{tikzpicture}
    \node at (0,0) (A) {$ab$} ; 
    \node at (2.5,0) (B) {$a^2b$} ; 
    \node at (0,2.5) (D) {$bab$} ; 
    \node at (2.5,2.5) (C) {$abab$} ; 

    \draw [fill=gray, opacity=0.3] (A) rectangle (C);
    \path[->] (A) edge node[auto=left] {$a$} (B);
    \path[->] (A) edge node[auto=left] {$b$} (D);
    \path[->] (B) edge node[auto=left] {$b$} (C);
    \path[->] (D) edge node[auto=left] {$a$} (C);
\end{tikzpicture}%
\begin{tikzpicture}
    \node at (0,0) (A) {$ab$} ; 
    \node at (2.5,0) (B) {$aba$} ; 
    \node at (0,2.5) (D) {$bab$} ; 
    \node at (2.5,2.5) (C) {$abab$} ; 

    \path[->] (A) edge node[auto=left] {$a$} (B);
    \path[->] (A) edge node[auto=left] {$b$} (D);
    \path[->] (B) edge node[auto=left] {$b$} (C);
    \path[->] (D) edge node[auto=left] {$a$} (C);
\end{tikzpicture}%
}

\newcommand{\CubicalExamples}{%
\begin{tikzpicture}
    \node at (0,0) (A) {$1$} ; 
    \node at (2.5,0) (B) {$a$} ; 
    \node at (0,2.5) (D) {$b$} ; 
    \node at (2.5,2.5) (C) {$ab$} ; 

    \draw [fill=gray, opacity=0.3] (A) rectangle (C);
    \path[->] (A) edge node[auto=left] {} (B);
    \path[->] (A) edge node[auto=left] {} (D);
    \path[->] (B) edge node[auto=left] {} (C);
    \path[->] (D) edge node[auto=left] {} (C);
\end{tikzpicture}%
\begin{tikzpicture}
  \pic[fill=gray!75] at (0,0,0) {cube={size=1}};
  \node[below,inner sep=2pt] at (0,0,0) {$1$};
  \node[below,inner sep=2pt] at (2,0,0) {$a$};
  \node[left,inner sep=3pt] at (0,2,0) {$c$};
  \node[below,inner sep=2pt] at (0,0,-2) {$b$};
  \node[below,inner sep=2pt] at (2,0,-2) {$ab$};
  \node[above,inner sep=2pt] at (0,2,-2) {$bc$};
  \node[right,inner sep=2pt] at (2,2,0) {$ac$};
  \node[above,inner sep=2pt] at (2,2,-2) {$abc$};
\end{tikzpicture}%
\begin{tikzpicture}
    \node[right] at (0,0) (A) {$1$} ; 
    \node at (0,1.5) (B) {$a$} ; 
    \node at (2,0) (C) {$ab$} ; 
    \node at (0,-1.5) (D) {$b$} ; 
    \node at (-2,0) (E) {$ba$} ; 
    \node[circle,fill,inner sep=1pt] at (0,0) {};

    \fill[gray, opacity=0.3] (A.center) -- (B.center) -- (C.center) -- (D.center) -- (A.center) -- cycle;

    \fill[blue!30, opacity=0.3] (A.center) -- (B.center) -- (E.center) -- (D.center) -- (A.center) -- cycle;

    \path[->] (A) edge node[auto=left] {} (B);
    \path[->] (A) edge node[auto=left] {} (D);
    \path[->] (B) edge node[auto=left] {} (C);
    \path[->] (D) edge node[auto=left] {} (C);
    \path[->] (B) edge node[auto=left] {} (E);
    \path[->] (D) edge node[auto=left] {} (E);  
\end{tikzpicture}%
}

\newcommand{\ExampleGIns}{%
\begin{tikzpicture}[scale=0.7]
\GraphInit[vstyle=Classic]
\SetVertexNormal[MinSize=2pt, FillColor=black]
\tikzset{
  EdgeStyle/.append style={
    ->,
    shorten >=2pt                    
  }
}
\Vertex[Lpos=90,L=\scalebox{1}{$aaab$},x=0.4042cm,y=3.5103cm]{v3}
\Vertex[L=\scalebox{1}{$aab$},x=2.0163cm,y=3.5119cm]{v5}
\Vertex[Lpos=10,L=\scalebox{1}{$aabb$},x=1.6464cm,y=1.3431cm]{v6}
\Vertex[Lpos=180,L=\scalebox{1}{$aaabb$},x=0.2741cm,y=2.0927cm]{v4}

\Vertex[L=\scalebox{1}{$aabbb$},x=0.0cm,y=0.5747cm]{v7}

\Vertex[Lpos=90,L=\scalebox{1}{$aa$},x=3.7451cm,y=4.876cm]{v2}
\Vertex[Lpos=90,L=\scalebox{1}{$a$},x=5.7715cm,y=5.0cm]{v1}
\Vertex[L=\scalebox{1}{$1$},x=7.6416cm,y=4.9206cm]{v0}

\Vertex[L=\scalebox{1}{$b$},x=9.2cm,y=3.5155cm]{v12}
\Vertex[L=\scalebox{1}{$bc$},x=8.1906cm,y=1.9352cm]{v14}
\Vertex[L=\scalebox{1}{$bbc$},x=7.4088cm,y=0.0cm]{v13}
\Vertex[Lpos=-90,L=\scalebox{1}{$abbc$},x=5.6477cm,y=0.0914cm]{v9}
\Vertex[Lpos=-90,L=\scalebox{1}{$abb$},x=3.4874cm,y=0.2187cm]{v8}

\Vertex[Lpos=100,L=\scalebox{1}{$ac$},x=5cm,y=3.6194cm]{v11}
\Vertex[Lpos=90,L=\scalebox{1}{$c$},x=6.8cm,y=3.5cm]{v15}
\Vertex[Lpos=170,L=\scalebox{1}{$abc$},x=5cm,y=1.593cm]{v10}

\Vertex[L=\scalebox{1}{$cc$},x=7.8cm,y=3.7cm]{v16}
\Vertex[L=\scalebox{1}{$ab$},x=6.33cm,y=2.57cm]{v17}

\Edge[](v0)(v1) 
\Edge[](v0)(v12)
\Edge[](v0)(v15)
\Edge[](v1)(v2)
\Edge[](v1)(v11)
\Edge[](v2)(v5)
\Edge[](v3)(v4)
\Edge[](v5)(v3)
\Edge[](v6)(v4)
\Edge[](v5)(v6)
\Edge[](v6)(v7)
\Edge[](v8)(v6)
\Edge[](v8)(v9)
\Edge[](v10)(v9)
\Edge[](v13)(v9)
\Edge[](v11)(v10)
\Edge[](v14)(v10)
\Edge[](v15)(v11)
\Edge[](v12)(v14)
\Edge[](v14)(v13)
\Edge[](v15)(v14)
\Edge[](v15)(v16)
\Edge[](v1)(v17)
\Edge[](v12)(v17)
\Edge[](v17)(v10)
\end{tikzpicture}}

\newcommand{\ExampleDiagramI}{\mathord{%
\begin{tikzpicture}[baseline=-0.65ex,scale=1]
\node[circle,draw=white, fill=white, inner sep=0pt] (AA) at (0.5em,0) {$a$};
\node[circle,draw=white, fill=white, inner sep=0pt] (BB) at (10.5em,0) {$b$};
\node[inner sep=0pt] (A) at (1em,0) {};
\node[inner sep=0pt] (B) at (4em,0) {};
\node[inner sep=0pt] (C) at (4em,0) {};
\node[inner sep=0pt] (D) at (7em,0) {};
\node[inner sep=0pt] (E) at (7em,0) {};
\node[inner sep=0pt] (F) at (10em,0) {};
\node[inner sep=0pt] (G) at (11em,0) {};
\node[inner sep=0pt] (H) at (12em,0) {};

\node[circle,draw=white, fill=white, inner sep=0pt] (AA1) at (3.5em,-1em) {$a$};
\node[circle,draw=white, fill=white, inner sep=0pt] (BB1) at (7.5em,-1em) {$b$};
\node[inner sep=0pt] (A1) at (0em,-1em) {};
\node[inner sep=0pt] (B1) at (3em,-1em) {};
\node[inner sep=0pt] (C1) at (4em,-1em) {};
\node[inner sep=0pt] (D1) at (7em,-1em) {};
\node[inner sep=0pt] (E1) at (8em,-1em) {};
\node[inner sep=0pt] (F1) at (11em,-1em) {};
\node[inner sep=0pt] (G1) at (12em,-1em) {};

\draw (A) to node [above] {} (B);
\draw (C) to node [above] {} (D);
\draw (E) to node [above] {} (F);
\draw (G) to node [above] {} (H);

\draw (A1) to node [below] {} (B1);
\draw (C1) to node [below] {} (D1);
\draw (E1) to node [below] {} (F1);
\draw (F1) to node [above] {} (G1);

\draw (A)--(A1);
\draw (B)--(B1);
\draw (C)--(C1);
\draw (D)--(D1);
\draw (E)--(E1);
\draw (F)--(F1);
\draw (G)--(F1);
\draw (H)--(G1);
\end{tikzpicture}}}

\begin{abstract}
We introduce the Insertion Chain Complex, a higher-dimensional extension of insertion graphs, as a new framework for analyzing finite sets of words. We study its topological and combinatorial properties, in particular its homology groups, as measures of the complexity for word sets. After establishing its theoretical foundations, we investigate the computational and combinatorial aspects of these complexes. Among other results, we classify minimal 1-dimensional cycles and prove that every finitely generated abelian group can be realized as the homology of the insertion complex for some set of words. We also identify conditions that guarantee vanishing homology. These results provide new invariants for characterizing finite sets of words through word-based topological structures and their properties. 
\end{abstract}
    \maketitle



{\small
\tableofcontents
}

\section{Introduction}

    Multiple approaches have been developed to measure the complexity of a given sequence of symbols (e.g. \cite{ilie2004word,ivanyi1987,KOLMOGOROV}), with Kolmogorov complexity being the most notable. Other techniques have been proposed to explore the complexity and evolution of words across different languages \cite{covington1996algorithm,Mufwene2013}. In these approaches, the `randomness' of the symbols within individual words is of interest. More recently, methods based on natural language processing (NLP) and large language models (LLMs) have also been applied to study sets of words, though these focus on semantic meaning and rely on large corpora of text for training \cite{mahajan2025revisiting,reimers2019sentence,devlin2019bert}. In contrast to these approaches, there are numerous situations in which understanding the way words `interact' within a collection of words is necessary, i.e., the complexity of sets of words or the randomness of words within a collection, is of interest. These kind of questions often arise 
    when dealing with code words or genetic sequences. For instance, approaches to conduct comparative analyses of sets of DNA sequences across diverse organisms \cite{YU2011,webb2001} are currently relevant.
  
    A specific instance where sets of words need to be analyzed and compared arises in sequencing data (words over the alphabet ${A, C, T, G}$) obtained from experiments on DNA double strand break repair. During the repair process, which occurs simultaneously in millions of copies of the DNA sequence, many nucleotide insertions and deletions appear at the breaking point, producing a collection of DNA strings~\cite{Storici2022,DSB-repair-review2021}.  Different breaking points, or different repair mechanisms, generate distinct sets of DNA strings, so understanding the structure of each set of words, as well as differences across them, is of great interest. These collections of strings consist of words that differ from the reference (originally uncut) DNA by sequences of insertions and deletions of nucleotides (single symbols). 
    In~\cite{Storici2022}, these sets of words are represented via the newly introduced \textit{variation-distance graphs} (see also \cite{channagiri2024}), which are closely related to the graphs studied in \cite{Eckler1987} and to the Levenshtein graphs of \cite{RUTH2023}. 
    All of these graphs, in which two words are connected by an edge whenever they differ by a single nucleotide insertion or deletion, provide a one-dimensional representation of the set of words. 
    Although insertions and deletions at the break points were expected to be inherently random, the resulting graphs representing the experimental data exhibit visually distinguishable patterns in the different experimental designs. This suggests that insertion/deletion of symbols at the repair sites may not be completely random, and motivates a closer examination of the topological properties of these representations, as a useful way to better understand the differences among the collections of DNA words. 
   
	In this work, we extend these graph ideas to higher dimensions by introducing the \textit{Insertion Chain Complex} of a set of words. Topological properties of the complex, such as the homology groups, provide a measure of the {\it complexity of the set of words}. In our construction, homology is expected to vanish when the set of words is `complete', implying randomness within the set. Conversely, when many words that would naturally appear with random insertion of symbols are absent, the insertion complex contains `holes', producing non-zero homology in some dimension. Non-zero homology in higher dimensions indicates the absence of larger collections of words within the set. Throughout this paper, we restrict our attention to finite sets of words. Our constructions are particularly useful when single symbol insertions or deletions are the primary source of variation within the set of words. 

    We begin by reviewing common notation in combinatorics of words, and define the insertion graph. In Section~\ref{sec:blocks} we introduce blocks and explore their properties, including a classification of block isomorphism classes up to dimension 4. In Section~\ref{sec:insertion_chain_complex}, we introduce the Insertion Block Complex, the main focus of this paper, and show a few properties.
    Section~\ref{sec:homology} contains our primary results on homology of Insertion Chain Complexes. In Theorem~\ref{thm:classification_cycles_1dim}, we classify the shortest 1-dimensional cycles while Theorem~\ref{thm:realizable_homology} demonstrates that all finitely generated abelian groups can be realized as the homology of the Insertion Chain Complex for some set of words. Theorem~\ref{thm:min_cycles_d} examines the size of minimal cycles in each dimension, and Theorem~\ref{thm:null_homology_one_dim} 
    provides conditions under which sets of words define complexes for which all homology vanishes. We close with short concluding remarks posing open questions and possible lines for further study of complexes arising in other insertion/deletion contexts.  
    In order to maintain a smoother flow in the main body of the paper, some of the more tedious and lengthy proofs are placed in an  Appendix after the main text.


	\section{Preliminaries}\label{sec:prelims}
	\subsection{Words Notation and Equations}
	
	We introduce some definitions and notation from word combinatorics that we use throughout the rest of this paper. For a more comprehensive overview, we refer the reader to~\cite{harrison1978,lothaire1997}.
    
For a natural number $n\in\N$, we denote $[n]=\{1,2,\dots, n\}$. For a finite alphabet of symbols $\Sigma$, its corresponding set of words $\Sigma^*$ is the collection of all finite concatenations (words) of symbols in $\Sigma$, including the \textit{empty word}, which we denote 1. 
 We use letters $a, b, c \ldots $ for elements of $\Sigma$, and $x,y,z$ for words in $\Sigma^*$. 
	
Let $x\in\Sigma^*\setminus\{1\}$ be a word, so     $x=a_1a_2\cdots a_n$ for $a_i\in\Sigma$.  
The {\it length} of $x$ is  $|x|=n$, and we set $|1|=0$. If $i\in [n]$, we denote 
    $x[i]=a_i$,  its $i$-th symbol. We define 
    $\ell(x)=a_n$ to be the last symbol of $x$. 
    Similarly, for a set of indices $I=\{i_1<i_2<\dots< i_r\}\subset[n]$, we denote by $x[I]=a_{i_1}a_{i_2}\dots a_{i_r}$. In the special case when $I$ is an interval $I=[i,j]\cap\N$, we denote $x[i,j]=x[I]=a_ia_{i+1}\dots a_{j-1}a_j$ and we say $x[i,j]$ is a {\it factor} of $x$.  
    Denote by $x^R$ the {\it reverse} word, that is $x^R=a_na_{n-1}\cdots a_2a_1$. 
	
The set $\Sigma^*$ equipped with concatenation forms a monoid with the empty word as the identity element. 
For any symbol $a \in \Sigma$ and integer $n \in \N$, we 
use 
notation $a^n = a \cdots a$.
In particular, $a^0=1$.
As a result, any word $x \in \Sigma^*$ can be uniquely expressed as $x=a_1^{r_1}a_2^{r_2}\cdots a_k^{r_k}$ for symbols $a_i\in\Sigma$ and integers $r_i\in\N$, such that $r_i>0$ and $a_i\neq a_{i+1}$.
	
	Given two words $x,y\in\Sigma^*$, where $x=a_1\cdots a_n$,
    we say that $y$ is a \textit{subword} of $x$, denoted $y\leq x$, if there exists a set of indices $I\subset[n]$ such that $y=x[I]$. It follows that 
    if $x=a_1^{r_1}\cdots a_k^{r_k}$ and $y\leq x$, then there is 
    a sequence of powers $s_1, \dots, s_r$ with $0\leq s_i\leq r_i$, such that $y=a_1^{s_1}\cdots a_k^{s_k}$, although this sequence may not be unique. 
Moreover,  
    the relation $\leq$ defines a partial order on $\Sigma^*$, 
    and the length function $|\cdot|$, is 
    monotonic with respect to this partial order: if $x\leq y$ then $|x|\leq|y|$. 
	
	The following lemma, 
    proposed by Lyndon and Sch\H{u}tzenberger in~\cite{schutzenberger1962} (see 
    also Lemma 1 in~\cite{petersen1996}), along with its corollary, establishes 
    common word equations we use in the rest of the paper.
	
	\begin{lemma}\label{lemma:commutative_words}
		Given two words $v, w \in\Sigma^*$, they commute $vw=wv$ if and only if, they are both powers of the same word $z\in\Sigma^*$, that is, if there are non-negative integers $s,t\in\N$ such that $v=z^s$ and $w=z^t$. 
	\end{lemma}
	
	\begin{coro}\label{cor:simple_word_equations}
		Given symbols $a,b\in\Sigma$, 
        and $x\in\Sigma^*$
        the following hold:
		\begin{enumerate}
			\item $ax=xa$, iff $x=a^t$, for some natural $t\geq 0$.
			\item If $a\neq b$, $abx=xab$ iff $x=(ab)^t$, for some natural $t\geq 0$.
			\item If $a\neq b$, $abx=xba$ iff $x=(ab)^ta$, for some natural $t\geq 0$.
            \item $ax=xb$, iff $a=b$ and $x=a^t$ for some natural $t\geq 0$.
		\end{enumerate}
	\end{coro}
	
	\begin{proof}
		These results follow straightforwardly from Lemma \ref{lemma:commutative_words},  we prove part 3 as an illustrative example. Note that we must have $|x|\geq 1$, and since the RHS ends in the symbol $a$, we must have $x=x'a$, for some $x'\in\Sigma^*$. Then, we get the equation  $abx'a=x'aba$, which implies $abx'=x'ab$, so $x'$ and $y=ab$ commute. By Lemma~\ref{lemma:commutative_words} 
        they are powers of a common word. Since $ab$  cannot be a power of any other word, we must have $x'=(ab)^t$ for some non-negative integer $t$. Then $x=x'a=(ab)^ta$,  as desired. 
	\end{proof}

	\subsection{The Insertion Graph}\label{sec:insertion-graph}
	One way 
    to study the complexity of a set of words is by assessing how similar
    the words are among 
    each other. This similarity can be gauged, among other methods, by computing the \textit{edit} or \textit{Levenshtein distance} between them, that is, the minimum number of insertions, deletions, and subsitutions to get from one word to another. 
	We focus our attention on insertions and deletions only. We  say that two words $w$ and $w'$ are \textit{connected},  $w\sim w'$, if they differ by a single symbol insertion, i.e., 
    if there exist words $ x, y\in \Sigma^*$, and a single symbol $a\in\Sigma$, such that $$w= x y\text{ and } w'= x a y.$$ 
	
	Given this binary relation, 
    we 
    consider the corresponding directed graph, which 
    can be understood as a 1-dimensional topological object: 
	
	\begin{defi}[Insertion Graph]\label{def:insertion_graph}
	Let 
    $W\subset \Sigma^*$ be a set of 
    words. 
    graph consisting of:
		\begin{itemize}
			\item Vertices: The elements of $W$. 
			\item Edges: $(w_1,w_2)\in E(\Gins{W})$ if and only if $w_1\sim w_2$ and $w_2$ is obtained by a single insertion from $w_1$.
		\end{itemize}
	\end{defi}
	
	Different variations of this graph have been extensively studied, see for instance \cite{channagiri2024,Eckler1987,RUTH2023}. 
 The following lemma establishes a couple of 
 useful facts about these graphs.   
	
	\begin{figure}
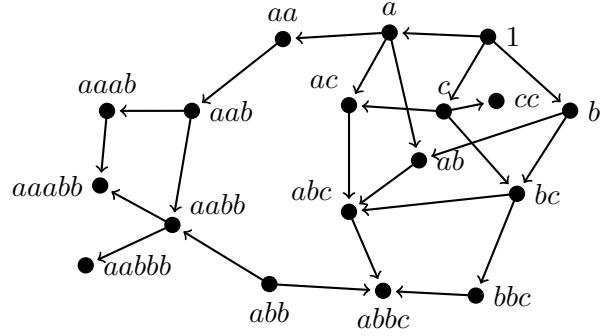

		\centering
		\ExampleGIns
		\caption{$\Gins{W}$ for a set of words in the alphabet $\Sigma=\{a,b,c\}$.
        Directions of edges not depicted.}		\label{fig:examplegraph}
	\end{figure}

	\begin{lemma}\label{lemma:g_bipartite}
		For any set of words $W$, $\Gins{W}$ is bipartite, and 
        acyclic. 
	\end{lemma}
	
	\begin{proof}
		Note that if $(w_1,w_2)$ is a directed edge of $\Gins{W}$, we  have $|w_2|=|w_1|+1$. Then, the vertices can partition $V(\Gins{W})=A \sqcup B$, where $A=\{w\in V: |w|\text{ is even}\}$, and $B=\{w\in V: |w|\text{ is odd}\}$.
        For the acyclic claim, suppose by way of contradiction $w_0w_1\dots w_n=w_0$ is a directed cycle. Then $|w_0|=|w_n|=|w_{n-1}|+1=|w_{n-2}|+2=\dots=|w_0|+n$, which is absurd. Thus, 
        $\Gins{W}$ is acyclic. 
	\end{proof}

   Figure~\ref{fig:examplegraph} illustrates the insertion graph for a set of words $W$. We can distinguish that certain subsets of words exhibit a more organized internal structure, and the overall graph naturally suggests the 1-dimensional skeleton of a higher-dimensional object. To motivate the introduction of blocks in the next section, we explain how one can attach 2-dimensional cells to $\Gins{W}$ by \textit{filling in} squares.
    
    Suppose $w_0,w_1,\dots, w_n=w_0$ 
    are such that $w_i\neq w_j$ for all $0\leq i<j<n$ and $w_i\sim w_{i+1}$. 
    Assume that 
    $u=w_0,w_1,\dots, w_i=v$ forms a path 
    from $u$ to $v$, and $u=w_n,w_{n-1},\dots w_{i+1},w_i=v$ is a different path from $u$ to $v$. We are interested in characterizing how similar or `compatible' these two paths are.
	In particular, consider the simplest case of 
    length 4 (the shortest possible) of the form $u,w_1,v,w_2$, with $|u|<|w_1|=|w_2|<|v|$. This gives two paths from $u$ to $v$: $u,w_1,v$ and $u,w_2,v$. Intuitively, we consider these paths `compatible' if the corresponding symbol insertions are independent and commutative. We aim to introduce a topological object where the existence of nontrivial homology will help us identify incompatible paths. Thus, if both paths are compatible, we expect a `filled-in' square; otherwise, it should remain empty. 
    Figure~\ref{fig:emptyfilledsquares} illustrates two cycles of length 4 with $u=ab$ and $v=abab$. In the first cycle, the insertions of the letters $a$ and $b$ occur at the same relative position, independently of which one occurs first, so the square is filled in. This is not the case in the second cycle, since the $a$-insertion in the edge $ab\sim aba$ occurs at the end of the string, while it happens at the start of the string in the edge $bab\sim abab$,  once the $b$ is present, so the cycle remains unfilled.  We define precisely 2-dimensional blocks in the next section. 
	
	\begin{figure}
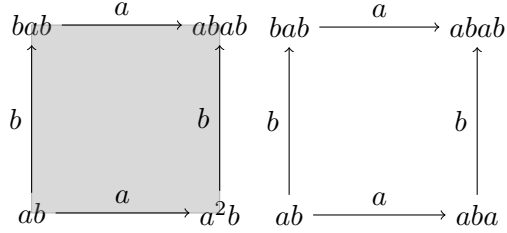

		\centering
		\FigureFilledSquare
		\caption{A filled-in square and a square that doesn't need to be filled in}
		\label{fig:emptyfilledsquares}
	\end{figure}
	
	The following examples illustrate some of the properties that a generalization of the insertion graph should satisfy. We point out that even though in some cases \textit{cubical complexes} arise (see \cite{hatcher2002algebraic, mischaikow2004computational, savvidou2010face}), they are not enough for our purpose. 
	
	\begin{examples}\leavevmode
		\begin{itemize}
			\item  If we consider $W=\{1,a,b,c,ab,ac,bc,abc\}$, $\Gins{W}$ is isomorphic to the 1-skeleton of a cube, in which each 2-face is a 4-cycle with two compatible paths. Moreover, the three insertions $a$, $b$, and $c$, to get from $1$ to $abc$ are also compatible, so the whole cube should be regarded as a solid. In this case, a cubical complex of $[0,1]^3$ has the desired structure.
			
            \item Consider the set of words $W=\{1,a,b,ab,ba\}$. Let $K$ be a topological structure that generalizes $\Gins{W}$, and `fills in' compatible blocks as in the previous example. Note that the cycles $\gamma_1:1\sim a\sim ab\sim b$ and $\gamma_2: 1\sim a\sim ba\sim b$ have two compatible insertions, so they must be filled, resulting in two corresponding squares $Q_1, Q_2\in K$. Hence $K=Q_1\cup Q_2$ satisfies our requirements. Note, however, that the intersection $Q_1\cap Q_2$ consists of the two edges $(1,a)\cup (1,b)$, meaning $K$ can't be a cubical complex, since non-empty intersections of squares in a cubical complex must always result in a single face (edge or vertex). 
		\end{itemize}
	\end{examples}
	
	\begin{figure}
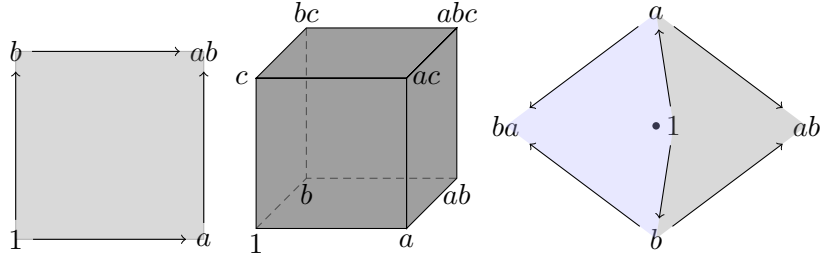

		\centering
		\CubicalExamples
		\caption{Spaces associated to different sets of words.}
		\label{fig:cubical_examples}
	\end{figure}
	
	\section{Blocks}\label{sec:blocks}
    \subsection{Definitions}
	In this section, we formalize the idea of \textit{compatible paths} and introduce blocks, as a topological  structure to understand them. While we illustrate the main definitions with examples, the first part of this section focuses primarily on establishing precise and well-defined notation for working with blocks.
    
	\begin{defi}\label{def:block}
		Given $\Sigma=\{a,b,c, \dots\}$ a finite set of symbols. Let $E=\left\{(1,a):a\in\Sigma\right\}$. Then, for each $m=0,1,\dots$ we define
		\begin{align*}
			\E_m&:= \Sigma^*\times E\times\Sigma^*\times\dots\times E\times\Sigma^*
		\end{align*}
		where there are $m$ factors of $E$ in $\E_m$. We refer to the elements of $\E_m$ as \textbf{$m$ dimensional blocks} (or just \textbf{$m$-blocks}). For any $\sigma\in \E_m$, we denote its dimension as $\dim(\sigma)=m$. 
	\end{defi}
	
    If $\sigma\in \E_m$ is an $m$-block, we will write it as 
	$$x_0(1,a_1)x_1(1,a_2)\cdots x_{m-1}(1,a_m)x_m,$$ where each $a_i\in \Sigma$, and $x_i\in\Sigma^*$.
	
	\begin{note2}\leavevmode
		\begin{itemize}
			\item Note $\E_0=\Sigma^*$, so 0-blocks are simply \textbf{words} in $\Sigma$. We may also refer to them as \textbf{vertices}, as an analogy to the insertion graph $\Gins{\Sigma^*}$.  
            
			\item $1$-blocks can be considered 
            as edges in the insertion graph $\Gins{\Sigma^*}$. Indeed, if $(w_1,w_2)$ is an edge in $\Gins{\Sigma^*}$, $w_1=x_0x_1$ and $w_2=x_0ax_1$, for some $x_0,x_1\in\Sigma^*$ and some $a\in\Sigma$. Then, by `factoring' common words out we get $(w_1,w_2)=(x_0x_1,x_0ax_1)``="x_0(1,a_0)x_1$ which is an element of $\E_1$. 
		\end{itemize}
	\end{note2}
	
Note that the representation 
of edges in $\Gins{\Sigma^*}$ as 1-blocks is not unique. For example, consider the edge $(ab, aab)$, which corresponds to both $(1, a)ab$ and $a(1, a)b$. These represent the insertion of the symbol $a$ in two different positions. However, since the words that 
are source and  target of the edge 
are identical in both representations, we 
regard these $1$-blocks 
as equivalent. 
Hence we 
work with equivalence classes of blocks.
	
	\begin{defi}\label{def:equivalent_blocks}\leavevmode
		\begin{enumerate}
			\item Given $\sigma_1,\sigma_2\in \E_0$, we say they are \textbf{equivalent}, $\sigma_1\equiv \sigma_2$, if $\sigma_1=\sigma_2$ as words in $\E_0=\Sigma^*$. 
			\item For $m\geq 1$, given two $m$-blocks $\sigma_1,\sigma_2\in\E_m$, we may write $\sigma_i=\hat{\sigma_i}(1,a_i)x_i$ for $i=1,2$ and  $\hat{\sigma_i}\in\E_{m-1}$, $a_i\in\Sigma$ and $x_i\in\Sigma^*$. We say $\sigma_1$ is \textbf{equivalent} to $\sigma_2$, and write $\sigma_1\equiv \sigma_2$, if 
            $\hat{\sigma_1}a_1x_1\equiv \hat{\sigma_2}a_2x_2$ and $\hat{\sigma_1}x_1\equiv\hat{\sigma_2}x_2$ as $(m-1)-$blocks. 
		\end{enumerate}
	\end{defi}

 	Suppose $(w_1,w_2)$ is an edge in $\Gins{\Sigma^*}$ that can be associated to more than one 1-block. That is, we can write $w_1=x_1x_2$ and $w_2=x_1ax_2$, but also $w_1=y_1y_2$ and $w_2=y_1by_2$. Thus $(w_1, w_2)$ corresponds to both 1-blocks: $x_1(1,a)x_2$ and $y_1(1,b)y_2$. However,  since $x_1x_2=w_1=y_1y_2$ and $x_1ax_2=w_2=y_1by_2$, both blocks are equivalent, according to Definition \ref{def:equivalent_blocks}. Therefore, to each edge we can associate exactly one equivalence class of 1-blocks. Hence, to each equivalence class of 1-blocks, we can also associate the vertices of the edge, that is: %
$V(x_0(1,a_1)x_1)=\{x_0x_1,x_0ax_1\}$. We extend such 
assignment of vertices to blocks of higher 
dimensions.
	
	\begin{defi}\label{def:vertices}
		Let $\sigma\in \E_m$ be an $m-$block given by $\sigma=x_0(1,a_1)x_1\cdots(1,a_m)x_m$.
		\begin{itemize}
			\item For a subset of indices $I\subset [m]$, we define the \textbf{vertex} $v_I(\sigma)$ of $\sigma$ as 
			$$v_I(\sigma)=x_0\xi_1 x_1\xi_2\cdots \xi_m x_m,  \text{ where } \xi_i=\begin{cases}
				a_i\text{ if } i\in I\\
				1\text{ if } i\not\in I
			\end{cases}$$
			\item We define the  \textbf{set of vertices of $\sigma$}, as 
			$$V(\sigma)=\left\{ v_I(\sigma): I\subset [m]\right\}.$$
		\end{itemize}
	\end{defi}
	
	\begin{examples}\leavevmode
		\begin{itemize}
			\item For a 0-block $w\in\E_0=\Sigma^*$, $[0]=\emptyset$, and $v_\emptyset(w)=w$, so $V(w)=\{w\}$.
			\item For a 1-block $\sigma=x_0(1,a)x_1$, $$V(\sigma)=\left\{v_\emptyset(\sigma), v_{\{1\}}(\sigma)\right\}=\left\{x_0x_1,x_0ax_1\right\}.$$ 
			\item For a 2-block $\sigma=x_0(1,a_1)x_1(1,a_2)x_2$, 
			$$V(\sigma)=\left\{ v_\emptyset(\sigma), v_{\{1\}}(\sigma),v_{\{2\}}(\sigma), v_{\{1,2\}}(\sigma)\right\}=\left\{x_0x_1x_2,x_0a_1x_1x_2,x_0x_1a_2x_2,x_0a_1x_1a_2x_2\right\}.$$
		\end{itemize}
	\end{examples}

Observe that the inclusion of subsets of $[m]$ induces an ordering on the vertices of $\sigma$. Specifically, if $I_1 \subset I_2 \subset [m]$, then $v_{I_1}(\sigma) \leq v_{I_2}(\sigma)$. In particular, the word $w_m = v_\emptyset(\sigma) = x_0 x_1 \cdots x_m$ is the minimum, and $w_M = v_{[m]}(\sigma) = x_0 a_1 x_1 \cdots a_m x_m$, the maximum  of $V(\sigma)$. Therefore, $w_m \leq v \leq w_M$ for any $v \in V(\sigma)$.

Each equivalence class of 1-blocks corresponds to exactly one edge in $\Gins{\Sigma^*}$, as equivalent 1-blocks have the same vertices. This holds for higher dimensional blocks as well.
 \begin{lemma}\label{lemma:vertices_equivalence_class}
		If $\sigma_1\equiv \sigma_2$ are two equivalent $m$-blocks and $I\subset [m]$ then $v_I(\sigma_1)=v_I(\sigma_2)$. In particular, $V(\sigma_1)=V(\sigma_2)$. 
\end{lemma}


\begin{remark}
An $m$-block $\sigma = x_0(1, a_1) \cdots (1, a_m)x_m$ can then be understood as the structure that encapsulates the set of words (or vertices) $V(\sigma)$ generated by transitioning from its minimum word $w_m = v_\emptyset(\sigma) = x_0 x_1 \cdots x_{m-1} x_m$ to its maximum word $w_M = v_{[m]}(\sigma) = x_0 a_1 x_1 \cdots x_{m-1} a_m x_m$, by incorporating the symbols $a_1, \cdots, a_m$ independently, producing $m$ degrees of freedom. That is, $V(\sigma)$ includes all vertices $v$ satisfying $w_m \leq v \leq w_M$ that are obtained through all possible insertions of the symbols $a_i$ at their designated positions. In this sense, we have a collection of ``compatible paths'' from $w_m$ to $w_M$, as mentioned in the previous section, and an $m$-block is the higher dimensional object \textit{filling in} the space between the paths.

By Lemma \ref{lemma:vertices_equivalence_class}, equivalent blocks represent the same set of vertices, and Theorem \ref{thm:unique_block_by_vertices} will establish the converse for valid blocks.
\end{remark}

Instead of proving this lemma directly, we present a generalization to higher-dimensional sub-blocks in Lemma \ref{lemma:equivalence_subblocks}. The concept of vertices, which can be seen as 0-dimensional sub-blocks, can be extended to define sub-blocks of any dimension. 

\begin{defi}\label{def:sub_blocks}
    Let $\sigma\in \E_m$ be an $m-$block given by $\sigma=x_0(1,a_1)x_1\cdots(1,a_m)x_m$. For $I^+, I^-\subset [m]$, disjoint subsets of indices, we define the $k$-block $\sigma(I^+,I^-)$, where $k=m-|I^+|-|I^-|$, as $$\sigma(I^+,I^-)=x_0\lambda_1 x_1\lambda_2\cdots x_{m-1}\lambda_m x_m,$$ where the $\lambda_i$'s are 0- or 1-blocks given by:
	$$\lambda_i=\begin{cases}
		a_i & \text{ if } i\in I^+\\
		1 & \text{ if } i \in I^-\\
		(1,a_i) & \text{ if } i\not\in I^+\cup I^-
	\end{cases}.$$
	We refer to the index sets as \textbf{upper} indices for $I^+$ and \textbf{lower} indices for $I^-$.
\end{defi}

Note that this indeed generalizes definition \ref{def:vertices} of vertices, since $v_I(\sigma)=\sigma(I,[m]\setminus I)$, for any $I\subset [m]$.

\begin{examples}\label{ex:sub-blocks}\leavevmode
\begin{itemize}
    \item Let $\sigma=abc(1,a)(1,b)x(1,c)(1,d)$. A few sub-blocks of $\sigma$ are $\sigma(\{2\},\emptyset)=abc(1,a)bx(1,c)(1,d)$, $\sigma(\{1,2\},\{4\})=abcabx(1,c)$ and $\sigma(\emptyset, \{2,3\})=abc(1,a)x(1,d)$. 
    \item Let $\sigma = (1,a)(1,b)(1,a)$ be a $3$-block. There are 27 possible disjoint subsets $I^+, I^- \subset [3]$, each producing a sub-block, as shown in Table \ref{tale:ex_subblocks} in the appendix. Notice that all 0-dimensional sub-blocks correspond to the vertices of $\sigma$, and we obtain sub-blocks of all dimensions from 0 up to 3, with $\sigma$ itself being a sub-block. Additionally, some sub-blocks can be generated by multiple index sets: for instance, $\sigma(\emptyset, \{1,2\}) = (1,a) = \sigma(\emptyset, \{2,3\})$. Moreover, different index sets may produce equivalent sub-blocks: for example, $\sigma(\{1\}, \{2\}) = a(1,a) \equiv (1,a)a = \sigma(\{3\}, \{2\})$.
\end{itemize}

\end{examples}

 Generalizing Lemma \ref{lemma:vertices_equivalence_class}, sub-blocks of equivalent blocks are themselves equivalent. 
\begin{lemma}\label{lemma:equivalence_subblocks}
	Let $\sigma_1\equiv\sigma_2\in\E_m$ be two equivalent $m$-blocks. Then, for any $I^+, I^-\subset [m]$ with $I^+\cap I^-=\emptyset$ we have $\sigma_1(I^+,I^-)\equiv\sigma_2(I^+,I^-)$. 
\end{lemma}

\begin{proof}
	We proceed by induction on $m$. Note that when $m=0$ we have $\sigma_1=x=\sigma_2\in\Sigma^*$ and the only possible sets of indices are $I^+=I^-=\emptyset$ for which we have $\sigma_1(I^+,I^-)=x=\sigma_2(I^+,I^-)$. 
	
	Suppose now $m\geq 1$ and that the lemma holds for blocks of any dimension $k<m$. 
	
	Let $\hat{\sigma_1},\hat{\sigma_2}\in\E_{m-1}$ be $(m-1)$-blocks,  $x_1,x_2\in\Sigma^*$ be words, and $a_1,a_2\in\Sigma$ be symbols, such that $\sigma_i=\hat{\sigma_i}(1,a_i)x_i$, for $i=1,2$. Since $\sigma_1\equiv\sigma_2$, we have $\hat{\sigma}_1x_1\equiv\hat{\sigma}_2x_2$ and $\hat{\sigma}_1a_1x_1\equiv\hat{\sigma}_2a_2x_2$. All of these blocks are of dimension $m-1$, so by the induction hypothesis, for any disjoint sets $J^+,J^-\subset[m-1]$, we will have 
\begin{align}
	\hat{\sigma}_1(J^+,J^-)x_1=(\hat{\sigma}_1x_1)(J^+,J^-) &\equiv (\hat{\sigma}_2x_2)(J^+,J^-)=\hat{\sigma}_2(J^+,J^-)x_2 \text{, and}\label{eq:equiv_1}\\
	\hat{\sigma}_1(J^+,J^-)x_1a_1=\left(\hat{\sigma}_1a_1x_1\right)(J^+,J^-) &\equiv\left(\hat{\sigma}_2a_2x_2\right)(J^+,J^-)=\hat{\sigma}_2(J^+,J^-)a_2x_2 \label{eq:equiv_2}
\end{align}

Let now $I^+,I^-\subset [m]$ be disjoint sets. If $m\not\in I^+\cup I^-$, then $I^+, I^-\subset[m-1]$ and $$\sigma_i(I^+,I^-)=\left(\hat{\sigma_i}(1,a_i)\right)(I^+,I^-)=\hat{\sigma_i}(I^+,I^-)(1,a_i)x_i$$
By the equivalences (\ref{eq:equiv_1}) and (\ref{eq:equiv_2}), $\sigma_1(I^+,I^-)\equiv \sigma_2 (I^+,I^-)$. 

If instead $m\in I^+\cup I^-$, let $J^\pm=I^\pm\cap[m-1]$, and note that $\sigma_i(I^+,I^-)=\hat{\sigma_i}(J^+,J^-)\lambda_ix_i$ where $$\lambda_i=\begin{cases}
	a_i &\text{ if }m\in I^+\\
	1 &\text{ if }m\in I^-
\end{cases}.$$ 
Here again we get, $\hat{\sigma_1}(J^+,J^-)\lambda_1x_1 \equiv \hat{\sigma_2}(J^+,J^-)\lambda_2x_2$, because of equivalence (\ref{eq:equiv_1}) if $m\in I^-$ and because of equivalence (\ref{eq:equiv_2}) if $m\in I^+$. This completes the proof. 

\end{proof}

	The following propositions allow for 
    better understanding of the relationship between equivalent blocks and 
    ease 
    the comparison of 
    blocks. 

 \begin{prop}\label{prop:comparing_eq_blocks}
 Let $\lambda, \tau, \sigma_1, \sigma_2$ be blocks such that $\dim(\sigma_1)=\dim(\sigma_2)=m$. Then, 
 $$\lambda\sigma_1\tau\equiv \lambda\sigma_2\tau \text{ if and only if }\sigma_1\equiv\sigma_2.$$
 \end{prop}

Thus, in general, we can determine whether two blocks are equivalent by focusing only on the parts of their expressions that differ.

 \begin{proof}
    We first note that Definition \ref{def:equivalent_blocks} can also be stated by reducing the first edge $(1,a_i)$ instead of the last, that is if $\dim(\sigma_1)=\dim(\sigma_2)=m$, we get the $(m+1)$-block equivalence $x_1(1,a_1)\sigma_1\equiv x_2(1,a_2)\sigma_2$ if and only if $x_1\sigma_1\equiv x_2\sigma_2$ and $x_1a_1\sigma_1\equiv x_2a_2\sigma_2$ as $m$-blocks. Indeed, this is trivial when $m=0$. Then, proceeding by induction, write $\sigma_i=\hat{\sigma}_i(1,b_i)y_i$ and observe that applying both the definition and the induction hypothesis, we get that $x_1(1,a_1)\sigma_1\equiv x_2(1,a_2)\sigma_2$ if and only if we have the four $(m-1)$-block equivalences:
    \begin{align*}
        x_1\hat{\sigma}_1y_1\equiv x_2\hat{\sigma}_2y_2, \quad& x_1a_1\hat{\sigma}_1y_1\equiv x_2a_2\hat{\sigma}_2y_2, \\
        x_1\hat{\sigma}_1b_1y_1\equiv x_2\hat{\sigma}_2y_2b_2, \quad&\text{and}\quad x_1a_1\hat{\sigma}_1b_1y_1\equiv x_2a_2\hat{\sigma}_2b_2y_2.
    \end{align*}
    By the definition, these hold if and only if $x_1\sigma_1\equiv x_2\sigma_2$ and 
    $x_1a_1\sigma_1\equiv x_2a_2\sigma_2$, as desired. 

    For $\lambda=x, \tau=y\in\Sigma^*$, the proposition follows straightforward from an inductive argument on $\dim(\sigma_i)$. However, we omit the details here.

    Then, if $\lambda=x_0(1,a_1)\cdots (1,a_k)x_k$ and $\tau=y_0(1,b_1)\cdots (1,b_r)y_r$, applying the definition $k$ times on the first edge, and $r$ times on the last edge, we observe that $\lambda\sigma_1\tau\equiv \lambda\sigma_2\tau$ if and only if $v_I(\lambda)\sigma_1v_J(\tau)\equiv v_I(\lambda)\sigma_2v_J(\tau)$ for all $I\subset [k]$ and all $J\subset [r]$. Then, because all vertices $v_I(\lambda)$ and $v_J(\tau)$ are words in $\Sigma^*$, by the case $m=0$, these equivalences hold if and only if $\sigma_1\equiv\sigma_2$, as desired. 
 \end{proof}
	
	\begin{prop}\label{prop:simple_equation_words}
		Let $a\in\Sigma$ and $x,x',y,y'\in\Sigma^*$ be such that they satisfy $xx'=yy'$ and $xax'=yay'$. If $|y|\geq|x|$, then there is a natural $t\geq 0$ such that $y=xa^t$ and $x'=a^ty'$. 
	\end{prop}
	We may interpret the proposition as follows: Suppose we are given an edge $(w_1,w_2)$ and we know a decomposition $x,x'$, such that $w_1=xx'$ and $w_2=xax'$. Then, the only way we can produce a different decomposition $y, y'$, is if we \textit{move} some copies of $a$ at the start of $x'$ to the end of $x$.  
	
	\begin{proof}
		Since $|y|\geq|x|$ it must be that $|x'|\geq|y'|$.  Since $xx'=yy'$, we must have $y=x\hat{y}$ and $x'=\hat{x}y'$. Thus, simplifying the common affixes, we get the equations $\hat{x}=\hat{y}$ and $\hat{x}a=a\hat{y}$. Thus $\hat{x}a=a\hat{x}$, and by Corollary~\ref{cor:simple_word_equations} there is a natural $t\geq 0$ such that $\hat{x}=a^t$. The result then follows. 
	\end{proof}
	
	Applying this result to 1-blocks, we see that $x_0(1,a)x_1\equiv y_0(1,a)y_1$ if and only if $x_0=y_0a^t$ and $y_1=a^tx_1$,  or $y_0=x_0a^t$ and $x_1=a^ty_1$. So we can only produce equivalent edges by \textit{moving} copies of $a$ around the edge $(1,a)$. This is also the case for $m$-blocks, so we can collect all copies of each $a_i$ in front of the corresponding $(1,a_i)$, one edge at a time, to get a \textbf{canonical block} as the representative for each equivalent class of blocks. We formalize this idea in the following definition and theorem. 
	
	\begin{defi}\label{def:canonical_form}
		An $m-$block $\sigma=x_0(1,a_1)\cdots(1,a_m)x_m$ is said to be in \textbf{canonical form} if for each $i=1,\dots, m$, $x_i[1]\neq a_i$. 
	\end{defi}
	
	\begin{theorem}\label{thm:unique_canonical_form}
		Each equivalence class of $m-$blocks, has a unique representative in canonical form. 
	\end{theorem}
	\begin{proof}
		We start by proving its existence. For $m=0$ there's nothing to prove. Suppose a canonical representative always exist for $(m-1)-$blocks, for some $m\geq 1$. Let $\sigma=x_0(1,a_1)x_1\cdots (1,a_m)x_m$ be an $m-$block. Write $x_m=a_m^t\hat{x}_m$, for some $t\geq 0$ maximal, it is then straightforward to check $\sigma\equiv x_0(1,a_1)x_1\cdots x_{m-1}a_m^t(1,a_m)\hat{x}_m$. Consider then $\Bar{\sigma}= x_0(1,a_1)\cdots (1,a_{m-1})x_{m-1}a_m^t$, so $\sigma\equiv \Bar{\sigma}(1,a_m)\hat{x}_m$. Because $\Bar{\sigma}$ is an $(m-1)$-block, by the induction hypothesis there is a canonical representative $\hat{\sigma}=\hat{x}_0(1,b_1)\cdots (1,b_{m-1})\hat{x}_{m-1}$ with $\Bar{\sigma}\equiv\hat{\sigma}$. Since $v_\emptyset(\Bar{\sigma})=v_\emptyset(\hat{\sigma})$ and for each $i=1, \dots, m-1$ , whe have $v_{\{i\}}(\Bar{\sigma})=v_{\{i\}}(\hat{\sigma})$, we must have $a_i=b_i$. Thus $\sigma\equiv \hat{x}_0(1,a_1)\hat{x}_1\dots (1,a_{m-1})\hat{x}_{m-1}(1,a_m)\hat{x}_m$, where for each $i=1,\dots, m-1$, $\hat{x}_i[1]\neq a_i$ since $\hat{\sigma}$ is in canonical form, and also $\hat{x}_m[1]\neq a_m$, by construction.\\
		
		Suppose now that $\sigma_1\equiv\sigma_2$ are both in canonical form, and write
		\begin{align*}
			\sigma_1&=x_0(1,a_1)x_1(1,a_2)\cdots x_{m-1}(1,a_m)x_m\\
			\sigma_2&=y_0(1,b_1)y_1(1,b_2)\cdots y_{m-1}(1,b_m)y_m.
		\end{align*}
		Since $v_\emptyset(\sigma_1)=v_\emptyset(\sigma_2)$ and for each $i=1, \dots, m-1$, we have $v_{\{i\}}(\sigma_1)=v_{\{i\}}(\sigma_2)$, we must have $a_i=b_i$ (by counting symbols).  Suppose by way of contradiction that there is an index $i$ such that $x_{i}\neq y_{i}$, and fix the maximal such index, so $x_j=y_j$ for all $j>i$. Since we have $v_{[m]\setminus\{i\}}(\sigma_1)=v_{[m]\setminus\{i\}}(\sigma_2)$ and $v_{[m]}(\sigma_1)=v_{[m]}(\sigma_2)$, we get the set of equations $xa_ix'=ya_iy'$ and $xx'=yy'$, where $x=x_0a_1x_1\cdots x_{i-1}$, $x'=x_ia_{i+1}\cdots a_mx_m$ and $y=y_0a_1\cdots y_{i-1}$, $y'=y_ia_{i+1}\cdots a_my_m$. Moreover, since $a_{i+1}x_{i+1}\cdots a_mx_m=a_{i+1}y_{i+1}\cdots a_my_m$, we can delete the common suffixes to get the equations $xa_ix_i=ya_iy_i$ and $xx_i=yy_i$. Assuming $|y_i|>|x_i|\geq 0$, by Proposition \ref{prop:simple_equation_words}, $x_i=a_i^ty_i$ for some $t\geq 1$. But then $x_i[1]=a_i$ which can't be since $\sigma_1$ is in canonical form. Therefore each $x_i=y_i$, and $\sigma_1=\sigma_2$. 
	\end{proof}
	
	\begin{note1}
		To ease 
        notation, in the remaining sections, 
        an $m$-block 
        refers to its equivalence class as defined above. Similarly, instead of saying two blocks are equivalent, we 
        simply say they are the same and write $\sigma_1=\sigma_2$. 
	\end{note1}
	
	\subsection{Valid Blocks}
	
	\begin{example}
		Consider the 2-block $\sigma=(1,a)(1,b)$, for symbols $a,b\in\Sigma$. Then $V=V(\sigma)=\{1,a,b,ab\}$, so the insertion graph $\Gins{V}$ is a 4-cycle, where we can distinguish two different paths from $1$ to $ab$, that are \textit{compatible} in the same sense discussed in section \ref{sec:prelims}. We can see the 2-block as \textit{attaching  a square} to the graph, so we get a 2-dimensional topological space with no holes. 
		
	We would like the same intuition to hold for other blocks, however, consider the 2-block $\sigma=(1,a)(1,a)$. In this case, $V=V(\sigma)=\{1,a,a^2\}$, so $\Gins{V}$ is the union of the two edges $(1,a)$ and $(a,a^2)$. Attaching a square in the same fashion as before would produce a 2-sphere, which has a two-dimensional hole. Even more problematic, computing the boundary of this block, as we will define it in Section~\ref{sec:insertion_chain_complex}, yields zero. Allowing single blocks to have vanishing boundaries would undermine the intuition we are trying to capture. To prevent such situations, we exclude 2-blocks like $\sigma$, and call any block in which a pair $(1,a)(1,a)$ appears \textbf{invalid}.

	\end{example}
	
	\begin{defi}\label{defi:invalid}
		Given a block $\sigma\in\E_m$, with canonical form $\sigma=x_0(1,a_1)x_1\cdots (1,a_m)x_m$, we say $\sigma$ is a \textbf{valid block} if whenever $x_i=1$, for any $i=1,\cdots, m-1$, then $a_i\neq a_{i+1}$. Otherwise, we say $\sigma$ is \textbf{invalid} or \textbf{degenerate}. We denote by $\Ev_m$ the set of all valid $m$-blocks. 
	\end{defi}
	
	\begin{remark}
		If $\sigma$ is a valid block, then $a_i=a_{i+1}$ only if $x_i$ includes a symbol different to $a_i$. Otherwise, $x_i=a_i^r$ for some $r\geq 0$, and writing the canonical form of $\sigma$ we can collect $a_i^r$ in front of $(1,a_i)$, so we are left with only the trivial word between $(1,a_i)$ and $(1,a_{i+1})$, which would be invalid according to Definition \ref{defi:invalid}. 
	\end{remark}

   The following lemma gives some characterizations for invalid blocks based on their vertices and sub-blocks.  Its proof is straightforward; thus, we omit it here but include it in the \hyperlink{proof:charac_valid}{appendix} for completeness.
  
  \begin{lemma}\label{lemma:charac_valid}
  	Let $\sigma\in \E_m$ be an $m$-block. Then, the following are equivalent:
   \begin{enumerate}
   \item $\sigma$ is invalid. 
    \item There exists an $1\leq i<m$, such that $v_{\{i\}}(\sigma)= v_{\{i+1\}}(\sigma)$. 
    \item There exists an $1\leq i<m$, such that $\sigma(\{i\},\emptyset)=\sigma(\{i+1\},\emptyset)$.
    \item There are indices $1\leq j<i\leq m$, such that $v_{[m]\setminus\{i\}}(\sigma)= v_{[m]\setminus\{j\}}(\sigma)$.
   \end{enumerate}
  \end{lemma}

	The reciprocal of Lemma \ref{lemma:vertices_equivalence_class} holds for valid blocks, that is, each valid $m$-block is uniquely determined (up to equivalence) by its set of vertices. 
	
	\begin{theorem}\label{thm:unique_block_by_vertices}
	Let  $\sigma_1\in \Ev_m$ and $\sigma_2\in\Ev_n$ be valid blocks such that $V(\sigma_2)= V(\sigma_1)$, then $\sigma_1=\sigma_2$. 
	\end{theorem}

    Rather than proving this Theorem directly, we prove a slightly more general version stated below:

    \begin{theorem}\label{thm:unique_block_by_vertices2}
    Let  $\sigma_1\in \Ev_m$ and $\sigma_2\in\Ev_n$ be valid blocks such that $V(\sigma_2)\subset V(\sigma_1)$,  $v_\emptyset(\sigma_1)=v_\emptyset(\sigma_2)$, and  $v_{[m]}(\sigma_1)=v_{[n]}(\sigma_2)$,  then $\sigma_1=\sigma_2$.   
	\end{theorem}
	
    	\begin{proof}
		Consider the canonical representation for $\sigma_1$ and $\sigma_2$ as follows
		\begin{align*}
			\sigma_1=& x_0(1,a_1) x_1\dots x_{m-1}(1,a_m) x_m\\
			\sigma_2=& y_0(1,b_1) y_1\dots y_{n-1}(1,b_n) y_n
		\end{align*}
		
		Let $w_m = v_\emptyset(\sigma_1) = v_\emptyset(\sigma_2)$ and $w_M = v_{[m]}(\sigma_1) = v_{[n]}(\sigma_2)$. In $\sigma_1$,  we have $w_m = x_0 x_1 \dots x_m$  and  $w_M = x_0 a_1 x_1 \dots x_{m-1} a_m x_m$, so  $|w_M|-|w_m|=m$.  The same computation  for $\sigma_2$ gives $|w_M|-|w_m|=n$, which implies $n = m$. Thus $\dim(\sigma_1) = \dim(\sigma_2) = m$.
		
		Let $W_i=\{w\in V(\sigma_i): |w|=m-1\}$, for $i=1, 2$. Clearly $W_2\subset W_1$, since  $V(\sigma_2)\subset V(\sigma_1)$. Also note  that if $w\in W_i$, then $w=v_{[m]\setminus\{j\}}(\sigma_i)$, for some $j\in[m]$. Then, because $\sigma_1$ and $\sigma_2$ are valid, by Lemma \ref{lemma:charac_valid}, we must have $v_{[m]\setminus\{j\}}(\sigma_i)\neq v_{[m]\setminus\{k\}}(\sigma_i)$ for all $1\leq j<k\leq m$, and $i=1,2$. Thus $|W_1|=|W_2|=m$, and hence  $W_1=W_2=W$. 
		
		We proceed by induction to prove $ x_{i-1}= y_{i-1}$ and $a_i=b_i$ for all $0\leq i\leq m$. Note that the base case $i=0$ holds vacuously. 
  
		Suppose we have $1\leq r\leq m-1$, such that the claim holds for any $i\leq r$. We will prove $ x_r= y_r$. Suppose by way of contradiction that they are different, and suppose, without loss of generality, that $s=|y_r|<| x_r|$. Consider the  vertex $w_M$, written using both blocks:
		\begin{align*}
			w_M=x_0a_1 x_1\cdots  x_{r-1}a_r &\phantom{|}|\phantom{|}  x_ra_{r+1}\cdots a_m x_m\in V(\sigma_1) \\
			& \shortparallel\\
			w_M=y_0b_1 y_1\cdots  y_{r-1}b_r &\phantom{|}|\phantom{|}  y_rb_{r+1}\cdots b_m y_m\in V(\sigma_2)
		\end{align*}
		Where we use a vertical line as  a visual aid to denote that corresponding $ x$'s and $ y$'s, and, $a$'s and $b$'s, agree up to that point. Note we must have $ x_r[1,s+1]= y_rb_{r+1}$, so $b_{r+1}= x_{r}[s+1]$. 
		
		Now consider the shortest vertex $w_m$, written using both cells:
		\begin{align*}
			x_0 x_1\cdots  x_{r-1} &\phantom{|}|\phantom{|}  x_r x_{r+1}\cdots  x_m\in V(\sigma_1) \\
			& \shortparallel\\
			y_0 y_1\cdots  y_{r-1} &\phantom{|}|\phantom{|}  y_r y_{r+1}\cdots  y_m\in V(\sigma_2)
		\end{align*}
		
		We observe two cases:
		
		\textbf{Case 1:} $ y_{r+1}\neq 1$, then we must have $ y_{r+1}[1]= x_r[s+1]=b_{r+1}$. But since $\sigma_2$ is in canonical form, this can't be. 
		
		\textbf{Case 2:} $ y_{r+1}=1$.  Consider the vertex $w=v_{[m]\setminus\{r+1\}}(\sigma_2)\in W$. Thus, there must be an index $i_0\in[m]$, such that $w=v_{[m]\setminus\{i_0\}}(\sigma_1)$. Suppose by way of contradiction that $i_0\leq r$, then we can align the representation of $w$ from both cells as follows:
		\begin{align*}
			w=  x_0a_1 x_1a_2\dots x_{i_0-1} \phantom{|} & |\phantom{|b_{i_0}}x_{i_0}a_{i_0+1}\cdots a_m x_m\\
			&\shortparallel\\
			w=  y_0b_1 y_1b_2\dots y_{i_0-1} \phantom{|} & | \phantom{|}b_{i_0} y_{i_0}\cdots\cdot\cdot\cdots b_m y_m
		\end{align*}
		Again, we use the vertical bar as a visual aid to show where the $a$'s and $ x$'s are equal to the $b$'s and $ y$'s, respectively. Since the symbols in both representations of $w$ must coincide, we must have $b_{i_0}= x_{i_0}[1]$ if $ x_{i_0}\neq 1$, or $b_{i_0}=a_{i_0+1}$ if $ x_{i_0}=1$. In the former case, since $i_0\leq r$, $a_{i_0}=b_{i_0}= x_{i_0}[1]$, which can't be since $\sigma_1$ is in canonical form. In the later, $a_{i_0}=b_{i_0}=a_{i_0+1}$, and $ x_{i_0}=1$, but this is also a contradiction since $\sigma_1$ is a valid cell. Therefore, $i_0>r$. Which means that all the $a$'s up to $a_r$ must appear in the representation of $w$. Aligning both expressions for $w$, we get
		\begin{align*}
			w= x_0a_1 x_1\cdots  x_{r-1}a_r \phantom{|}&|\phantom{|}  x_r\cdots\cdots a_m x_m \\
			& \shortparallel\\
			w=  y_0b_1 y_1b_2\cdots  y_{r-1}b_r\phantom{|}& | \phantom{|}  y_r\phantom{|}b_{r+2}\cdots b_m y_m.
		\end{align*}
		Thus we observe $ x_r[s+1]=b_{r+2}$, but since $b_{r+1}= x_r[s+1]=b_{r+2}$ and $ y_{r+1}=1$, this contradicts $\sigma_2$ being a valid cell. Therefore, we must have $| x_r|=| y_r|$, and thus $ x_r= y_r$. By then comparing the two ways of writing the longest word $w_M$, since they coincide up to $ x_r= y_r$, we must also have $a_{r+1}=b_{r+1}$, finishing with the induction. 
		
		We have now proven that $ x_i= y_i$ and $a_{i+1}=b_{i+1}$ for all $0\leq i\leq (m-1)$. Therefore, the longest word in both representations coincides up to $a_m$ and $b_m$, from which we can trivially deduce that also $ x_m= y_m$, finishing the proof. 
	\end{proof}

	\subsection{Faces of Valid Blocks}

	\begin{defi}[Faces]\label{def:faces} Given a valid $m$-block $\sigma\in\Ev_m$, and a valid $r$-block $\tau\in\Ev_r$, with $r\leq m$. We say that  $\tau$,  is a \textbf{face} of $\sigma$, denoted by $\tau\leq \sigma$, if there exist disjoint sets $I^+, I^-\subset [m]$, such that $\tau=\sigma(I^+,I^-)$. Furthermore, the \textbf{faces} of $\sigma$, $\calF(\sigma)$ is the set of all such valid blocks, that is:
		$$\calF(\sigma)=\{\tau: \tau\text{ is valid}, \tau=\sigma(I^+,I^-)\text{, } I^+, I^-\subset[m] \text{ and } I^+\cap I^-=\emptyset\}.$$
		If $\tau$ is a face of $\sigma$, and $\dim(\tau)=\dim(\sigma)-1$, we say that $\tau$ is a \textbf{facet} of $\sigma$, denoted by $\tau\prec\sigma$. 
	\end{defi}

     We emphasize that faces are defined exclusively for \textbf{valid} blocks, and, by definition, these faces are also valid blocks. In particular, note that the vertices $ V(\sigma) $ correspond to the 0-dimensional faces of $ \sigma $, which are determined by the subsets $ I \subset [m] $, where each vertex is given by $ v_I(\sigma) = \sigma(I, [m] \setminus I) $. Since all 0-blocks are valid by definition, these vertices are guaranteed to be valid as well. Additionally, note that $\sigma=\sigma(\emptyset,\emptyset)$ is always a face of itself, and it is the only face of dimension $m$. 
	
	 We note that the ``being a face of" relation defines a partial order on valid blocks, justifying the use of the notation $\leq$. While the proof of this is straightforward, we include it in the appendix as Lemma \ref{lemma:faces_ordering} for completeness.

 If $\sigma\in \Ev_m$ is a valid block, its facets are all the valid $(m-1)$-blocks that can be written as $\sigma(I^+,I^-)$. Because its dimension is given by $m-|I^+|-|I^-|=m-1$, all facets are either of the form $\sigma(\{i\},\emptyset)$ or $\sigma(\emptyset,\{i\})$, for some $i\in[m]$. We distinguish between these two types as follows. 

  \begin{defi}\label{def:facets}
     Let $\sigma\in\Ev_m$ be a valid $m$-block. Given $\tau\prec\sigma$ a facet of $\sigma$, we say $\tau$ is an \textbf{upper facet} if $\tau=\sigma(\{i\},\emptyset)$, and we say it is a \textbf{lower facet} if instead $\tau=\sigma(\emptyset,\{i\})$, for some $i\in[m]$. 
 \end{defi}
 
\begin{example}
    Consider $\sigma = (1,a)(1,b)(1,a)b \in \Ev_3$. Its upper facets are:
    $$\sigma(\{1\},\emptyset) = a(1,b)(1,a)b,\quad  \sigma(\{2\},\emptyset) = (1,a)b(1,a)b,\quad  \text{and}\quad  \sigma(\{3\},\emptyset) = (1,a)(1,b)ab.$$
    
    Meanwhile, it has only two lower facets:
    $$\sigma(\emptyset,\{1\}) = (1,b)(1,a)b, \quad \text{and} \quad \sigma(\emptyset,\{3\}) = (1,a)(1,b)b = (1,a)b(1,b),$$
    because the sub-block $\sigma(\emptyset,\{2\}) = (1,a)(1,a)b$ is not a valid block. 
    
    Note that all upper facets are valid, distinct, and already in canonical form. Moreover, their maximal vertex $aba^2$ is also the maximal vertex of $\sigma$. In contrast,  not all sub-blocks of the form $\sigma(\emptyset,\{i\})$ are valid or in canonical form, but when they do, they are all distinct. These observations hold for all valid cells, as we prove below, along with some additional properties for facets.
\end{example}

 \begin{prop}\label{prop:facets} Let $\sigma\in\Ev_m$ be a valid $m$-block, and let $\tau\prec\sigma$ be  a facet. Let $\eta\leq\sigma$ be a $k$-face, with $k<m$. Let $w=v_{[m]}(\sigma)$ be the maximum word of $\sigma$. Then, the following hold:
     \begin{enumerate}
         \item If $\tau$ is an upper facet, then its maximum word is also $w$, that is, $v_{[m-1]}(\tau)=w$, and there exists a unique $i\in[m]$ such that $\tau=\sigma(\{i\},\emptyset)$. Moreover, the sub-block $\sigma(\{i\},\emptyset)$ is already in canonical form. 
         \item If $\tau$ is a lower facet, then its maximum word is a proper subword of $w$, that is, $v_{[m-1]}(\tau)\lneq w$, and there exists a unique $i\in[m]$ such that $\tau=\sigma(\emptyset,\{i\})$.
        \item If the maximum word of $\eta$ is $v_{[k]}(\eta)=w$, then there exists some upper facet $\tau$, such that $\eta\leq\tau$.
        \item In general, if $\eta\leq\sigma$ is a face of dimension $k<m$, then there exists a facet $\tau$, such that $\eta\leq\tau$.  
     \end{enumerate}
 \end{prop}

 \begin{proof} Write $\sigma=x_0(1,a_1)\cdots (1,a_m)x_m$ in canonical form. Then $w=v_{[m]}(\sigma)=x_0a_1x_1\cdots x_{m-1}a_mx_m$.
 \begin{enumerate}
     \item  Let $\tau = \sigma(\{i\}, \emptyset)$ for some $i \in [m]$. Since $\sigma$ is in canonical form, we have $x_j[1] \neq a_j$ for all $j$. In the representation $\sigma(\{i\}, \emptyset)$ of $\tau$, the word following each edge $(1, a_j)$ is $x_j$ when $j \neq i - 1$, and $x_{i-1}a_ix_i$ when $j = i - 1$. Thus, $\tau$ is in canonical form if and only if the first symbol of $x = x_{i-1}a_ix_i$ is not $a_{i-1}$. If $x_{i-1} \neq 1$, then $x[1] = x_{i-1} \neq a_{i-1}$. On the other hand, if $x_{i-1} = 1$, then $x[1] = a_i$, and we must have $a_i \neq a_{i-1}$, as otherwise, $\sigma$ would be invalid. To prove the uniqueness of $i$, suppose, by way of contradiction, that there exists an index $j < i$ such that $\tau = \sigma(\{j\}, \emptyset)$. As shown above, both $\sigma(\{i\}, \emptyset)$ and $\sigma(\{j\}, \emptyset)$ are already in canonical form, and by Theorem \ref{thm:unique_canonical_form}, the canonical form is unique. Comparing the word before the $j$-th edge in both representations yields the equation $x_{j-1} = x_{j-1}a_jx_j$, which is a contradiction since $a_j \neq 1$. Therefore, we conclude that $j = i$, as desired.

    \item Let $\tau=\sigma(\emptyset, \{i\})$ for some $i\in [m]$. Its maximum word then is $v=v_{[m-1]}(\tau)=$ \\$x_0a_1x_1\cdots x_{i-1}x_i\cdots x_{m-1}a_mx_m$, which clearly is a proper subword of $w$, since it is a symbol shorter. Suppose, by way of contradiction, that there exists an index $j\neq i$ such that $\tau = \sigma(\emptyset, \{j\})$. Thus, they musth have the same maximal word $w$, which can be obtained as $v_{[m]\setminus\{i\}}(\sigma)=w=v_{[m]\setminus\{j\}}$. But then, by Lemma \ref{lemma:charac_valid}, $\sigma$ would be invalid. Therefore, it must be $j=i$, as desired. 
 
    \item  Since $\eta \leq \sigma$, there exist index sets such that $\eta = \sigma(I^+, I^-)$. It is straightforward to observe that the maximum word of $\eta$ can be written as $v_K(\sigma)$, where $K = [m] \setminus I^-$, and moreover, $|w| - |v_K(\sigma)| = |I^-|$. Since that the maximum word of $\eta$ is the same as that of $\sigma$, we must have $I^- = \emptyset$. Also since $\eta \neq \sigma$, we have $I^+ \neq \emptyset$. Let $i \in I^+$, and consider the upper facet $\tau = \sigma(\{i\}, \emptyset)$. Then, we have $\eta = \sigma(I^+, \emptyset) \leq \tau = \sigma(\{i\}, \emptyset)$, as can be easily verified.

    \item Suppose $k\leq m-2$, as otherwise $\eta$ is itself a facet, and there is nothing to prove. As before, we can write $\eta = \sigma(I^+, I^-)$, where $I^+\cap I^-=\emptyset$ and $|I^+|+|I^-|=m-k\geq 2$. Given that there may be multiple choices of indices $I^+$ and $I^-$ that yield the same (equivalent) face $\eta$, we select the one where the sum of the indices in $I^+\cup I^-$, denoted by $\text{sum}(I^+\cup I^-)$, is minimized. First, we will prove that there exists an index $i\in I^+\cup I^-$, such that $\eta_1=\sigma(I^+\setminus\{i\},I^-\setminus\{i\})$ is a valid block. Since $i$ can only belong to either $I^+$ or $I^-$, we have $|I^+ \cup I^- \setminus \{i\}| = m - k - 1$. Therefore, $\dim(\eta_1) = k + 1$, and we have $\eta \prec \eta_1 \leq \sigma$.

    Let $K=[m]\setminus I^+\setminus I^-$, and enumerate its elements as $K=\{i_1<i_2<\cdots<i_k\}$. Then $\eta=y_0(1,a_{i_1})y_1(1,a_{i_2})y_2\cdots (1,a_{i_k})y_k$, with $y_j=x_{i_j}\lambda_{i_j+1}x_{i_j+1}\cdots \lambda_{i_{j+1}-1}x_{i_{j+1}-1}$, where $i_0=0$, $i_{k+1}=m+1$ and $\lambda_j$ are given by 
    $$\lambda_j=\begin{cases}
        a_j &\text{ if } j\in I^+\\
        1 &\text{ if } j\in I^-
    \end{cases}.$$
    We now prove that such index $i$ exist, by considering the three following cases. 
    \begin{enumerate}
        \item If $i_k < m$, let $i = i_k + 1$. Then, $\eta_1$ will have the same expression as $\eta$ up to $(1, a_{i_k})$. Beyond this point, $\eta_1$ will be given by $\eta_1 = \cdots (1, a_{i_k}) x_{i_k} (1, a_{i_k+1}) y$, where $y \in \Sigma^*$. Since $\eta$ is valid, $\eta_1$ could only fail to be valid if $a_{i_k} = a_{i_k+1} = a$ and $x_{i_k} = a^t$ for some $t \geq 0$. However, this would imply that $\sigma$ is invalid, which leads to a contradiction. Therefore, $\eta_1$ must be valid.
        
        \item If $i_k=m$ but $i_1>1$, then let $i=i_1-1$. Then, the expression for $\eta_1$ will match that of $\eta$ starting from $(1, a_{i_1})$. Before this point, $\eta_1$ will be given by $\eta_1=y(1,a_{i_1-1})x_{i_1-1}(1,a_{i_1})\cdots$, with $y\in\Sigma^*$. Again, since $\eta$ is valid, $\eta_1$ could only be invalid if $a_{i_1-1}=a_{i_1}=a$ and $x_{i_1-1}=a^t$, but this is not possible given that $\sigma$ is valid. Thus, $\eta_1$ is valid.

        \item If $i_1 = 1$ and $i_k = m$, then, since $k < m - 1$, there must be a pair of indices $i_s, i_{s+1}$ such that $i_{s+1}$ is not consecutive to $i_s$, i.e., $i_{s+1} > i_s + 1$. Let $i=i_s+1$. The expression of $\eta_1$ will be identical to the expression of $\eta$, except in between $(1,a_{i_s})$ and $(1,a_{i_{s+1}})$, where $\eta_1$ will look like
        $$\eta_1=\cdots (1,a_{i_s})x_{i_s}(1,a_{i_s+1})y(1,a_{i_{s+1}})\cdots,$$
        where $y\in \Sigma^*$ is given by $y=\lambda_{i_s+1}x_{i_s+1}\cdots \lambda_{i_{s+1}-1}x_{i_{s+1}-1}$. Since $\eta$ and $\sigma$ are both valid, the only way $\eta_1$ could be invalid is if $a_{i_s+1}=a_{i_{s+1}}=a$ and $y=a^t$ for some $t\geq 0$. However, this would give a contradiction, since by swapping the index $i_{s+1}$ with $i_s+1$ in the index sets $I^\pm$, we can produce index sets $J^\pm$ such that $\sigma(J^+,J^-)$ is equivalent to $\eta$, but $\text{sum}(J^+\cup J^-)<\text{sum}(I^+\cup I^-)$. 

        Specifically, we define $$J^\pm=\begin{cases}
            I^\pm &\text{ if } i_{s+1}\not\in I^\pm\\
            I^\pm\setminus\{i_{s+1}\}\cup \{i_s+1\} & \text{ if } i_{s+1}\in I^\pm
        \end{cases}.$$
    Then, $\text{sum}(J^+\cup J^-)=\text{sum}(I^+\cup I^-)-i_{s+1}+i_s+1< \text{sum}(I^+\cup I^-)$, since $i_s+1<i_{s+1}$. Then, clearly $\eta=\sigma(I^+, I^-)=\sigma(J^+,J^-)$ since the corresponding terms $(1,a_{s+1})y(1,a_{s+1})$ commute, that is:
    \begin{align}
        (1,a_{i_s+1})y=(1,a)a^t &=a^t(1,a)=y(1,a_{i_{s+1}})\label{eq:comm1}\\
        (1,a_{i_s+1})ya_{i_{s+1}}=(1,a) a^{t+1} &= a^{t+1}(1,a)= a_{i_s+1}y(1,a_{i_{s+1}}) \label{eq:comm2}
    \end{align}
    So the corresponding sub-blocs are equivalent because of equality (\ref{eq:comm1}) if $i_s+1\in I^-$ and because of equality (\ref{eq:comm2}) if $i_s+1\in I^+$.     
    \end{enumerate}

    To finish the proof, now that we know that for any valid face $\eta$, there exists a valid face $\eta\prec\eta_1\leq\sigma$, we can construct a chain of faces $\eta\prec\eta_1\prec\eta_2\cdots\prec \eta_{m-k-1}\prec \sigma$, and thus $\tau=\eta_{m-k-1}$ gives the desired facet. 
 \end{enumerate}
 \end{proof}

\subsection{Isomorphic Blocks}
Theorem \ref{thm:unique_block_by_vertices}, together with Lemmas \ref{lemma:charac_valid} and \ref{lemma:equivalence_subblocks}, and Proposition \ref{prop:facets}, demonstrate that we can gain significant insight into a block structure through its vertices and the relationships between its sub-blocks. We use these relationships to provide a combinatorial definition of when two blocks are isomorphic. As we will see in the next section, block isomorphism will imply isomorphism of chain complexes with coefficients in $\Z_2$, and thus, they will also be isomorphic in a topological sense. With this definition in place, we will conclude this section by classifying all valid blocks in dimensions 0 through 4, up to isomorphism.

\begin{defi}[Block Isomorphism]\label{defi:block_isomorphism}
Given $\sigma_1, \sigma_2\in \Ev_m$ be two valid blocks, we say that they are \textbf{isomorphic}, denoted by $\sigma_1\simeq \sigma_2$ if there exists a permutation of the indices $\pi:[m]\to[m]$ such that the following conditions holds:
\begin{align*}
    \sigma_1(I^+,I^-) \text{ is valid }\quad &\text{if and only if}\quad \sigma_2(\pi(I^+),\pi(I^-))\text{ is valid, and}\\
    \sigma_1(I^+, I^-)=\sigma_1(J^+, J^-)\text{and valid}\quad &\text{if and only if}\quad \sigma_2(\pi(I^+), \pi(I^-))=\sigma_2(\pi(J^+), \pi(J^-))\text{and valid}.
\end{align*}

For any index sets $I^\pm, J^\pm\subset[m]$ with $I^+\cap I^-=\emptyset$ and $J^+\cap J^-=\emptyset$. 

\end{defi}
This means that two blocks are isomorphic if, up to a permutation of their edges, the structure of their sub-blocks is preserved. Checking that this structure holds for all pairs of index sets can quickly become tedious. However, we can reduce the effort required, as we only need to focus on the vertices, as demonstrated in the following theorem.

\begin{theorem}\label{thm:isom_blocks_vertices}
    Given two valid $m$-blocks $\sigma_1, \sigma_2\in\Ev_m$, they are isomorphic $\sigma_1\simeq \sigma_2$ if and only if there exists a permutation $\pi:[m]\to [m]$ such that the following condition holds:
    $$v_I(\sigma_1)=v_J(\sigma_1)\quad\text{if and only if}\quad v_{\pi(I)}(\sigma_2)=v_{\pi(J)}(\sigma_2)$$

    for any $I, J\subset [m]. $
\end{theorem}
\begin{proof}\leavevmode
\begin{itemize}
    \item $(\Rightarrow)$: Suppose $\sigma_1\simeq \sigma_2$. For any subset $I\subset[m]$, we have $\pi([m]\setminus I)=[m]\setminus\pi(I)$. Since $v_I(\sigma_i)=\sigma_i(I,[m]\setminus I)$ we get:
    \begin{align*}
        v_I(\sigma_1)=v_J(\sigma_1) & \Leftrightarrow \sigma_1(I,[m]\setminus I)=\sigma_1( J, [m]\setminus J)       \Leftrightarrow \sigma_2(\pi(I),\pi([m]\setminus I))=\sigma_2(\pi(J),\pi( [m]\setminus J))\\ &\Leftrightarrow \sigma_2(\pi(I),[m]\setminus \pi(I))=\sigma_2(\pi(J), [m]\setminus \pi(J)) \Leftrightarrow v_{\pi(I)}(\sigma_2)=v_{\pi(J)}(\sigma_2). 
    \end{align*}
    \item $(\Leftarrow)$: Suppose that $\sigma_1, \sigma_2\in \Ev_m$ are valid $m$-blocks satisfying the condition of the theorem. Write $\sigma_1=x_0(1,a_1)\cdots (1,a_m)x_m$ and $\sigma_2=y_0(1,b_1)\cdots (1,b_m)y_m$ in canonical form. 
    
    We first prove that if $\tau_1=\sigma_1(I^+, I^-)$ is invalid, then $\tau_2=\sigma_2(\pi(I^+), \pi(I^-))$ is also invalid. The reverse direction follows similarly by switching the role of $\pi$ with $\pi^{-1}$. Let $k=\dim(\tau_1)=\dim(\tau_2)$. By Lemma \ref{lemma:charac_valid}, there exists $i\in [k]$ such that $v_{[k]\setminus\{i\}}(\tau_1)=v_{[k]\setminus\{i+1\}}(\tau_1)$. Then, there exist distinct indices $j, r\in [m]$ such that $$v_{[k]\setminus\{i\}}(\tau_1)=v_{[m]\setminus I^+\setminus\{j\}}(\sigma_1)=v_{[m]\setminus I^+\setminus\{r\}}(\sigma_1)=v_{[k]\setminus\{i+1\}}(\tau_1).$$ Then, we also have $$v_{[k]\setminus\{s\}}(\tau_2)=v_{[m]\setminus\pi(I^+)\setminus\{\pi(j)\}}(\sigma_2)=v_{[m]\setminus\pi(I^+)\setminus\{\pi(s)\}}(\sigma_2)=v_{[k]\setminus\{\ell\}}(\tau_2),$$
    for some distinct indices $s, \ell\in [k]$. Thus, by Lemma \ref{lemma:charac_valid}, $\tau_2$ is also invalid, as desired.

    Consider now the index sets $I_i^\pm$ for $i=1,2$. such that, $\eta=\sigma_1(I_1^+,I_1^-)=\sigma_1(I_2^+,I_2^-)$ and it is a valid block, with $\dim(\eta)=k$. Denote by $J^\pm_i=\pi(I^\pm_i)$ for $i=1,2$. As we proved above, $\tau_1=\sigma_2(J_1^+, J_1^-)$ and $\tau_2=\sigma_2(J_2^+,J_2^-)$ must be valid as well. Clearly $|J_1^+|+|J_1^-|=|I_1^+|+|I_1^-|=|I_2^+|+|I_2^-|=|J_2^+|+|J_2^-|$, so $\dim(\tau_1)=\dim(\tau_2)=k$. 
    Since they are both valid blocks of the same dimension, we just need to prove that $V(\tau_1)=V(\tau_2)$, so by Theorem \ref{thm:unique_block_by_vertices}, we will get $\tau_1=\tau_2$ as needed. We focus on showing $V(\tau_1)\subset V(\tau_2)$ as the other direction is identical. 

    Using the expressions for $\sigma_1$ and $\sigma_2$ already introduced, there must exist bijective functions $\psi_i:[k]\to [m]\setminus I_i^+\setminus I_i^-$ and $\rho_i: [k]\to [m]\setminus J_i^+\setminus J_i^-$, for $i=1, 2$,  such that we can write 
    \begin{align*}
        \eta = \sigma_1(I_i^+, I_i^-)  &= x_0^{(i)}(1,a_{\psi(1)})x_1^{(i)}\cdots x_{k-1}^{(i)}(1,a_{\psi(k)})x_k^{(i)}\text{  and}\\
        \tau_i &=y_0^{(i)}(1,b_{\rho_i(1)})y_1^{(i)}\cdots y_{k-1}^{(i)}(1,b_{\rho_i(k)})y_k^{(i)},
    \end{align*}
    for $i=1, 2$, and some words $x_j^{(i)}, y_j^{(i)}\in\Sigma^*$. 

    Let $v\in V(\tau_1)$, so there exists an index set $J_1\subset [k]$ such that $v=v_{J_1}(\tau_1)$. It is easy to note that $v=v_{J_1}(\tau_1)=v_{J_1^*}(\sigma_2)$, where $J_1^*=J_1^+\sqcup \rho_1(J_1)$, since it is obtained from $\sigma_2$ by substituting the edges $(1,b_j)$ with $b_j$ for $j\in J_1^+$ first, followed by those $j\in\rho_1(J_1)$ next, and substituting all remaining edges with 1. A similar change of indices can be proved in the same manner for all the remainder blocks and sub-blocks. 

    Note that $\rho_1(J_1)\subset [m]\setminus J_1^+\setminus J_1^-$, so $\pi^{-1}\rho_1(J_1)\subset [m]\setminus I_1^+\setminus I_1^-$ and thus $I_1=\psi_1^{-1}\pi^{-1}\rho_1(J_1)\subset [k]$ is well defined. Let $I_1^*=I_1^+\sqcup \psi_1(I_1)$, and note that $\pi(I_1^*)=J_1^*$. Let $w=v_{I_1^*}(\sigma_1)=v_{I_1}(\sigma_1(I_1^+, I_1^-))$. Because $\sigma_1(I_1^+,I_1^-)=\sigma_1(I_2^+,I_2^-)$, there must exist an index set $I_2\subset [k]$ such that $w=v_{I_2}(\sigma_1(I_2^+, I_2^-))$. Thus, letting $I_2^*=I_2^+\sqcup \psi_2(I_2)$ we get $w=v_{I_2^*}(\sigma_1)$. Therefore $v_{I_1^*}(\sigma_1)=v_{I_2^*}(\sigma_1)$, so by the hypothesis $v_{\pi(I_1^*)}(\sigma_2)=v_{\pi(I_2^*)}(\sigma_2)$. Note $\pi(I_2^*)=J_2^+\sqcup\pi\psi_2(I_2)=J_2\sqcup J_2$, if we define $J_2=\pi\psi_2(I_2)\subset [k]$.  Thus $$v=v_{J_1}(\tau_1)=v_{J_1^*}(\sigma_2)=v_{\pi(I_1^*)}(\sigma_2)=v_{\pi(I_2^*)}(\sigma_2)=v_{J_2^+\sqcup J_2}(\sigma_2)=v_{J_2}(\sigma_2(J_2^+,J_2^-))=v_{\rho^{-1}_2(J_2)}(\tau_2).$$
    Therefore, $v\in V(\tau_2)$ as desired, finishing the proof. 
    \end{itemize}    
\end{proof}

\newpage
\begin{examples}\leavevmode
    \begin{enumerate}
        \item $\sigma_1 = (1, a)(1, b)(1, c)$ is isomorphic to $\sigma_2 = (1, a)b(1, a)b(1, a)$. To see this, note that both are valid 3-blocks with $|V(\sigma_1)| = |V(\sigma_2)| = 8$. So each possible index set $I\subset [3]$ produces a different vertex, i.e. $v_I(\sigma_i) = v_J(\sigma_i)$ only when $I = J$. Hence the hypothesis of Theorem \ref{thm:isom_blocks_vertices} is satisfied trivially with $\pi$ as the identity permutation.

        \item $\sigma_1= (1,a)(1,b)(1,a)(1,b)(1,a)(1,b)$ is not isomorphic to $\sigma_2=(1,a)(1,b)(1,c)(1,a)(1,b)(1,c)$. To see this, notice that $a=v_{\{1\}}(\sigma_1)=v_{\{3\}}(\sigma_1)=v_{\{5\}}(\sigma_1)$, however, going through all possible singleton index subsets of $[6]$, we can only produce each vertex of $\sigma_2$ in two ways. So for any possible permutation $\pi$ of $[6]$ we can't produce $v_{\{\pi(1)\}}(\sigma_2)=v_{\{\pi(3)\}}(\sigma_2)=v_{\{\pi(5)\}}(\sigma_2)$. 
    \end{enumerate}
\end{examples}

We point out that an isomorphism of blocks induces a natural bijection between the faces of the blocks, determined entirely by a bijection between their vertices. Suppose $\sigma_1, \sigma_2 \in \Ev_m$ are isomorphic valid blocks. Then the map $\Psi: \calF(\sigma_1) \to \calF(\sigma_2)$ defined by $\Psi(\sigma_1(I^+, I^-)) = \sigma_2(\pi(I^+), \pi(I^-))$, where $\pi$ is the permutation given by the isomorphism, is clearly a well-defined bijection. Moreover, if we restrict it to $\Psi: V(\sigma_1) \to V(\sigma_2)$ by $$\Psi(v_I(\sigma_1))=\Psi(\sigma_1(I,[m]\setminus I))=\sigma_2(\pi(I),[m]\setminus\pi(I))=v_{\pi(I)}(\sigma_2),$$
we get a bijection on the vertices, by Theorem \ref{thm:isom_blocks_vertices}, and it completely determines $\Psi$ on $\calF(\sigma_1)$. This is because for any $\tau_1 \leq \sigma_1$, $\Psi(\tau_1)$ can be defined as the unique valid block $\tau_2 \leq\sigma_2$ such that $\Psi(V(\tau_1)) = V(\tau_2)$. The existence of such $\tau_2$ is clear since $\tau_1 = \sigma_1(I^+, I^-)$ for some index sets, and thus $\tau_2 = \sigma_2(\pi(I^+), \pi(I^-))$ is a valid face with $V(\tau_2) = \Psi(V(\tau_1))$. Uniqueness follows from Theorem \ref{thm:unique_block_by_vertices}, which states that valid blocks are completely determined by their vertices.

The following example characterizes some cases when it is easy to recognize whether two blocks are isomorphic.

 \begin{example}\label{ex:isomorphic_blocks}
    Let $\sigma_1 \in \Ev_m$ be a valid $m$-block given in canonical form by $\sigma_1 = x_0(1, a_1) \cdots (1, a_m) x_m$, where $a_i \in \Sigma$ and $x_i \in \Sigma^*$. For the following definitions of $\sigma_2$, we obtain the isomorphism $\sigma_1 \simeq \sigma_2$:
    \begin{itemize}
        \item \textbf{(Symbol Permutation)}. If $\phi: \Sigma \to \Sigma$ is a permutation of the symbols in the alphabet, then $\sigma_2 = \phi(x_0)(1, \phi(a_1)) \cdots (1, \phi(a_m)) \phi(x_m)$ is clearly isomorphic to $\sigma_1$, with the permutation of indices $\pi$ given by the identity. Here, if $x = b_1 \cdots b_n \in \Sigma^*$ with $b_i \in \Sigma$, we denote $\phi(x) = \phi(b_1) \cdots \phi(b_n)$.
        \item \textbf{(Affixes)}. If $w_0, w_1 \in \Sigma^*$ are fixed words, then defining $\sigma_2 = w_0 \sigma_1 w_1$ produces a block isomorphic to $\sigma_1$. This follows because $(w_0 \sigma_1 w_1)(I^+, I^-) = w_0 \sigma_1(I^+, I^-) w_1$.
        \item \textbf{(Reverse)}. For a word $x = b_1 \cdots b_n$ with $b_i \in \Sigma$, denote its reverse by $x^R = b_n b_{n-1} \cdots b_1$. Define $\sigma_2 = \sigma_1^R = x_m^R(1, a_m) x_{m-1}^R(1, a_{m-1}) \cdots x_1^R(1, a_1) x_0^R$. Then, using the permutation $\pi: i \mapsto m - i$, it is straightforward to see that $\sigma_1 \simeq \sigma_2^R$.
    \end{itemize}
\end{example}

We finish this section with a complete classification of the valid blocks of dimensions 0 through 4, up to isomorphism. 

\begin{theorem}\label{thm:block_classification1_3}
    Let $\sigma\in\Ev_m$ be a valid block of dimension $0\leq m\leq 3$. Then, depending on $m$, $\sigma$ is isomorphic to $\bar{\sigma}$, given as one of the following:
    \begin{enumerate}
        \item If $m=0$, then $\bar{\sigma}=1$.
        \item If $m=1$, then $\bar{\sigma}=(1,a)$.
        \item If $m=2$, then $\bar{\sigma}=(1,a)b(1,a)$.
        \item If $m=3$, then either $\bar{\sigma}=(1,a)(1,b)(1,a)$ or $\bar{\sigma}=(1,a)b(1,a)b(1,a)$. 
    \end{enumerate}
\end{theorem}

\begin{proof}\leavevmode
\begin{enumerate}
    \item For dimensions $m=0$, $1$, or $2$, it is straightforward to see that all possible choices of index sets $I \subset [m]$ produce distinct vertices for both $v_I(\sigma)$ and $v_I(\bar{\sigma})$. The only potential issue arises when $m=2$, if $v_{\{1\}}(\sigma) = x_0 a_1 x_1 x_2 = x_0 x_1 a_2 x_2 = v_{\{2\}}(\sigma)$, but this is impossible since $\sigma$ is valid. Thus, every choice of index set $I \subset [m]$ yields a different vertex, so the condition of Theorem \ref{thm:isom_blocks_vertices} is trivially satisfied, with $\pi$ being the identity permutation. Therefore, $\sigma \simeq \bar{\sigma}$ as desired.

    \item If $m=3$, let $\sigma=x_0(1,a_1)x_1(1,a_2)x_2(1,a_2)x_3$. As in the example, we can remove affixes and find a correspondence of symbols, so $\sigma\simeq (1,a)x(1,b)y(1,c)$, where $ax\neq xb$ and $by\neq yc$ because it is a valid block. To ease the notation, write $\sigma=(1,a)x(1,b)y(1,c)$. It is clear to note that if $v_I(\sigma)=v_J(\sigma)$, then we  must have $|I|=|J|$, this is because the length of the vertex $v_I(\sigma)$ is $|v_\emptyset(\sigma)|+|I|$. We then consider the following cases:
    \begin{enumerate}
        \item Suppose there are sets $I, J$ with $|I|=|J|=1$ such that $v_I(\sigma)=v_J(\sigma)$. Because $\sigma$ is valid, Lemma \ref{lemma:charac_valid} implies the only choice is $I=\{1\}$ and $J=\{3\}$ or vice-versa. Hence $v_{\{1\}}(\sigma)=axy=xyc=v_{\{3\}}(\sigma)$, so by Corollary~\ref{cor:simple_word_equations}, we must have $c=a$ and $x=a^t, y=a^r$ for some $t, r\geq 0$. Therefore, moving like terms and removing affixes we get:
        $$\sigma\simeq (1,a)x(1,b)y(1,c)=(1,a)a^t(1,b)a^r(1,c)=a^t(1,a)(1,b)(1,a)a^r\simeq (1,a)(1,b)(1,a). $$

        \item Suppose all singleton indices $I\subset[3]$ produce different vertices $v_I(\sigma)$. We can check that all possible tuples $I\subset [3]$ will also produce different vertices. Up to reversing, we only need to check two cases:
        \begin{itemize}
            \item If $v_{\{1,2\}}(\sigma)=v_{\{1,3\}}(\sigma)$, that is $axby=axyc$ but then $by=yc$, which can't happen. 
            \item If $v_{\{1,2\}}(\sigma)=v_{\{2,3\}}(\sigma)$, that is $axby=xbyc$. By Corollary~\ref{cor:simple_word_equations}, that implies $a=c=b$, which can't happen either. 
        \end{itemize}
        Thus, all possible sets $I\subset [3]$ produce different vertices. The same can be easily checked for $(1,a)b(1,a)b(1,a)$. Thus, by Theorem \ref{thm:isom_blocks_vertices}, we get $\sigma\simeq (1,a)b(1,a)b(1,a)$, as desired. 
    \end{enumerate}
     Finally, the fact that the blocks $(1,a)(1,b)(1,a)$ and $(1,a)b(1,a)b(1,a)$ are not isomorphic is straightforward, as they have a different number of vertices.
\end{enumerate}
\end{proof}

\begin{figure}
	\centering
		\begin{subfigure}[a]{0.50\textwidth}
			\centering
		\includegraphics[width=0.6\textwidth]{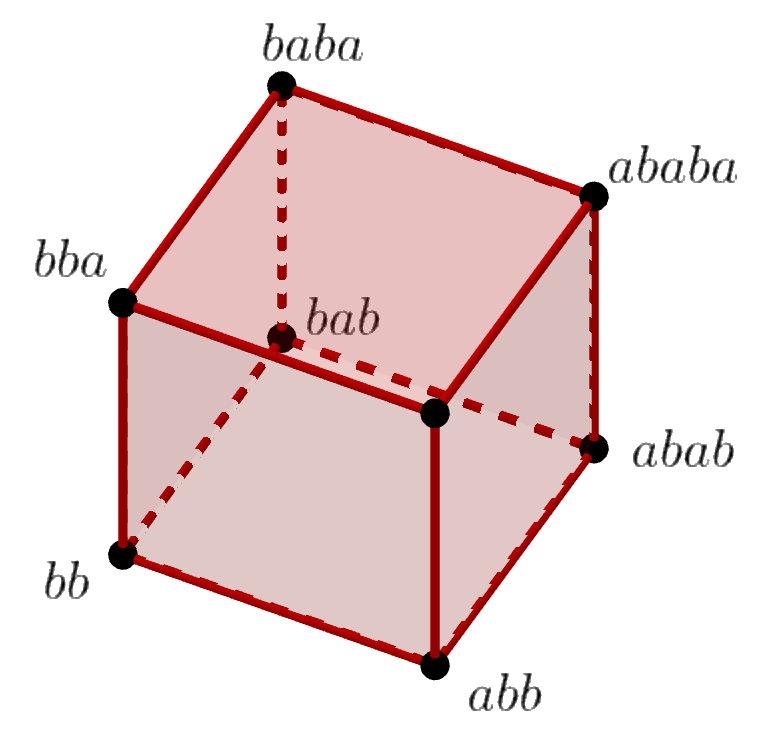}
			\caption{$(1,a)b(1,a)b(1,a)$}
		\end{subfigure}%
		\begin{subfigure}[a]{0.5\textwidth}
			\centering
        \includegraphics[width=0.6\textwidth]{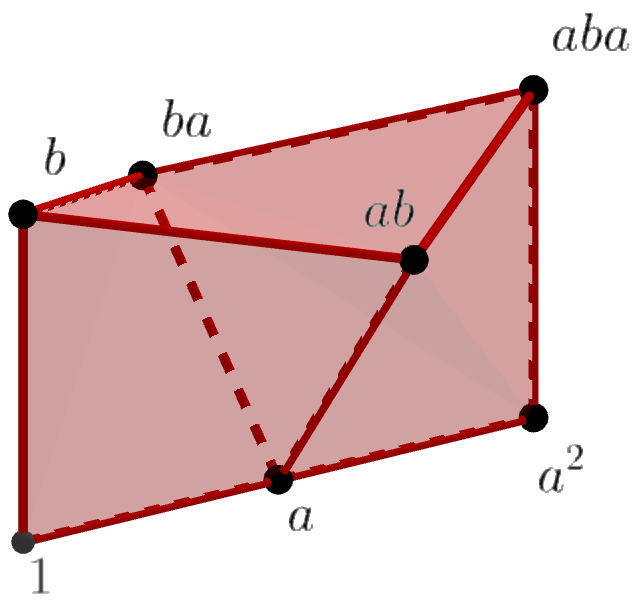}
			\caption{$(1,a)(1,b)(1,a)$}
		\end{subfigure}
		\caption{Possible 3-blocks, up to isomorphism}
\end{figure}

\begin{theorem}\label{thm:block_classification4}
    Let $\sigma\in\Ev_4$ be a valid block of dimension 4. Then, $\sigma$ is isomorphic to exactly one of the following blocks:
    \begin{enumerate}
         \item $\bar{\sigma}_1=(1,a)(1,b)(1,a)(1,b)$
        \item $\bar{\sigma}_2=(1,a)(1,b)(1,a)b(1,a)$
        \item $\bar{\sigma}_3=(1,a)(1,b)a^2(1,b)(1,a)$
        \item $\bar{\sigma}_4=(1,a)(1,b)a(1,b)(1,a)$
        \item $\bar{\sigma}_5=(1,a)(1,b)ab(1,a)(1,b)$
        \item $\bar{\sigma}_6=
(1,a)b(1,a)b(1,a)b(1,a)$
    \end{enumerate}
\end{theorem}

The proof of Theorem \ref{thm:block_classification4} follows a similar approach to that of Theorem \ref{thm:block_classification1_3}; however, it requires the consideration of many more cases. To maintain the flow of the text, we have relocated its proof to the \hyperlink{proof:block_classification4}{appendix}.

\begin{remark}  
    Note that the classification of all valid block isomorphism types given by Theorems \ref{thm:block_classification1_3} and \ref{thm:block_classification4} requires blocks using at most two distinct symbols, $\Sigma = \{a, b\}$. Let $\mu_\text{sym}(m)$ denote the minimum number of symbols needed to classify all isomorphic types of valid $m$-blocks. From Theorems \ref{thm:block_classification1_3} and \ref{thm:block_classification4}, it is clear that $\mu_\text{sym}(0) = 0$, $\mu_\text{sym}(1) = 1$, and $\mu_\text{sym}(m) = 2$ for $m = 2, 3$, and 4. It is straightforward to observe that $2 \leq \mu_\text{sym}(m) \leq m$ for any $m \geq 4$. Indeed, at least two distinct symbols are required to form a valid block, while $m$ symbols are sufficient in the case where all vertices $v_I(\sigma)$ are distinct. Any coincidence of vertices $v_I(\sigma) = v_J(\sigma)$ would imply that the sets of symbols $\{a_i: i \in I\}$ and $\{a_j: j \in J\}$ must be equal, thus requiring fewer than $m$ symbols to describe the isomorphism type of $\sigma$. The behavior of $\mu_\text{sym}(m)$ for $m \geq 5$ remains an open question. In particular, we wonder whether $\mu_\text{sym}(m) > 2$ for any $m$, or if two symbols suffice for all dimensions.  
    \end{remark}

\section{Insertion Block and Insertion Chain Complexes}
\label{sec:insertion_chain_complex}

\subsection{Insertion Block Complex}
In section \ref{sec:blocks}, we defined valid $m$-blocks for understanding sets of words that arise transitioning from a minimum word $w_m = v_\emptyset(\sigma) = x_0 x_1 \cdots x_{m-1} x_m$ to a maximum word $w_M = v_{[m]}(\sigma) = x_0 a_1 x_1 \cdots x_{m-1} a_m x_m$, by incorporating the symbols $a_1, \cdots, a_m$ independently. We also introduced a canonical form for valid blocks, what it means for blocks to be isomorphic, and a boundary operator for blocks. With all these tools, we are finally ready to define a topological object associated to any set of words. 
	
\begin{defi}[Insertion Block Complex]\label{def:Insertion_block_complex}
Let $\Sigma$ be a set of symbols, and consider $\Ev_m$ the corresponding sets of valid $m$-blocks in $\Sigma$. Given a set of words $W\subset\Sigma^*$, we define its \textbf{(induced) Insertion Block Complex} $\Cins{W}$, as the collection of all valid cells $\sigma\in \Ev_m$ such that $\sigma$ is supported in $W$, i.e. 
$$\Cins{W}=\left\{\sigma: \sigma\in\Ev_m\text{ for some }m, \text{ and } V(\sigma)\subset W\right\}. $$
We define the \textbf{dimension} of $\Cins{W}$, $\dim(\Cins{W})=n$, to be the maximum dimension of its blocks, that is $n=\max\left\{ \dim(\sigma): \sigma\in\Cins{W}\right\}.$ 
\end{defi}

\begin{remark}
We point out that if $\tau\leq \sigma$ and $\sigma\in \Cins{W}$, that means $V(\sigma)\subset W$, and since $V(\tau)\subset V(\sigma)$, we also have $\tau\in\Cins{W}$. So the Insertion Block Complex is closed under taking faces. We also note that if $W$ is finite, then $\Cins{W}$ will be finite as well, since, by Theorem \ref{thm:unique_block_by_vertices}, we will get at most one block per each subset of vertices, thus $|\Cins{W}|\leq 2^{|W|}$.     
\end{remark}

\newpage
\begin{examples}\label{ex:insertion_blocks}\leavevmode
    \begin{enumerate}
        \item If $W=\{1,a,ab,bab,ba,c,ac,bd,bde\}$, then $\Cins{W}$ is the following set (shown in Figure \ref{fig:examples_Cins}a) \begin{align*}
            \Cins{W}=&\{1,a,ab,bab,ba,c,ac,bd,bde &\text{(0-blocks)}\\
            &(1,a), a(1,b), b(1,a), ba(1,b),(1,b)ab,(1,c), a(1,c), (1,a)c,bd(1,e) & \text{(1-blocks)}\\
            &(1,b)a(1,b),(1,a)(1,c)\} &\text{(2-blocks)}
        \end{align*}
        
        \item Let $\sigma=(1,a)(1,b)aba(1,b)$. Then $\sigma$ is the maximal block in $\Cins{V(\sigma)}$, and we can understand it as representing 3 \textit{independent} insertions to go from the word $aba$ to $ababab$. However, there is an extra relationship among the words in $V(\sigma)$ that is not captured by $\sigma$, since we also get the edge $abab(1,a)\in\Cins{V(\sigma)}$, joining the vertices $abab, ababa\in V(\sigma)$, that is not an edge of $\sigma$. $\Cins{V(\sigma)}$ is depicted in Figure \ref{fig:examples_Cins}b.
    \end{enumerate}
\end{examples}

\begin{note2}\leavevmode
\begin{enumerate}
    \item  Given a valid $m$-block $\sigma$, clearly $\sigma\in \Cins{V(\sigma)}$, and it must be the maximum block in the insertion complex. Indeed, if $\tau\in\Cins{V(\sigma)}$, and if $\dim(\tau)=k$, then the length difference between its maximal and minimal words is $|v_{[k]}(\tau)|-|v_\emptyset(\tau)|=k$, and since both $v_\emptyset(\tau), v_{[k]}(\tau)\in V(\sigma)$, this difference is at most $m$, so $k\leq m$. Moreover, if $\dim(\tau)=m$, it must be that $v_\emptyset(\tau)=v_\emptyset(\sigma)$ and $v_{[m]}(\tau)=v_{[m]}(\sigma)$, and since $V(\tau)\subset V(\sigma)$, by Theorem \ref{thm:unique_block_by_vertices2}, we must have $\tau=\sigma$.
    \item Recall that the $k$-skeleton of $\Cins{W}$, denoted by $\Cins{W}^{(k)}$ is the subset of blocks of dimension at most $k$. In particular, the $1$-skeleton $\Cins{W}^{(1)}$ consists of the vertices in $W$ and (equivalence classes of) valid 1-blocks, which, as we explored in section \ref{sec:blocks}, correspond exactly to the edges of the insertion graph $\Gins{W}$, so we get $\Cins{W}^{(1)}\simeq \Gins{W}$. 
\end{enumerate}
\end{note2}

	\begin{figure}
		\centering
		\begin{subfigure}[b]{0.05\textwidth}
			{}
		\end{subfigure}
		~
		\begin{subfigure}[b]{0.45\textwidth}
			\centering
            \includegraphics[height=2in]{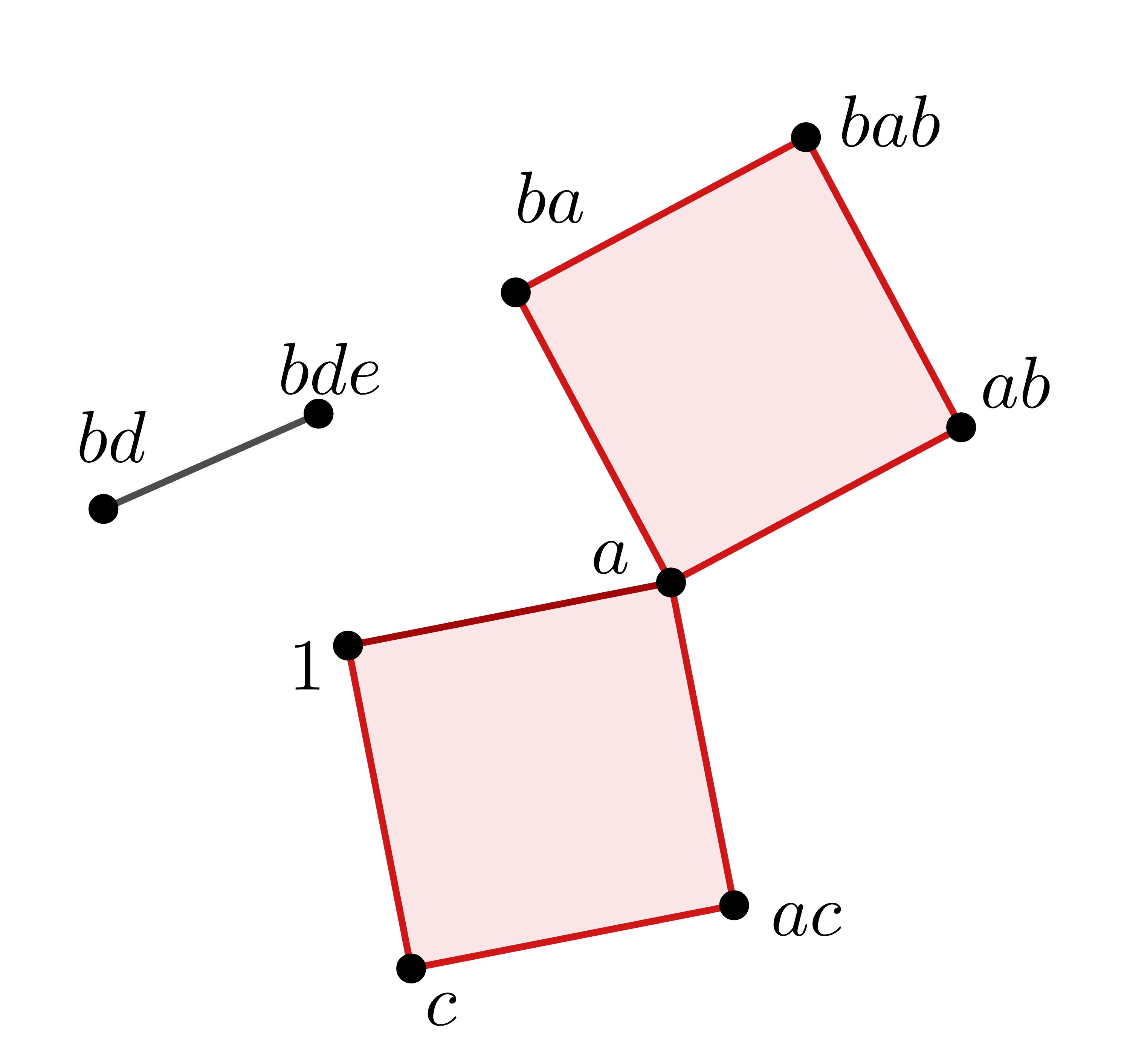}
			\caption{$\Cins{W}$ for $W=\{1,a,ab,bab,ba,c,ac,bd,bde\}$}
		\end{subfigure}%
		~ 
		\begin{subfigure}[b]{0.45\textwidth}
			\centering
			\includegraphics[height=2in]{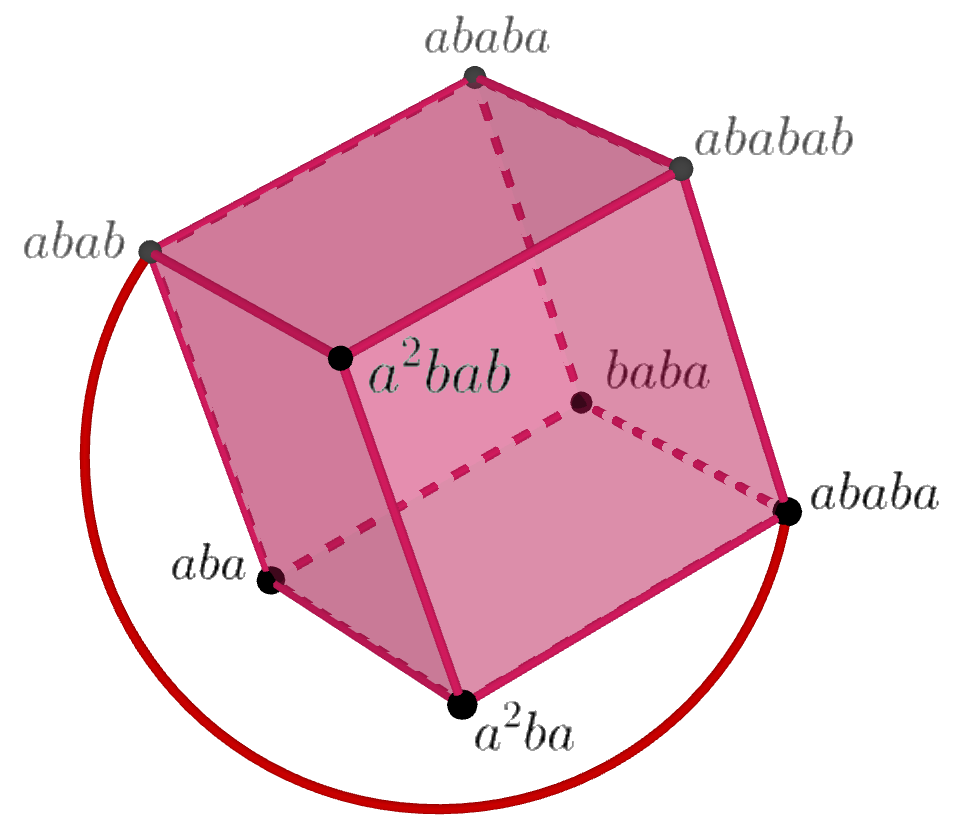}
			\caption{$\Cins{V(\sigma)}$ for $\sigma=(1,a)(1,b)aba(1,b)$}
		\end{subfigure}
		~
		\begin{subfigure}{0.05\textwidth}
			{}
		\end{subfigure}
		\caption{Examples of Insertion Block Complexes $\Cins{W}$. }\label{fig:examples_Cins}
	\end{figure}

\subsection{Boundary Operator}
As introduced earlier, blocks help us study how sets of words are connected by \textit{filling up holes}.  A well-established approach to defining and analyzing such \textit{holes} is through homology (see e.g., \cite{hatcher2002algebraic}). For this, we now define the \textbf{boundary operator}, a key tool for this framework.
 
Fix a finite alphabet $\Sigma$, and consider the sets  $\Ev_m$ of valid $m$-blocks with symbols from $\Sigma$ for $m \geq 0$. Recall that for any set $S$, we can construct the free abelian group generated by $S$, $\Z[S]$, consisting of (finite) formal sums of elements from $S$ with coefficients in $\Z$. In particular, we may consider the modules generated by $\Ev_m$, meaning that any element $c \in \Z\left[\Ev_m\right]$ can be expressed as $c = \sum_{i=1}^n \alpha_i \sigma_i$, where $\alpha_i \in \Z$ and $\sigma_i \in \Ev_m$. A linear map $\lambda: \Z[\Ev_m] \to \Z[\Ev_k]$ is then completely determined by its action on the valid $m$-blocks.

\begin{defi}\label{defi:boundary}
    Given a valid $m$-block $\sigma\in \Ev_m$, with $m\geq 1$, we define its boundary $\partial\sigma$, as the formal sum of its facets in $\Z[\Ev_{m-1}]$ given as follows:
    $$\partial \sigma =\sum_{i=1}^m (-1)^{m+1}\left[ \bar{F}_i(\sigma)-\underline{F}_i(\sigma) \right], $$
    where $\bar{F}_i(\sigma)=\sigma(\{i\},\emptyset)$ and $$\underline{F}_i(\sigma)=\begin{cases}
        \sigma(\emptyset,\{i\}) &\text{ if it is a valid } (m-1)\text{-block}\\
        0 &\text{ otherwise}
    \end{cases}.$$
\end{defi}

By Proposition \ref{prop:facets}, all upper-facets $\bar{F}_i(\sigma)$ are distinct valid $(m-1)$-blocks. For the lower-facets, we specify that $\underline{F}_i(\sigma)$ is a valid block, otherwise we set it to zero. Thus, $\partial\sigma$ is a linear combination of all facets of $\sigma$, with $\pm1$ as coefficients. In particular, this means $\partial\sigma\neq 0$ for any $\sigma\in\Ev_m$ with $m\geq 1$, since $\partial\sigma$ will always include the $m$ different upper-facets of $\sigma$.  

By Proposition \ref{prop:facets}, all upper-facets $\bar{F}_i(\sigma)$ are distinct and valid $(m-1)$-blocks. For the lower-facets, we define $\underline{F}_i(\sigma)$ to be a valid block, otherwise, we set it to zero. Hence, $\partial \sigma$ is a linear combination of all the facets of $\sigma$, with coefficients $\pm1$. In particular, this implies that $\partial \sigma \neq 0$ for any $\sigma \in \Ev_m$ with $m \geq 1$, since $\partial \sigma$ always contains the $m$ distinct upper-facets of $\sigma$.

Since $\partial$ is defined for each valid $m$-block, we can naturally extend it to a linear map $\partial: \mathbb{Z}[\mathcal{E}^\square_m] \to \mathbb{Z}[\mathcal{E}^\square_{m-1}]$. This same definition applies to every dimension $m \geq 1$, giving rise to a family of linear maps $\partial_m$, which we complete by defining $\partial_0 \equiv 0$ as the constant zero map. For simplicity, we adopt the common convention of always using the symbol $\partial$ without subscripts, with its meaning determined by context.

\begin{lemma}\label{lemma:partial_partial_zero}
Let $\sigma\in\Ev_m$ be a valid $m$-block, then
    $$\partial\partial\sigma =0. $$
\end{lemma}

This result follows from a straightforward computation that should be familiar to readers experienced with simplicial or cubical complexes. Although some extra attention is required to address cases involving invalid sub-blocks, the argument itself is simple. For this reason, we have omitted the proof here and move it to the \hyperlink{proof:partial_partial_zero}{appendix}.

\subsection{Insertion Chain Complex}
We will now introduce a \textbf{chain complex}, as is customary in homological algebra (see e.g. \cite{hatcher2002algebraic, cartan1999homological}). Recall that a chain complex $C_*$ is a collection of abelian modules $C_i$ for $i \in \Z$, together with linear boundary maps $\partial_i: C_i \to C_{i-1}$, satisfying $\partial_i \partial_{i-1} \equiv 0$ for all $i \in \Z$. By selecting the abelian modules as subsets of $\mathbb{Z}[\Ev_m]$ and using the boundary operator $\partial$, we can construct a chain complex for any set of blocks that is closed under taking faces.

\begin{defi}\label{def:insertion_chain_complex}\leavevmode
\begin{enumerate}
    \item Let $S$ be a set of valid blocks that is closed under taking faces, that is, if $\sigma\in S$ and $\tau\leq \sigma$, then $\tau\in S$. Let $n=\max\{\dim(\sigma): \sigma\in S\}$, and for each $0\leq k\leq n$, define $S_k=\{\sigma\in S: \dim(\sigma)=k\}$. 
    Then, its associated \textbf{Chain Complex} $\Chain{S}$, is the chain complex with modules $\ChainSub{k}{S}=\Z[S_k]$ for $0\leq k\leq n$, $\ChainSub{k}{S}=0$ otherwise, and $\partial$ the corresponding boundary operator as defined in Definition \ref{defi:boundary}. Note Lemma \ref{lemma:partial_partial_zero} guarantess this is a chain complex. 
    \item Given a set of words $W\subset\Sigma^*$, we define its \textbf{Insertion Chain Complex} $\ChainBlock{W}$ to be the chain complex associated with the collection of blocks $\Cins{W}$, that is $\ChainBlock{W}=\Chain{\Cins{W}}$. 
    \end{enumerate}
\end{defi}

\begin{remark}
    While we have chosen to define Chain Complexes with coefficients in $\Z$, it is important to note that Chain Complexes can also be defined with coefficients in any abelian group, such as $\Z_k$. Unless otherwise specified, we will work with coefficients in $\Z$, but we note that certain results are proved only for $\Z_2$, which will be explicitly stated where applicable.
\end{remark}

With $\ChainBlock{W}$ established as a chain complex, we can apply the standard techniques from homological algebra, including defining cycles, boundaries, and calculating homology. Instead of providing all these definitions here, we direct the reader to \cite{hatcher2002algebraic} and \cite{cartan1999homological} for a detailed treatment.


We observe that isomorphic blocks are topologically equivalent as well, in the sense that their chain complexes are isomorphic. However, as seen in the Example \ref{ex:insertion_blocks}, depending on the specific words at the vertices, constructing the insertion chain complex on the vertices of a block may include additional blocks that are not faces of the original. Therefore, we  prove that the chain complexes of the faces of two equivalent blocks are isomorphic over the ring $\Z_2$.

\begin{lemma}\label{lemma:isom_blocks_to_chains}
    If $\sigma_1\simeq\sigma_2$ are two equivalent valid blocks, then $\ChainBlock{\calF(\sigma_1)}$ is isomorphic to $\ChainBlock{\calF(\sigma_2)}$, as chain complexes with coefficients in $\Z_2$. In the special case where they are equivalent via the identity or a reverse permutation, they are isomorphic with coefficients in $\Z$.
\end{lemma}
\begin{proof}
    Suppose $\sigma_1, \sigma_2\in\Ev_m$ are equivalent through a permutation $\pi$ of $[m]$. Then, defining ${f:\calF(\sigma_1)\to\calF(\sigma_2)}$ as $f(\sigma_1(I^+, I^-))=\sigma_2(\pi(I^+), \pi(I^-))$ for all disjoint $I^+, I^-\subset [m]$ clearly produces a bijective map. Moreover, restricting $f$ to $f_k$ between $k$-faces, produces a collection of bijections between $\calF_k(\sigma_1)$ and $\calF_k(\sigma_2)$. Thus, we may extend these bijections linearly on the modules ${f_k:\Z_r\left[\calF_k(\sigma_1)\right]\to \Z_r\left[\calF_k(\sigma_2)\right]}$, obtaining a collection of isomorphisms, for any $r\geq 1$. To finish the proof, we need to show that these degree isomorphisms commute with the boundary maps when $r=2$.  

    Let $\tau \leq \sigma_1$ be a valid $k$-dimensional sub-block. Write $\tau=\sigma_1(I^+, I^-)$, so $\tau'=f_k(\tau)=\sigma_2(\pi(I^+), \pi(I^-))$. Let $K=\{i_1<i_2<\cdots<i_k\}=[m]\setminus I^+\setminus I^-$, and note $\pi(K)=\{j_1<j_2<\cdots<j_k\}=[m]\setminus\pi(I^+)\setminus\pi(I^-)$. Clearly the map $\psi:[k]\to[k]$ given by $\psi(r)=s$ if $\pi(i_r)=j_s$ is a permutation. Moreover, note that for $J^\pm\subset [k]$ disjoint index sets, we have 
    \begin{align*}
        f_k\left(\tau(J^+, J^-)\right)&=f_k\left(\sigma_1(I^+\cup \{i_r:r\in J^+\}, I^-\cup\{i_r: r\in J^-\}), \text{ and}\right)\\
        &=\sigma_2(\pi(I^+)\cup \{\pi(i_r):r\in J^+\}, \pi(I^-)\cup\{\pi(i_r): r\in J^-\}))\\
        &= \sigma_2\left(\pi(I^+)\cup \{j_s: s\in \psi(J^+)\},\pi(I^-)\cup \{j_s: s\in \psi(J^-)\} \right)\\
        &=\tau'\left(\psi(J^+), \psi(J^-) \right).
    \end{align*}
     Thus $f_k(\bar{F}_i(\tau))=f_k(\tau(\{i\},\emptyset))=\tau'(\{\psi(i)\}, \emptyset)=\bar{F}_{\psi(i)}(\tau')$ and similarly $f_k(\underline{F}_i(\tau))=\underline{F}_{\psi(i)}(\tau')$, since they are either both valid facets or both equal to zero. 
    Then, we observe
    \begin{align*}
        f_k(\partial\tau)&=\sum_{i=1}^k (-1)^{i+1} \left[f_k(\bar{F}_i(\tau))-f_k(\underline{F}_i(\tau))\right]=\sum_{i=1}^k (-1)^{i+1}\left[ \bar{F}_{\psi(i)}(\tau')-\underline{F}_{\psi(i)}(\tau')\right]=\\
        &= \sum_{i=1}^k (-1)^{\psi^{-1}(i)+1}\left[\bar{F}_i(\tau')-\underline{F}_i(\tau') \right] \stackrel{*}{=}\partial(\tau')=\partial(f_k(\tau)) & (*)
    \end{align*}
    where the first equality in line $(*)$ is obtained by reordering the terms, and the equality $\stackrel{*}{=}$ holds provided $(-1)^{\psi^{-1}(i)+1}=(-1)^{i+1}$, which will always be the case when our coefficients are computed in $\Z_2$, proving the first part of the theorem. 
    
   In the special case $\pi$ is the identity permutation, $\psi$ will also be the identity, and the equality $(*)$ holds trivially for coefficients in $\Z$. If $\pi$ is the reverse function, $\pi(i)=m+1-i$, then we have $\psi(i)=k+1-i$. Hence defining $f_k(\sigma(I^+, I^-))=(-1)^{k-1}\sigma(I^+, I^-)$ instead, equation $(*)$ will still hold.
\end{proof}

In general, if $\Phi: \ChainBlock{W_1} \to \ChainBlock{W_2}$ is a chain isomorphism, then it induces a bijection on words when restricted to $\Phi: W_1 \to W_2$, and for higher dimensions, it is entirely determined by its mapping on vertices. Specifically, $\Phi(\sigma) = \sigma'$, where $\sigma'$ is the unique valid block such that $V(\sigma') = \Phi(V(\sigma))$. Indeed, since $\partial \Phi(\sigma) = \Phi(\partial \sigma)$, $\Phi$ gives a bijection between the facets of $\sigma$ and those of $\Phi(\sigma)$. If $v \in V(\sigma)$, by Proposition \ref{prop:facets}, there exists a chain of facets $v \prec \tau_1 \prec \tau_2 \prec \dots \prec \tau_{m-1} \prec \sigma$, where each $\tau_i$ is a facet of $\tau_{i+1}$. Thus, we obtain the chain $\Phi(v) \prec \Phi(\tau_1) \prec \cdots \prec \Phi(\sigma)$, and in particular $\Phi(v) \leq \Phi(\sigma)$, so $\Phi(v) \subset V(\Phi(\sigma))$ and thus $\Phi(V(\sigma))\subset V(\Phi(\sigma))$. The reverse direction follows similarly. Therefore, to construct a chain isomorphism, it is sufficient to specify a bijection between the sets of words, as long as the induced map on faces commutes with the boundary operator.

\begin{prop}
   Let $W\subset\Sigma^*$ be a set of words. Then, the following bijections of words $\Phi:W\to \Phi(W)$, provide chain isomorphisms $\Phi:\Cins{W}\to\Cins{W'}$, where $W'=\Phi(W)$.
   \begin{enumerate}
        \item \textbf{(Symbol Permutation)}. Let $\phi: \Sigma \to \Sigma$ be a permutation of the symbols in the alphabet, and for if $x = b_1 \cdots b_n \in W$ with $b_i \in \Sigma$, define $\Phi(x) = \phi(b_1) \cdots \phi(b_n)$.
        \item \textbf{(Affixes)}. If $w_0, w_1 \in \Sigma^*$ are fixed words, then define $\Phi(x)=w_0xw_1$ for all $x\in W$.  
        \item \textbf{(Reverse)}. For a word $x = b_1 \cdots b_n$ with $b_i \in \Sigma$, denote its reverse by $x^R = b_n b_{n-1} \cdots b_1$. Define $\Phi(x) =x^R$. 
   \end{enumerate}
\end{prop}

\begin{proof}
    Note that for any of these definitions for $\Phi$, if $\tau\in\ChainBlock{W}$, then $\Phi(\tau)$ is isomorphic to $\tau$ as blocks, as we showed in Example \ref{ex:isomorphic_blocks}, with the map between their indices being either the identity or the reverse map. Then, the proof given in Lemma \ref{lemma:isom_blocks_to_chains} also shows that $\partial \Phi(\tau)=\Phi(\partial(\tau))$, after setting the signs of $\text{sgn}(\Phi(\tau))=(-1)^{\dim(\tau)+1}$ for the reverse mapping. Thus, we get the Chain isomorphism as desired. 
\end{proof}

\section{Homology}
\label{sec:homology}

We conclude this paper by examining the meaning and behavior of homology in the Insertion Chain Complex for different sets of words. In this section, we characterize all minimal nontrivial 1-dimensional cycles. We demonstrate that every abelian group can be realized as the $k$-th homology group of the Insertion Chain Complex for some set of words. We also investigate the minimal sets of words that yield the same homology as $k$-spheres. Finally, we prove that for many sets of words consisting of a minimum word, a maximum word, and all intermediate words, the homology vanishes in all dimensions.

\subsection{Nontrivial 1-Dimensional Cycles}

Recall that a chain $c\in\ChainBlock{W}$ is a nontrivial cycle, if $\partial(c)=0$, but $c\neq \partial c'$ for any other chain $c'$ of one dimension higher. Note that 
all single vertices are 0-dimensional cycles, and pairs of vertices are not 0-dimensional cycles if and only if they expand an edge, i.e. if they differ by exactly one symbol insertion or deletion. Classifying 1-dimensional cycles is the the first non-trivial case. Since the 1-skeleton of $\Cins{W}$ can be seen as the graph $\Gins{W}$, all 1-dimensional cycles are cycles in the graph-theoretic sense. Since $\Gins{W}$ is bipartite, the shortest possible cycles are of length 4. As all 2-cells are squares (Theorem \ref{thm:block_classification1_3}), and their boundaries are cycles of length 4, it becomes important to classify all cycles of length 4 that can't be expressed as the boundary of a two cell.

\begin{theorem}\label{thm:classification_cycles_1dim}
    Let $\Sigma$ be a finite alphabet, and $W\subset\Sigma^*$, with $|W|=4$, be a set of 4 vertices. Then, if $H_1(\ChainBlock{W})\neq 0$, then $\ChainBlock{W}$ is isomorphic via reflection, symbol permutation or affixes to $\ChainBlock{V}$ where $V\subset \{a,b\}^*$ is one of the following:
    \begin{enumerate}
        \item $V_1=\left\{(ab)^t, a(ab)^t, (ab)^ta,(ab)^{t+1} \right\}$, for some $t\geq 1$. 
        \item $V_2=\left\{(ab)^t, a(ab)^t, (ab)^tb,(ab)^{t+1} \right\}$, for some $t\geq 2$. 
        \item $V_3=\left\{(ab)^t,(ab)^ta,b(ab)^t,(ab)^{t+1}\right\}$, for some $t\geq 1$. 
         \item $V_4=\left\{(ab)^ta,b(ab)^ta,(ab)^{t+1},b(ab)^t\right\}$, for some $t\geq 0$.
        \item $V_5=\left\{(ab)^{t+1},(ab)^{t+1}a,b(ab)^{t+1},b(ab)^ta\right\}$, for some $t\geq 0$.
         \item $V_6=\left\{(ab)^{t+1},a(ab)^{t+1},(ab)^{t+1}b,a(ab)^b\right\}$, for some $t\geq 1$.              
    \end{enumerate}
\end{theorem}

The proof of this theorem requires considering numerous cases and solving systems of simultaneous word equations. To maintain the flow of the paper, we have opted to move the proof to the \hyperlink{proof:classification_cycles_1dim}{appendix}.

\begin{figure}
\centering
\begin{tikzpicture}
    \node at (0,0) (A) {$(ab)^t$} ; 
    \node at (2,0) (B) {$a(ab)^t$} ; 
    \node at (0,2) (D) {$(ab)^ta$} ; 
    \node at (2,2) (C) {$(ab)^{t+1}$} ; 
    
    \path[->] (A) edge node[auto=left] {$a$} (B);
    \path[->] (A) edge node[auto=left] {$a$} (D);
    \path[->] (B) edge node[auto=left] {$b$} (C);
    \path[->] (D) edge node[auto=left] {$b$} (C);
\end{tikzpicture}
\begin{tikzpicture}
    \node at (0,0) (A) {$(ab)^t$}; 
    \node at (2,0) (B) {$a(ab)^t$}; 
    \node at (0,2) (D) {$(ab)^tb$}; 
    \node at (2,2) (C) {$(ab)^{t+1}$}; 
    
    \path[->] (A) edge node[auto=left] {$a$} (B);
    \path[->] (A) edge node[auto=left] {$b$} (D);
    \path[->] (B) edge node[auto=left] {$b$} (C);
    \path[->] (D) edge node[auto=left] {$a$} (C);
\end{tikzpicture}
\begin{tikzpicture}
    \node at (0,0) (A) {$(ab)^t$}; 
    \node at (2,0) (B) {$(ab)^ta$}; 
    \node at (0,2) (D) {$b(ab)^t$}; 
    \node at (2,2) (C) {$(ab)^{t+1}$}; 
    
    \path[->] (A) edge node[auto=left] {$a$} (B);
    \path[->] (A) edge node[auto=left] {$b$} (D);
    \path[->] (B) edge node[auto=left] {$b$} (C);
    \path[->] (D) edge node[auto=left] {$a$} (C);
\end{tikzpicture}\\
\begin{tikzpicture}
    \node at (0,0) (A) {$(ab)^ta$}; 
    \node at (2,0) (B) {$b(ab)^ta$}; 
    \node at (0,2) (D) {$(ab)^{t+1}$}; 
    \node at (2,2) (C) {$b(ab)^t$}; 
    
    \path[->] (A) edge node[auto=left] {$b$} (B);
    \path[->] (A) edge node[auto=left] {$b$} (D);
    \path[<-] (B) edge node[auto=left] {$a$} (C);
    \path[<-] (D) edge node[auto=left] {$a$} (C);
\end{tikzpicture}
\begin{tikzpicture}
    \node at (0,0) (A) {$(ab)^{t+1}$}; 
    \node at (2,0) (B) {$(ab)^{t+1}a$}; 
    \node at (0,2) (D) {$b(ab)^{t+1}$}; 
    \node at (2,2) (C) {$b(ab)^{t}a$}; 
    
    \path[->] (A) edge node[auto=left] {$a$} (B);
    \path[->] (A) edge node[auto=left] {$b$} (D);
    \path[<-] (B) edge node[auto=left] {$a$} (C);
    \path[<-] (D) edge node[auto=left] {$b$} (C);
\end{tikzpicture}
\begin{tikzpicture}
    \node at (0,0) (A) {$(ab)^{t+1}$}; 
    \node at (2,0) (B) {$a(ab)^{t+1}$}; 
    \node at (0,2) (D) {$(ab)^{t+1}b$}; 
    \node at (2,2) (C) {$a(ab)^{t}b$}; 
    
    \path[->] (A) edge node[auto=left] {$a$} (B);
    \path[->] (A) edge node[auto=left] {$b$} (D);
    \path[<-] (B) edge node[auto=left] {$a$} (C);
    \path[<-] (D) edge node[auto=left] {$b$} (C);
\end{tikzpicture}
\caption{Nontrivial 1-dimensional Cycles (Theorem \ref{thm:classification_cycles_1dim}.)}
    \label{fig:enter-label}
\end{figure}

	\subsection{Realizing Homology}
	A natural question regarding the Insertion Chain Block construction is the realizability of homology groups; specifically, determining which abelian groups can arise as the homology of the Insertion Chain Block of a set of words. In this section, we demonstrate that, in fact, all finitely generated abelian groups are realizable in this manner. This result follows from Theorem \ref{thm:words_extend_cubical}, where we show that Insertion Chain Blocks generalize cubical complexes, as any cubical complex can be represented as the insertion chain block complex of a set of words in $\{a,b\}^*$. 
	
	We offer now a brief review of Cubical Complexes but direct the unfamiliar reader to additional sources (such as \cite{hatcher2002algebraic, mischaikow2004computational, savvidou2010face}) for a more in-depth treatment.

    Recall that an \textbf{elementary interval} is a closed interval $I\subset\R$ of the form $I=[k,k+1]$ or $I=[k,k]=\{k\}$, for some integer $k\in\Z$. An elementary interval $I=\{k\}$ is said to be \textbf{degenerate}, while an interval $I=[k,k+1]$ is \textbf{nondegenerate}. An elementary cube $Q$ is a finite product of elementary intervals, i.e. $$Q=I_1\times I_2\times\dots\times I_d\subset\R^d,$$ where each $I_i$ is an elementary interval. The dimension of $Q$, $\dim(Q)$ is the number of nondegenerate intervals among $I_1,\cdots, I_d$.
	
	\begin{defi}
		A \textbf{cubical complex} $K$ in $\R^d$ is a collection of elementary cubes in $\R^d$ satisfying:
		\begin{enumerate}
			\item If $Q_1, Q_2\subset\R^d$ are elementary cubes, $Q_1\subset Q_2$ and $Q_2\in K$, then $Q_1\in K$.
			\item If $Q_1, Q_2\in K$ and $Q=Q_1\cap Q_2\neq\emptyset$, then $Q$ is an elementary cube as well. 
		\end{enumerate}
		The \textbf{geometric realization} of $K$, which we will also denote by $K$, is the union of all cubes in $K$ as a subset of $\R^d$. 
	\end{defi}

	Given a cubical complex $K$, its \textbf{cubical chain complex} $\C_*(K)$ is the chain complex where each $C_k(K)$ is a free module generated by the elementary cubes in $K$ of dimension $k$, and the boundary operator is defined on elementary cubes as: 
	\begin{enumerate}
	    \item If $Q=I_1$ is an elementary interval, then $$\partial_k Q=\begin{cases}
        0&\text{ if } Q=\{l\}, \\
        \{l+1\}-\{l\}&\text{ if } Q=[l,l+1].
    \end{cases}$$
			
	\item If $Q=I_1\times I_2\times\dots\times I_d$ for $d\geq 1$. Let $P=I_2\times \dots\times I_d$, so $Q=I_1\times P$.  Then $$\partial_k Q=(\partial_k I_1)\times P + (-1)^{\dim( I_1)}I_1\times\partial_k P.$$
	\end{enumerate}

	
	
	
	
		

	Given a cubical complex $K$, recall that its (cubical) barycentric subdivision $sd(K)$, is the cubical complex with vertices the barycenters of the faces of $K$. The faces of $sd(K)$ are identified by the closed intervals $[F,G]=\{C\in\calF(K): F\leq C\leq G\}$ in the poset $\calF(K)\setminus\{\emptyset\}$ of nonempty faces of $K$. In particular, maximal faces of $sd(K)$ correspond to maximal intervals $[F_0,F_d]$ where $F_0$ is a vertex of a maximal cube $F_d$. It is a well known fact that the geometric realization of $K$ and $sd(K)$ are homeomorphic (see \cite{mischaikow2004computational,savvidou2010face}), and so, in particular, their homology groups are isomorphic. 
		
	We first prove the next lemma, which establishes that given a cubical complex, considering its barycentric subdivision produces a homeomorphic cubical complex that is completely determined by its vertices. 
	\begin{lemma}\label{lemma:vertices_cubical}
		Let $K$ be a cubical complex in $\R^d$, and consider $K'=2sd(K)$ be its cubical barycentric subdivision stretched by a factor of 2. If $C$ is a unit cube in $\R^d$ such that $V(C)\subset V(K')$, then $C\in K'$. 
	\end{lemma}
	
	\begin{proof}
		Recall that the vertices of $sd(K)$ are the barycenters of each cube in $K$. In particular, if  $C=I_1\times I_2\times\dots\times I_d\subset K$ with  $I_i=\{m_i\}$ or $I_i=[m_i,m_i+1]$ for some $m_i\in \Z$, then its corresponding vertex in $K'=2sd(K)$ is given by $v_C=2bar(C)=\left(v_1,\dots, v_d\right)$ where $$v_i=\begin{cases}
			2m_i &\text{ if } I_i=\{m_i\}\\
			2m_i+1&\text{ if } I_i=[m_i,m_i+1]
		\end{cases}$$
		Thus, the coordinates of $v_C$ are even if the corresponding elementary interval is 0-dimensional, and odd if it's 1-dimensional. Note that the converse is also true, if $v=(v_1,v_2,\dots, v_d)\in V(K')$, then there is a cube $C_v\in K$ such that $v=2bar(C_v)$, and $C_v=I_1\times\dots\times I_d$ is given by 
		$$I_i=\begin{cases}
			\left\{\dfrac{v_i}{2}\right\} & \text{ if } v_i\text{ is even}\\
			\left[\dfrac{v_i-1}{2},\dfrac{v_i+1}{2}\right] &\text{ if }v_i\text{ is odd}.
		\end{cases} $$
		
		Let now $C\subset\R^d$ be a unit cube such that $V(C)\subset V(K')$. Thus $C=I_1\times \dots\times I_d$, and we partition the indices into $E_0\sqcup O_0 \sqcup E_1\sqcup O_1=[n]$, according to the parity of the endpoints of the intervals $I_i$ as follows:  
		\begin{eqnarray*}
			I_i=\begin{cases}
				\{2m_i\} & \text{ if }i\in E_0\\
				\{2m_i+1\} & \text{ if } i\in O_0\\
				[2m_i,2m_i+1] & \text{ if } i\in E_1\\
				[2m_i-1,2m_i] & \text{ if }\in O_1 
			\end{cases}
		\end{eqnarray*}
		
		Let $u=(u_1,\dots,u_d)$ and $w=(w_1,\dots, w_d)$ be the vertices of $V(C)$ with the least and most odd coordinates, respectively; that is,
		
		\begin{eqnarray*}
			u_i=\begin{cases}
				2m_i&\text{ if }i\in E_0\cup E_1\cup O_1\\
				2m_i+1&\text{ if }i\in O_0
			\end{cases}\text{ and }%
			w_i=\begin{cases}
				2m_i&\text{ if }i\in E_0\\
				2m_i+1&\text{ if }i\in O_0\cup E_1\\
				2m_i-1&\text{ if }i\in O_1	
			\end{cases},
		\end{eqnarray*}
		and let $C_u, C_w\in K$ be their corresponding cubes in $K$. We claim that
		$v\in V(C)$, if and only if,  $C_u\subset C_v\subset C_w$ as sub-cubes, and so, the cube $C$ corresponds to the interval $\left[C_u, C_w\right]$ in the Poset of $K$, which is a cube in $sd(K)$, and thus $C\in K'$ as desired. 
		
		Note that if $v\in V(C)$ and $C_v=I_1\times\dots\times I_d$, then $I_i=\{m_i\}$ if $v_i$ is even, $I_i=[m_i, m_i+1]$ if $v_i=2m_i+1$, and $I_i=[m_i-1,m_i]$ if $v_i=2m_i-1$. In either case, $I_i({C_u})\subset I_i\subset I_i({C_w})$, which means $C_u\subset C_v\subset C_w$. On the other hand, if $C=I_1\times\dots\times I_d\in K$ such that $C_u\subset C\subset C_w$, and $v=v_C$, we get: $\{m_i\}= I_i$ if $i\in E_0$; $I_i=[m_i,m_i+1]$, if $i\in O_0$; $\{m_i\}\subset I_i\subset [m_i,m_i+1]$ if $i\in E_1$; and $\{m_i\}\subset I_i\subset [m_i-1,m_i] $, if $i\in O_1$. Therefore, $v_i=2m_i$ if $i\in E_0$; $v_i=2m_i+1$ if $i\in O_0$; $v_i=2m_i$ or $2m_i+1$ if $i\in E_0$; and $v_i=2m_i$ or $2m_i-1$ if $i\in O_1$. Meaning $v\in V(C)$, thus finishing the proof. 
	\end{proof}
	
	The following theorem establishes that, in a sense, Insertion Chain Complexes are an extension of Cubical complexes. 
	\begin{theorem}\label{thm:words_extend_cubical}
		Let $K$ be a finite cubical complex in $\R^d$. Then, there exist a set of words $W\subset\{a,b\}^*$, such that the chain complex of $K$ and the insertion chain complex $\ChainBlock{W}$ have isomorphic homology groups. 
	\end{theorem}
	
	\begin{proof}
		Let $K'=2sd(K)$, so that $K'$ is homeomorphic to $K$ and, in particular, their homology groups are isomorphic. We may assume that all vertices of $K'$ lie in $\N^d$, after a finite translation if necessary. Define the map $\Psi:\N^d\to\{a,b\}^*$ as: $$\Psi(m_1,m_2,\dots, m_d)=ab^{m_1}ab^{m_2}\dots ab^{m_d},$$
		and let $W=\Psi(V(K'))$. We will prove $\C_{*}(K')\cong \ChainBlock{W}$ are chain isomorphic, so the result will immediately follow. 
		
		First, we extend $\Psi$ to elementary cubes as follows:
		
		$$C=I_1\times I_2\times\dots\times I_d\xmapsto{\Psi} a \xi_1 a\xi_2\dots a\xi_d, $$
		where $$\xi_i=\begin{cases}
			b^m &\text{if } I_i=\{m\}\\
			b^m(1,b) &\text{if } I_i=[m_,m+1].
		\end{cases}$$
		
		Note $\Psi(C)$ is indeed a valid block, since between every pair of edges $(1,b)$ we always have a symbol $a$, and it is already in canonical form, with all copies of $b$ in front of their respective edge $(1,b)$. Since $\Psi(C)$ has the same number of edges $b^{m_i}(1,b)$ as nondegenerate intervals $[m_i,m_i+1]$  in $C$, clearly $\dim(C)=\dim(\Psi(C))$. Moreover, it is straightforward to check $V(\Psi(C))=\Psi(V(C))$.
		
		Showing $\Psi$ is a chain map, that is, $\partial\circ\Psi=\Psi\circ\partial$ is straightforward: Take $C=I_1\times\dots\times I_d$ be an elementary cube, and let $i_1, \dots, i_k\in[d]$ be the indices of the one dimensional intervals $I_{i_j}=[m_{i_j},m_{i_j}+1]$. Thus $\Psi(C)=a\xi_1\dots a\xi_d$, where $\xi_{i_j}=b^{m_{i_j}}(1,b)$ and $\xi_i=b^{m_i}$ for the other indices. Then, we can rewrite $\Psi(C)=x_0(1,b)x_1(1,b)\dots x_{k-1}(1,b)x_{k}$, where $x_j\in\{a,b\}^*$. Note we get

  \begin{align*}
      \bar{F}_j(\Psi(C))&=x_0(1,b)\cdots x_j\ b\ x_{j+1}\cdots (1,b) x_k = a\xi_1\dots a\xi_{i_j-1}\ ab^{m_{i_j}+1}\ a\xi_{i_j+1}\dots a\xi_d\\
      &=\Psi\left(I_1\times\dots I_{i_j-1}\times \{m_{i_j}+1\}\times I_{i_j+1}\times\cdots\times I_d\right)\\
      \underline{F}_j(\Psi(C))&=x_0(1,b)\cdots x_j\ \ x_{j+1}\cdots (1,b) x_k = a\xi_1\dots a\xi_{i_j-1}\ ab^{m_{i_j}}\ a\xi_{i_j+1}\dots a\xi_d\\
      &=\Psi\left(I_1\times\dots I_{i_j-1}\times \{m_{i_j}\}\times I_{i_j+1}\times\cdots\times I_d\right)
  \end{align*}%
Therefore %
    \begin{align*}
			\partial\Psi(C)&=\sum_{j=1}^k (-1)^{j+1} \bar{F}_i(\Psi(C))-\underline{F}_i(\Psi(C))=\\
			&=\sum_{j=1}^k(-1)^{j-1} \left(\Psi\left(I_1\!\!\times\!\!\cdots\!\!\times\!\! I_{i_j-1}\!\!\times\!\!\{m_{i_j}+1\}\!\!\times\!\! I_{i_j+1}\!\times\!\cdots\!\!\times\!\! I_d\right)\right.\\
            &\left.-\Psi\left(I_1\!\!\times\!\!\cdots\!\!\times\!\! I_{i_j-1}\!\!\times\!\!\{m_{i_j}\}\!\!\times\!\! I_{i_j+1}\!\times\!\cdots\!\!\times\!\! I_d\right)\right)=\\
			&=\Psi\left( \sum_{j=1}^k(-1)^{j-1}I_1\times\cdots I_{i_j-1}\times\left[\{m_{i_j}+1\}-\{m_{i_j}\}\right]\times I_{i_j+1}\cdots\times I_d \right)=\\
			&=\Psi(\partial C). 
		\end{align*}
		
    To get isomorphisms at each dimension $\C_k(K')\xrightarrow{\Psi}\ChainBlockSub{k}{W}$, we need to show $\Psi$ is a bijection between $k$-dimensional elementary cubes in $K'$ and $k$-dimensional blocks in $\ChainBlockSub{k}{W}$, for each $k$. 
		
		For the injectivity, suppose $\Psi(C_1)=\Psi(C_2)$, with $C_j=I_1(C_j)\times\cdots\times I_d(C_j)$, for $j=1,2$. Then $\Psi(C_j)=a\xi^j_1\dots a\xi^j_d$ as above, and these two blocks are equivalent. By Theorem~\ref{thm:unique_canonical_form}, since they are already in canonical form, they are exactly the same. Then $\xi^1_i=\xi^2_i$ for each $i$, and hence $I_i(C_1)=I_i(C_2)$, so $C_1=C_2$. 
		
		For surjectivity, suppose $\sigma$ is a valid $k$-block in $\ChainBlockSub{k}{W}$. Then $V(\sigma)\subset W$, and, by its definition, all its vertices have exactly $d$ symbols $a$. Suppose $$\sigma=x_0(1,a_1)x_1\dots (1,a_k)x_k, $$ then it must be that all $a_i=b$, for $i=1, \dots, k$, otherwise its minimal and maximal word would have a different amount of symbols $a$.  Since $\sigma$ is valid, each $x_i$ for $i=1, \dots, k-1$ must have at least one symbol $a$, at the same time, $x_0$ must start with $a$ because all vertices of $\sigma$ do. Hence, $\sigma=a\xi_1a\xi_2\dots a\xi_d$, where each $\xi_i=b^{m_i}$ or $b^{m_i}(1,b)$, for some $m_i\in \N$. Consider then the elementary cube $C=I_1\times\dots\times I_d\subset\R^d$, where $I_i=\{m_i\}$ or $[m_i,m_i+1]$, correspondingly with $\xi_i$. Clearly $\Psi(V(C))=V(\sigma)\subset W$, and since $\Psi$ is one-to-one, it must be $V(C)\subset K'$. Then, by Lemma~\ref{lemma:vertices_cubical}, $C\in K'$. And $V(\Psi(C))=V(\sigma)$, so by Theorem \ref{thm:unique_block_by_vertices}, it must be $\Psi(C)=\sigma$ as desired. 
	\end{proof}

 	We now prove the desired result, that any finitely generated abelian group can be realized as the homology of a word block complex. 
	\begin{theorem}\label{thm:realizable_homology}
		Let $A$ be a finitely generated abelian group. For any integer $k>0$, there exists a set of words $W\subset\{a,b\}^*$, such that the $k$-th homology of its Chain Block Complex is $A$, that is $$H_k\left(\ChainBlock{W}\right)=A.$$
	\end{theorem}
	
	\begin{proof}
		It is a standard result that for every finitely generated abelian group $A$, we can construct a finite 2-dimensional simplicial complex $K$, such that $H_i(K)=\begin{cases}
		    A & \text{if } i=1\\
            0 &\text{if }  i\neq 1
		\end{cases},$ (see \cite{hatcher2002algebraic}). For $k>1$, we use the fact that the \textit{suspension} of a simplicial complex, which is again a simplicial complex, \textit{moves} the homology up, that is $H_i(K)=   H_{i+1}(\Sigma K)	$, where $\Sigma K$ is the suspension (see \cite{hatcher2002algebraic,kozlov2007combinatorial}). Thus, applying suspension $(k-1)$ times we get 
        \begin{align*}
            H_k(\Sigma^{k-1}K)&=H_1(K)=A,\text{and}\\
            H_i(\Sigma^{k-1}K)&=H_{i-k+1}(K)=0, \text{ for } i\neq k. 
        \end{align*}
		where $\Sigma^{k-1}K$ is a finite simplicial complex. 
		It is a well known fact that for any finite simplicial complex $\mathcal{S}$, there exists a finite cubical complex which is homeomorphic to $\mathcal{S}$, so in particular, it has the same homology groups (see, for example, Appendix A of \cite{davis2012geometry} , or the discussion in Section 2 of \cite{babson1997neighborly}). 
		Thus, let $\mathcal{C}$ be a finite cubical complex that is homeomorphic to $\Sigma^{k-1}K$. By Theorem \ref{thm:words_extend_cubical}, there is a set of words $W\subset\{a,b\}^*$, such that $$H_k(\ChainBlock{W})=H_k(\mathcal{C})=H_k(\Sigma^{k-1}K)=H_1(K)=A,$$ 
		which finishes the proof. 
	\end{proof}
	
	\subsection{Minimal Homological Spheres}
	Theorem \ref{thm:realizable_homology} provides an straightforward way of finding sets of words that realize the homology $H_k(\ChainBlock{W})=\Z$ and $H_i(\ChainBlock{W})=0$ for $i\neq k, 0$. However, it is an interesting question to find the cardinality of the smallest $W$ that satisfies this property. In particular, for $k=1$, this becomes finding the smallest $W$ that produces a block complex homemorphic to the circle $\mathbb{S}^1$. 
	Using the construction from the theorem, we start with the cubical complex $$\mathcal{C}=\{\{0\}\times[0,1], \{1\}\times[0,1],[0,1]\times\{0\},[0,1]\times\{0\}\}$$
	Then, Theorem \ref{thm:words_extend_cubical} produces the set $W=\{1,a,a^2,b,b^2,ab^2, a^2b, a^2b^2\}$, consisting of 8 words. However, the set of 4 words $W=\{a,ab,ba,b\}$, also produces a 1-cycle using less words, and this is in fact the minimum size of such a cycle. 
	
	\begin{defi}\label{def:min_cycle}
		Given a topological space $X$, let $\mu(X)$ denote the size of the smallest set of words $W,$ such that the homology of its insertion chain complex is isomorphic to the homology of $X$, that is, such that $H_k(\ChainBlock{W})\simeq H_k(X),$ for all $k$.  	\end{defi}
	In particular, we are interested in finding the minimal sets of words that produce nontrivial homology in each dimension, that is, finding $\mu(\S^k)$ for the $k$-th dimensional spheres.  Note this is the same as 
    $$\mu(\S^k)=\min\Bigg\lbrace |W|: H_i(\ChainBlock{W})=\begin{cases}
			\Z & i=0,k\\
			0 & i\neq 0,k
		\end{cases}\, \Bigg\rbrace.$$
	
	\newpage
    The following theorem encompasses our partial results to this question:
	\begin{theorem}\label{thm:min_cycles_d}
		The following bounds hold:
		\begin{enumerate}
			\item $\mu(\S^0)=1$
			\item $\mu(\S^1)=4$
			\item $\mu(\S^2)=8$
			\item For $k\geq 3$, $\mu(\S^k)\leq 3^{k+1}-1$. 
		\end{enumerate}
	\end{theorem}
	
	\begin{remark}
		While Definition \ref{def:min_cycle}  concerns the smallest set of words in an alphabet $\Sigma$ containing any number of symbols, the proof to Theorem \ref{thm:min_cycles_d} produces upper bounds using the alphabet $\{a,b\}$ consisting of only two symbols. Thus,  in dimensions zero, one, and two, the minimum can be achieved with just two symbols. It remains an open question whether the minimum can also be attained with only two symbols for higher dimensions. 
	\end{remark}
	
	\begin{proof}\leavevmode
		\begin{enumerate}
			\item Trivial taking $W=\{1\}$.
			\item Taking $W=\{a,ab,ba,b\}$ produces an upper bound. Because the 1-skeleton of $\ChainBlock{W}$ is the graph $\Gins{W}$, which is bipartite by Lemma \ref{lemma:g_bipartite}, all cycles must be even, and thus a 3-cycle is impossible. 
			
			\item Taking  the set of 8 words $W=\{a,a^2,b,b^2,ab,ba,bab,aba\}$, produces the block complex in Figure \ref{fig:min_2_cycle}, which has the desired homology.  
			
			We prove here that this cannot be achieved with any set of 5 words, i.e. $\mu(\S^2)\geq 6$. The proof that it is still not possible with 6 or 7 words was obtained with the help of a computer program, described in greater detail in Appendix 2. Our program examines all possible 1-skeletons of Block complexes using 6 or 7 vertices and considers all potential ways of attaching 2-faces to achieve the desired homology. We then verify that all these configurations are impossible to realize with words, as they all include sub-patterns that cannot be realized.
			
			Suppose $W=\{v_0,v_1,v_2,v_3,v_4\}$ produces the desired homology and that  $V=\{v_0,v_1,v_2,v_3\}$ are the vertices of a square $\sigma=x(1,a)y(1,b)z$ so $v_0=xyz$, $v_1=xayz$, $v_2=xybz$ and $v_3=xaybz$. In order to have a cycle, each edge of $\Cins{W}$ must appear in at least two squares, and because two vertices can only be connected if they differ by exactly one symbol, its straightforward to see that we must have $v_0$ connect to $v_4$ and $v_4$ connect to $v_3$, forming the squares $\{v_0,v_1,v_4,v_3\}$ and $\{v_0,v_2,v_4,v_3\}$. But then, the edge from $v_0$ to $v_4$ appears parallel and in the same square as both the edges $v_1v_3$ and $v_2v_3$, so we must have $a=b$. Hence, all three squares share initial vertex, final vertex, and all insertions are of the same symbol $a$. However, from lemma \ref{lemma:squares_sharing_edges}, we cannot have two different squares sharing two edges if $a=b$, as in both cases from the lemma, the squares would be invalid. 
			
			\begin{figure}
				\centering
				\includegraphics[width=0.8\linewidth]{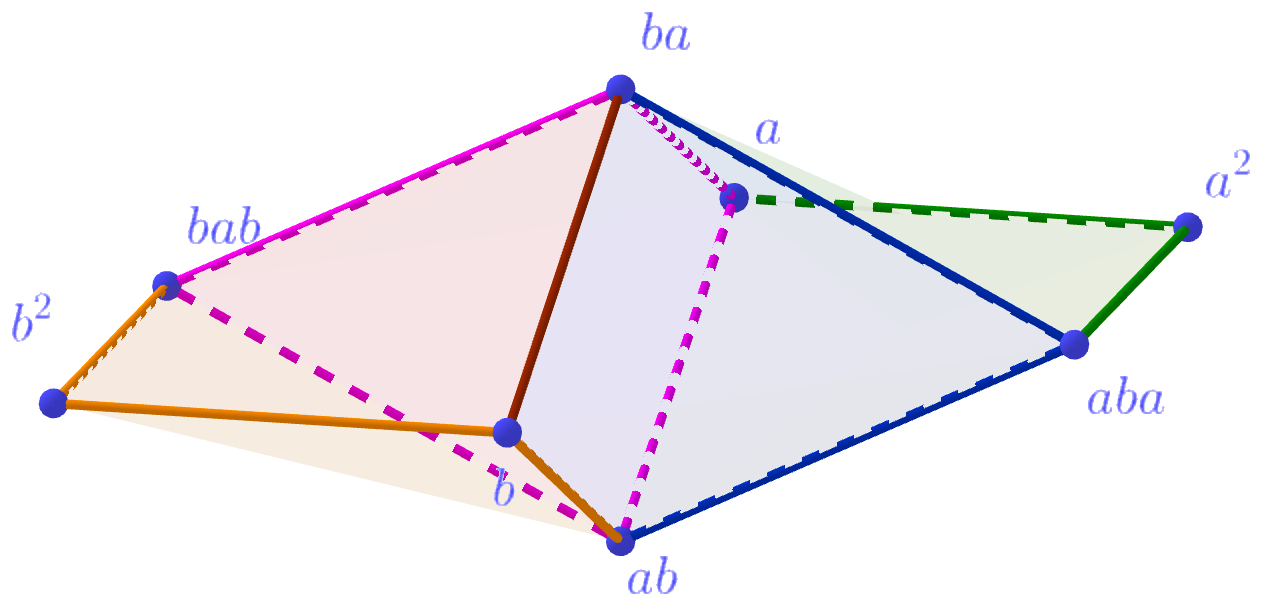}
				\caption{Minimal 2-cycle}
				\label{fig:min_2_cycle}
			\end{figure}
			
			\item We may produce the desired homology by starting with a minimal cubical complex that produces the desired homology. In particular, for any $k$, we may take $\mathcal{C}$ to be the boundary of the $(k+1)$-dimensional cube, which is homeomorphic to $\mathbb{S}^k$. Then, Theorem \ref{thm:words_extend_cubical} produces a set of words $W_0$, such that 
			$$H_i\left(\ChainBlock{W_0}\right)=H_i(\mathcal{C})=H_i(\mathbb{S}^k), \text{ for all } i.$$
			
			So $\mu(\S^k)\leq |W_0|$. Note that $W_0$, as constructed in the theorem, will have the same number of vertices as the barycentric subdivision of $\mathcal{C}$, which, by definition, has one vertex for each face in $\mathcal{C}$. It is well known that the $r$-th dimensional cube has a total of $3^r$ faces, so its boundary has $3^r-1$, and so with $r=k+1$ we obtain the result. \qedhere
		\end{enumerate}
	\end{proof}

    \subsection{Vanishing Homology}
	In our classification of the smaller 1-dimensional cycles, we observe that  types 1, 2 and 3, show two incompatible paths from the minimal word to the maximal word among the vertices, which produces the cycle. If, we had all \textit{missing} words in between the minimal and maximal, we would expect that the cycle would be \textit{filled up}. For example, consider the cycle of type 1 with $t=1$, that is $W=\{ab,aba,aab,abab\}$. If we \textit{complete} the set of words with the missing words between $ab$ and $abab$, we get $W'=\{ab,aba,aab,abab,abb,bab\}$, whose corresponding block complex has no cycles. We would expect this to be the same also for higher dimensional cycles,  and for any pair of initial and final words. While we will prove that this homology indeed vanishes in all dimensions, we are only able to establish it when the  initial word \textit{embeds uniquely} into the maximal word, as we define below. We still consider this result strong, since its contrapositive implies that any homology in block complexes with one minimal and maximal face, is produced by either \textit{missing} paths, or subwords that don't embed uniquely. 
	
	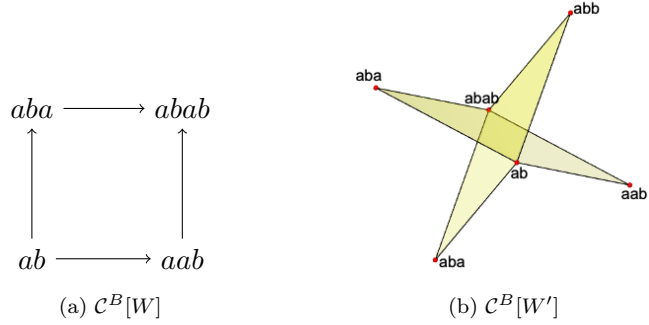
\begin{figure}
		\centering
		\begin{subfigure}[b]{0.20\textwidth}
			{}
		\end{subfigure}
		~
		\begin{subfigure}[b]{0.3\textwidth}
			\centering
			\begin{tikzpicture}
				\node at (0,0) (A) {$ab$} ; 
				\node at (2,0) (B) {$aab$} ; 
				\node at (0,2) (D) {$aba$} ; 
				\node at (2,2) (C) {$abab$} ; 
				
				\path[->] (A) edge node[auto=left] {} (B);
				\path[->] (A) edge node[auto=left] {} (D);
				\path[->] (B) edge node[auto=left] {} (C);
				\path[->] (D) edge node[auto=left] {} (C);
			\end{tikzpicture}
			\caption{$\Cins{W}$}
		\end{subfigure}%
		~ 
		\begin{subfigure}[b]{0.3\textwidth}
			\centering
			\includegraphics[height=1.5in]{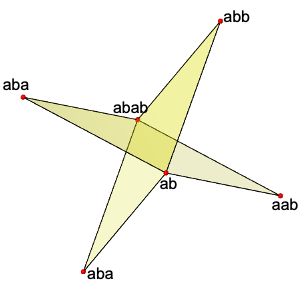}
			\caption{$\Cins{W'}$}
		\end{subfigure}
		~
		\begin{subfigure}{0.20\textwidth}
			{}
		\end{subfigure}
		\caption{Killing a cycle by including missing words in the paths $ab\to abab$. }
	\end{figure}
	
	\begin{defi}
		Let $w_m, w_M$ be words in $\Sigma$ and $w_m\leq w_M$ be a subword of $w_M$.

		\begin{enumerate}
			\item Write $w_M=a_1^{r_1}a_2^{r_2}\cdots a_m^{r_m}$ where $a_i\in\Sigma$, $a_i\neq a_{i+1}$, and $1\leq r_i$. We say that $w_m$ \textbf{embeds uniquely into} $w_M$, if there is only a unique sequence $q_1, \dots, q_m$, with $0\leq q_i\leq r_i$ such that $w_m=a_1^{q_1}a_2^{q_2}\cdots a_m^{q_m}$. 
			
			\item Denote by $\sub{w_m}{w_M}$ the set of all possible subwords of $w_M$  that contain $w_m$, that is  $$\sub{w_m}{w_M}=\{w\in\Sigma^*: w_m\leq w\leq w_M\}.$$
		\end{enumerate}
	\end{defi}
	
	\begin{examples}\leavevmode
		\begin{itemize}
			\item For any word $w\in\Sigma^*$, $1$ embeds uniquely in $w$, since if $w=a_1^{r_1}\dots a_m^{r_m}$, the only  sequence of powers to get 1 is the constant sequence $q_i=0$.  Similarly, $w$ always embeds uniquely into $w$, by choosing $q_i=r_i$. 
			
			\item Let $w_m=abc$, $w=abca$ and $w_M=a^2bcaba$. Note $w_m\leq w\leq w_M$, $w_m$ embeds uniquely into $w_M$, and also into $w$, while $w$ doesn't embeds uniquely into $w_M$. To see this, consider $w_M=a_1^{r_1}\dots a_6^{r_6}$ where $r_1=2$, and $r_i=1$ for $2\leq i\leq 6$. Then the only way to embed $w_m=a_1^{q_1}\dots a_6^{q_6}$ is by choosing the powers $q_1=q_2=q_3=1$ and $q_4=q_5=q_6=0$. Note this is also the only way to embed $w_m$ into $w$. However, $w$ can be embedded into $w_m$ in two ways: by choosing the powers $t_1=t_2=t_3=t_4=1$ and $t_5=t_6=0$, or by choosing $t_1=t_2=t_3=1$, $t_4=t_5=0$ and $t_6=1$. 
		\end{itemize}
	\end{examples}

    We establish the following lemma about uniquely embedded words. 

    \begin{lemma}\label{lemma:uniquely_embedding}
    Let $w_m\leq w_M$ be finite words such that $w_m$ embeds uniquely into $w_M$. If 
    $w\in\sub{w_m}{w_M}$ with $w\neq w_m$, then $w_m$ also embeds uniquely into $w$. 
    \end{lemma}

    \begin{proof}
        Suppose $w_M=a_1^{r_1}\cdots a_m^{r_m}$ and $w_m=a_1^{q_1}\cdots a_m^{q_m}$, where $a_i\in\Sigma$, $a_i\neq a_{i+1}$ for all $i$, and $0\leq q_i\leq r_i$.  Since $w\leq w_M$, there exist integers $0\leq t_i\leq r_i$, such that $w=a_1^{t_1}\cdots a_m^{t_m}$. Moreover, since $w_m\leq w$, there must exist integers $0\leq q'_i\leq t_i$ such that $w_m=a_1^{q_1'}\cdots a_m^{q_m'}$, with $0\leq q_i'\leq t_i$, but since $w_m$ embeds uniquely in $w_M$, and the integers $q_i'$ also provide such an embedding, so $q_i'=q_i$ for all $i$, and it is unique. Therefore, $w_m$ also embeds uniquely into $w$.  
    \end{proof}

    Our main theorem shows that the homology of the Insertion Chain Complex of $\sub{w_m}{w_M}$ vanishes whenever $w_m$ embeds uniquely into $w_M$. We conjecture that this result continues to hold even without the unique embedding requirement.
	\begin{theorem}\label{thm:null_homology_one_dim}
		Let $\Sigma$ be a finite alphabet and $w_M,w_m\in\Sigma^*$ words such that $w_m$ embeds uniquely into $w_M$.  Then $\ChainBlock{\sub{w_m}{w_M}}$ is connected and $$H_i\left(\ChainBlock{\sub{w_m}{w_M}}\right)=0\ \text{ for all }i\geq 1$$
	\end{theorem}

	\begin{proof}
		Denote $W=\sub{w_m}{w_M}$, and $\C=\ChainBlock{W}$. Suppose $w_M=a_1^{r_1}\cdots a_m^{r_m}$ and $w_m=a_1^{q_1}\cdots a_m^{q_m}$, where $a_i\in\Sigma$, $a_i\neq a_{i+1}$ for all $i$, and $0\leq q_i\leq r_i$. Let $w\in W$ be any vertex, and let $t_i$ be a sequence of powers such that $w=a_1^{t_1}\cdots a_m^{t_m}$, $q_i\leq t_i\leq r_i$. Let $N=\sum_{i=1}^m (r_i-t_i)$. Note that for each $j=0, 1, \dots, N$ we can define $w_j=a_1^{t_{1,j}}a_2^{t_{2,j}}\cdots a_m^{t_{m,j}}$ where $t_{i,0}=q_i$, $t_{i,N}=t_i$ for all $i$, and for each $0<j<M$ there is a unique index $k_j$, such that $t_{k_j,j}\neq t_{k_j,j-1}$, and instead $t_{k_j,j}=t_{k_j,j-1}+1$. Thus, $w_{j-1}$ and $w_j$ differ only by one insertion of the symbol $a_{k_j}$, so $w_{j-1}\to w_j$ is an edge in $\C$. Hence, we get a path $w_m=w_0\to w_1\to\cdots w_{N-1}\to w_N=w$ in $\C$. Therefore, $\C$ is connected.

		Fix $d\geq 1$, and let $\gamma$ be a $d$-dimensional cycle in $\C$, that is, $\gamma=\sum_i \xi_i \sigma_i$, for some integers $\xi_i\in \Z$, $\xi_i\neq 0$, and  distinct $d$-blocks $\sigma_i\in \C^{(d)}$. Denote by $L(\gamma)$ the length of the longest word among the vertices of $\gamma$, and, for any integer $\ell$, denote by $V_\ell(\gamma)$ the set of  vertices in $\gamma$ of length at least $\ell$. That is 
		\begin{align*}
			L(\gamma)&=\max\left\{ |w| : w\in V(\sigma_i), \text{ for some } \sigma_i\text{ in }\gamma\right\}\\
			V_\ell(\gamma) &= \left\{w: |w|\geq\ell \text{ and } w\in V(\sigma_i), \text{ for some } \sigma_i\text{ in }\gamma\right\}.
		\end{align*}
		
		We claim that for any $w\in V_{L}(\gamma)$ with $L=L(\gamma)$,  there exists a $d$-dimensional cycle $\gamma'$ that is homologous to $\gamma$, and such that $V_{L}(\gamma')=V_L(\gamma)\setminus\{w\}$. Applying this claim repeatedly, we will show that $\gamma$ is homologous to the trivial cycle. Since this can be done for any $d$-dimensional cycle, $H_d(\C)=0$ as desired. 
        
        Indeed, if $\gamma\neq 0$, write $V_L(\gamma)=\{w_1, \dots, w_k\}$, then there is a sequence of homologous $d$-dimensional cycles $$\gamma=\gamma_0 \sim \gamma_1\sim \cdots \sim\gamma_k=\gamma'$$ where $V_L(\gamma_i)=V_L(\gamma_{i-1})\setminus\{w_i\}$, and so $V_L(\gamma')=\emptyset$ making $L(\gamma')< L(\gamma)$. But then,  further repeating this process, there are homologous $d$-dimensional cycles $$\gamma\sim \gamma'\sim \gamma''\sim \dots$$ such that $L(\gamma)>L(\gamma')>L(\gamma'')>\dots$. Since $L(\cdot)$ is a non-negative integer, this process must eventually end. The process ends when we get the trivial (empty) cycle, as otherwise $V_{L(\gamma^{(k)})}(\gamma^{(k)})\neq \emptyset$ and the process could be continued. 

        Before proving the claim, we introduce the following notation. Given any chain $\gamma=\sum_i\xi_i\sigma_i$ in $\C$ and a vertex $w\in \cup_i V(\sigma_i)$, we define the star and link of $w$ in $\gamma$ as the following chains: 
        		\begin{align*}
	\star_\gamma(w)&=\sum_{\substack{i\\w\in V(\sigma_i)}}\xi_i\sigma_i & \link_\gamma(w)&=\partial \star_\gamma(w). 
		\end{align*}
        
    Readers familiar with simplicial complexes will notice similarities between our definitions of the star and link and the corresponding notions for simplices. However, there are important differences: in our setting, these notions are defined for vertices with respect to chains, rather than for arbitrary simplices in a complex. Furthermore, the outcome of our definitions is a chain, not a sub-complex. 
		
		We now proceed to prove the claim. Let $\gamma$ be a non-trivial $d$-dimensional cycle as above, and let $w=a_1^{t_1}\dots a_m^{t_m}$ be a vertex in $V_{L(\gamma)}(\gamma)$, where $a_i\in\Sigma$, $t_i\geq1$, and $a_i\neq a_{i+1}$ for all $i$.   Because $\gamma$ is a cycle we get 		
		\begin{align*}
			0&=\partial \gamma=\partial\left(  \sum_{\substack{i\\ w\in V(\sigma_i)}}\xi_i\sigma_i + \sum_{\substack{i\\ w\not\in V(\sigma_i)}}\xi_i\sigma_i \right) = \link_\gamma(w) + \sum_{\substack{i\\ w\not\in V(\sigma_i)}}\xi_i\partial\sigma_i
		\end{align*}
		
		Note that if $w\not\in\sigma_i$, then $w\not\in\tau$ for any $\tau\prec\sigma_i$, so $w$ is not a vertex of any block in $\partial\sigma_i$, and hence it is not a vertex of any block in the second summation. Since the whole sum equals zero, all blocks in $\link_\gamma(w)$ cancel out with the blocks in the sum $\displaystyle\sum_{ w\not\in V(\sigma_i)}\xi_i\partial\sigma_i$, so $w$ is not a vertex of any block in $\link_\gamma(w)$, that is,  $\star_{\link_\gamma(w)}(w)=0$. This also means that in $\star_\gamma(w)$, all facets that contain $w$ cancel each other out when computing the boundary.  In particular, $\star_\gamma(w)$ must have at least two $d$-blocks, so that their facets containing $w$ cancel out. 
			
		Since each $\sigma_i$ in $\star_\gamma(w)$ is a $d$-block with $w$ as its maximal word, there exists a set of $d$ indices $I_i\subset [m]$ such that $$\sigma_i=\lambda_1\lambda_2\dots \lambda_m,\quad \text{where }\lambda_j=\begin{cases}
			a_j^{t_j} & \text{ if } j\not\in I_i\\
			a_j^{t_j-1}(1,a_j)& \text{ if }j\in I_i
		\end{cases}. $$
		Moreover, because $a_i\neq a_{i+1}$ and $t_j\geq 1$, this is already in canonical form,  and by Theorem~\ref{thm:unique_canonical_form} each set of indices is unique.
		
		Let $J=\bigcup_i I_i$, be the union of the indices corresponding to the blocks in $\star_\gamma(w)$. Note $|J|\geq d+1$, since there are at least two different blocks $\sigma_1$, $\sigma_2$, and their corresponding sets $I_1, I_2$ are distinct. For any $I\subset J$, define  $$\sigma_I(w)=\lambda_1\cdots \lambda_m, \quad \text{where } \lambda_j=\begin{cases}
			a_j^{t_j} & \text{ if } j\not\in I\\
			a_j^{t_j-1}(1,a_j)& \text{ if }j\in I
		\end{cases}. $$
		In particular, $\sigma_{I_i}(w)=\sigma_i$, and all of these are also already in canonical form. 

         Let $\C'$ be the chain complex of all the blocks $\sigma_I(w)$ and their faces, that is, $\sigma$ is a block in $\C'$ if and only if there is a set of indices $I\subset J$, such that $\sigma$ is a face of $\sigma_I(w)$. Since we include all faces, $\partial$ can be restricted to $\C'$, and both $\star_\gamma(w)$ and $\link_\gamma(w)$ are chains in this complex $\C'$. Moreover, we get $\C'\subset\C$ as a subcomplex, since the vertices of all the blocks $\sigma_I(w)$ are contained in $\sub{w_m}{w}\subset V(\C)$.    Indeed, if $v\in\sigma_I(w)$, then $v$ has the form $v=a_1^{s_1}a_2^{s_2}\dots a_m^{s_m}$, where $s_j\in \{t_j, t_j-1\}$ if $j\in I$ and $s_j=t_j$ otherwise. Clearly, $v\leq w$. Furthermore, for each $j\in J$, there exists an $i_0$ such that $j\in I_{i_0}$, note the vertex $v_j=a_1^{t_1}\dots a_{j-1}^{t_{j-1}}a_j^{t_j-1}a_{j+1}^{t_{j+1}}\dots a_m^{t_m}$ is in $\sigma_{i_0}$, and since $V(\sigma_{i_0})\subset\sub{w_m}{w_M}$, then $w_m\leq v_j\leq w_M$. By Lemma~\ref{lemma:uniquely_embedding}, $w_m$ embeds uniquely into $v_j$, so if $w_m=a_1^{q_1}\cdots a_m^{q_m}$, where $q_i\leq t_i$ for all $i$, furthermore we have  $q_j\leq t_{j}-1$. Hence, for either possible choice of $s_j$, we have $q_j\leq s_j$, for all $j\in I$ and thus $w_m\leq v$ as well, so $v\in \sub{w_m}{w}$ for all $v\in V(\sigma_I(w))$. 
        	
		Let $\Delta(J)$ be the simplex with vertex set $J$, where the order of the vertices is given by their numerical value. Define the functor $\Psi: \C'\to \C(\Delta(J))$ onto the simplicial chain complex of $\Delta (J)$, as follows:
		
		Let $\sigma\in \C'$ be an $r$-block, then there must exist a set of indices $I\subset J$ such that $\sigma\leq\sigma_I(w)$. If $w\in V(\sigma)$, it is its maximal word and $I$ can be chosen so that $\sigma=\sigma_I(w)$,  write $I=\{k_1, k_2, \dots, k_r\}$ with $k_i<k_{i+1}$ and define $\Psi(\sigma)$ to be the $(r-1)$-simplex given by  $$\Psi(\sigma)=\Psi(\sigma_I(w))=[k_1,k_2, \dots, k_r].$$ If instead $w\not\in V(\sigma)$, define $\Psi(\sigma)=0$. Note $\Psi(\sigma)$ is well defined since $\sigma_I(w)$ is in canonical form, and so it is unique due to Theorem \ref{thm:unique_canonical_form}. Also, $\Psi$ is onto, since all simplices in $\C(\Delta(J))$ correspond to subsets of $J$, and any subset $I\subset J$ has a preimage $\Psi(\sigma_I(w))=I$.  Moreover, it is indeed a chain map that commutes with $\partial$, as we can verify for $I=[k_1, \dots, k_r]$. Indeed, let its corresponding block be $\sigma_I(w)=x_0(1,a_{k_1})x_1(1,a_{k_2})\cdots(1,a_{k_r})x_r$,  let $\Omega\subset[k]$ be the set of indices such that $a_{k_{i-1}}\neq a_{k_{i+1}}$ or $x_{i-1}x_i\neq 1$, then
        
		\begin{align*}
			\Psi(\partial\sigma_I(w))&=\Psi\left( \sum_{i=1}^r (-1)^{i+1}x_0(1,a_{k_1})\cdots x_{i-1}a_{k_i}x_i\cdots (1,a_{k_r})x_r\right.\\
			&- \left.\sum_{i\in\Omega} x_0(1,a_{k_1})\cdots x_{i-1}x_i\cdots (1,a_{k_r})x_r\right)\\
			&= \sum_{i=1}^r (-1)^{i+1}\Psi\left(x_0(1,a_{k_1})\cdots x_{i-1}a_{k_i}x_i\cdots (1,a_{k_r})x_r\right) +0\\
			&=\sum_{i=1}^r (-1)^{i+1}[k_1,k_2,\dots, k_{i-1}, k_{i+1},\dots, k_r]\\
			&=\partial [k_1, \dots, k_r]=\partial\Psi\left(x_0(1,a_{k_1})\cdots (1,a_{k_r})x_r\right)
		\end{align*}
		where $\Psi$ applied to the sum indexed by $\Omega$ is zero, since none of those blocks include $w$ as a vertex (since we remove one instance of the symbol $a_{k_i})$. 
		
		Note $\Psi(\star_\gamma(w))$ is a $(d-1)$-dimensional cycle in $\C(\Delta(J))$, since $\partial \Psi\left( \star_\gamma(w)\right)=\Psi(\partial\star_\gamma(w))=\Psi(\link_\gamma(w))=0$, given that $\link_\gamma(w)$ doesn't have any block with $w$ as a vertex so they are mapped to zero.  Recall $\Delta(J)$ is contractible and of dimension $|J|-1\geq d$, so there exists a $d$-chain in $\C(\Delta(J))$, whose boundary is $\Psi(\star_\gamma(w)))$. Let one such $d$-chain be $\tau_\text{simp}=\sum_i \alpha_i c_{\text{simp},i}$, where each $c_{\text{simp},i}$ is a $d$-simplex in $\Delta(J)$, and take $J_i\subset J$, such that  $\Psi(\sigma_{J_i}(w))=c_{\text{simp},i}$. Let $\tau=\sum_i \alpha_i \sigma_{J_i}(w)$, then
		\begin{align*}
			\Psi\left( \partial \tau -\star_\gamma(w) \right) &= \partial\Psi(\tau)-\Psi(\star_\gamma(w))\\
			&= \partial \tau_{\text{simp}} -\Psi(\star_\gamma(w)) =0
		\end{align*}
		
		So in the chain $\partial \tau -\star_\gamma(w)$ there are no blocks with $w$ as a  vertex, and since it is contained in $\C'$, all its vertices must have a length strictly less than $|w|$. Then, setting $\gamma'=\gamma-\partial\tau$,  clearly $\gamma$ is homologous to $\gamma'$, and we observe $$\gamma'=\left( \gamma -\star_\gamma(w)\right)-\left(\partial\tau-\star_\gamma(w)\right),$$
		where both parenthesis are $d$-chains in $\C$. In the former, we have removed from $\gamma$ all blocks containing $w$, and the later is a combination of blocks whose vertices have length less than $|w|=L$. Therefore, $V_L(\gamma')=V_L(\gamma)\setminus\{w\}$ as desired. 
	\end{proof}

\section{Concluding Remarks}
Symbol insertions and deletions in words have been studied in formal language theory for decades~\cite{Lila-phd}, often focused on the generative power of these operations and the closure properties of the sets of words with respect to the operations~\cite{alhazov2022regulated,Ito-1997,Verlan2007}. In particular, it is known that single symbol contextual insertion can generate any recursively enumerable language, therefore it has universal computational power~\cite{RE-ins-2002}. 
In most of those studies, the interest has been on the insertion/deletion types that generate certain classes of sets of words, rather than the differences among the obtained words within the set. 

With the notion of the letter insertion complex and its topological properties, we introduced a topological measure of complexity for sets of words. This measure captures how ``dense'' a set is, in the sense that non-trivial homology typically indicates that there are sufficiently many ``missing'' words from the set, which would be complete if all possible symbol insertions were included. If one generates words up to a given length by randomly adding symbols, all words would eventually appear, and in that case all homology of the corresponding complex would vanish.  Thus, homology can also be regarded as a measure of ``non-randomness.''  

We showed that when the embedding of the ``minimal'' subword within the maximal subword is unique, and the set contains all words ``in between,'' then the corresponding complex has vanishing homology (Theorem~\ref{thm:null_homology_one_dim}). We conjecture that this result remains valid even without the unique embedding, but this question remains open.  

Unlike simplicial complexes, determining the size of minimal $k$-spheres in the word insertion complex is not straightforward. Theorem~\ref{thm:min_cycles_d} gives the precise size of the minimal set of words that produce non-zero homology for dimensions $0,1,$ and $2$, while for dimension $\ge 3$ we provide only a bound.  We believe this bound is not tight,and providing a better estimate—particularly determining whether minimal sets of words producing $k$-spheres can be achieved using only two symbols—remains an open problem.  

Finally, we note that another promising direction for future research is to combine our ideas with previously studied insertion/deletion systems, as in~\cite{alhazov2022regulated,Ito-1997,Verlan2007}.  
Since insertion/deletion systems are defined by a finite set of rules iteratively applied to a finite set of words, it would be interesting to investigate how such rules affect the topology of the insertion chain complex, that is, to understand what structural changes in the insertion word complexes are induced by these rules and how the corresponding homology varies.

\section*{Acknowledgments}
  This research was under auspices of the Southeast Center for Mathematics and Biology, an NSF-Simons Research Center for Mathematics of Complex Biological Systems, under National Science Foundation Grant No. DMS-1764406 and Simons Foundation Grant No. 594594 as well as W.M. Keck Foundation.

\bibliographystyle{abbrvnat}
\bibliography{References}

\appendix
\renewcommand{\thesection}{\Alph{section}} 
\makeatletter
\renewcommand\@seccntformat[1]{\appendixname\ \csname the#1\endcsname.\hspace{0.5em}}
\makeatother

\newpage
\section{Additional Proofs and Supporting Lemmas}\label{appendix}
\renewcommand{\thelemma}{A.\arabic{lemma}}
\renewcommand{\thetheorem}{A.\arabic{lemma}}

This appendix contains the remaining proofs from the paper, which were relocated either due to being straightforward or because they required the analysis of many similar cases. Additionally, we provide some auxiliary lemmas specifically needed for these proofs, as well as a table from Example \ref{ex:sub-blocks}.

\subsection{Blocks}
\begin{table}
    \centering
    \begin{tabular}{|c||cc|c|}\hline
     dim&  $I^+$&  $I^-$& $\sigma(I^+,I^-)$\\ \hline\hline
        &$\emptyset$&$\{1,2,3\}$&$1$\\
        &$\{1\}$&$\{2,3\}$&$a$\\
        &$\{2\}$&$\{1,3\}$&$b$\\
        $0$&$\{3\}$&$\{1,2\}$&$a$\\
        &$\{1,2\}$&$\{3\}$&$ab$\\
        &$\{1,3\}$&$\{2\}$&$aa$\\
        &$\{2,3\}$&$\{1\}$&$ba$\\
        &$\{1,2,3\}$&$\emptyset$&$aba$\\ \hline
        &$\emptyset$&$\{1,2\}$&$(1,a)$\\
        &$\emptyset$&$\{1,3\}$&$(1,b)$\\
        1&$\emptyset$&$\{2,3\}$&$(1,a)$\\
        &$\{1,2\}$&$\emptyset$&$ab(1,a)$\\
        &$\{1,3\}$&$\emptyset$&$a(1,b)a$\\
        &$\{2,3\}$&$\emptyset$&$(1,a)ba$\\ \hline
    \end{tabular}\hspace{2em}%
    \begin{tabular}{|c||cc|c|}\hline
     dim&  $I^+$&  $I^-$& $\sigma(I^+,I^-)$\\  \hline\hline 
     &$\{1\}$&$\{2\}$&$a(1,a)$\\
        &$\{1\}$&$\{3\}$&$a(1,b)$\\
        $1$&$\{2\}$&$\{1\}$&$b(1,a)$\\
        &$\{2\}$&$\{3\}$&$(1,a)b$\\
        &$\{3\}$&$\{1\}$&$(1,b)a$\\
        &$\{3\}$&$\{2\}$&$(1,a)a$\\ \hline
        &$\emptyset$&$\{1\}$&$(1,b)(1,a)$\\
        &$\emptyset$&$\{2\}$&$(1,a)(1,a)$\\
        $2$&$\emptyset$&$\{3\}$&$(1,a)(1,b)$\\
        &$\{1\}$&$\emptyset$&$a(1,b)(1,a)$\\
        &$\{2\}$&$\emptyset$&$(1,a)b(1,a)$\\
        &$\{3\}$&$\emptyset$&$(1,a)(1,b)a$\\ \hline
        $3$&$\emptyset$&$\emptyset$&$(1,a)(1,b)(1,a)$\\ \hline
    \end{tabular}
   \caption{All sub-blocks of $(1,a)(1,b)(1,a)$ }
    \label{tale:ex_subblocks}
\end{table}

\begin{lemma}[Restatement of Lemma \ref{lemma:charac_valid}]\hypertarget{proof:charac_valid}
  	Let $\sigma\in \E_m$ be an $m$-block. Then, the following are equivalent:
   \begin{enumerate}
   \item $\sigma$ is invalid. 
    \item There exists an $1\leq i<m$, such that $v_{\{i\}}(\sigma)= v_{\{i+1\}}(\sigma)$. 
    \item There exists an $1\leq i<m$, such that $\sigma(\{i\},\emptyset)=\sigma(\{i+1\},\emptyset)$.
    \item There are indices $1\leq j<i\leq m$, such that $v_{[m]\setminus\{i\}}(\sigma)= v_{[m]\setminus\{j\}}(\sigma)$.
   \end{enumerate}
  \end{lemma}

\begin{proof}[Proof of Lemma \ref{lemma:charac_valid}]
        Write $\sigma=x_0(1,a_1)\cdots (1,a_m)x_m$ in canonical form. 
        \begin{itemize}
            \item $(1\Rightarrow 2):$ Suppose $\sigma$ is invalid. Thus, there is an index $1\leq i< m$, such that $a_i=a_{i+1}$ and $x_i=1$. Hence $a_ix_i=a=x_ia_{i+1}$, so we get
            \begin{align}
			v_{\{i\}}(\sigma) = x_0x_1\cdots x_{i-1}&(a_ix_i)x_{i+1}\cdots x_m\label{eq:invalid_charac}\\
			 = x_0x_1\cdots x_{i-1}&(x_ia_{i+1})x_{i+1}\cdots x_m=v_{\{i+1\}}(\sigma).\nonumber
            \end{align}
            \item $(2\Rightarrow 3):$ Suppose there is an index $1 \leq i < m$ such that $v_{\{i\}}(\sigma) = v_{\{i+1\}}(\sigma)$. Then, we get again equation (\ref{eq:invalid_charac}), and thus, canceling out all the common terms on both sides, we obtain that $a_i x_i = x_i a_{i+1}$. By Lemma \ref{prop:simple_equation_words}, we must have $a_i=a_{i+1}=a$ and $x_i=a^t$ for some $t\geq 0$. Note then that $a_ix_i(1,a_{i+1})=a^{t+1}(1,a)=(1,a)a^{t+1}=(1,a_i)x_ia_{i+1}$. Thus, by Property \ref{prop:comparing_eq_blocks}, we obtain the equation
             \begin{align}
            \sigma(\{i\},\emptyset)=x_0(1,a_1)x_1\cdots (1,a_{i-1}) x_{i-1} &\,\large[a_ix_i(1,a_{i+1})\large]\,x_{i+1}(1,a_{i+2})\cdots (1,a_m)x_m \label{eq:invalid_charac2}\\
            =x_0(1,a_1)x_1\cdots (1,a_{i-1}) x_{i-1} &\,\large[(1,a_i)x_ia_{i+1}\large]\,x_{i+1}(1,a_{i+2})\cdots (1,a_m)x_m =\sigma(\{i+1\},\emptyset).\nonumber 
        \end{align}%
        \item $(3\Rightarrow 4):$ Suppose $\sigma(\{i\},\emptyset)=\sigma(\{i+1\},\emptyset)$, then, we get equation (\ref{eq:invalid_charac2}) which implies, by Property \ref{prop:comparing_eq_blocks},  $a_ix_i(1,a_{i+1})=(1,a_i)x_ia_{i+1}$, so $a_ix_i=x_ia_{i+1}$. Thus we obtain the equation 
        \begin{align}
            v_{[m]\setminus\{i\}}(\sigma)=x_0a_1x_1\cdots a_{i-1}x_{i-1}&(x_ia_{i+1})x_{i+1}a_{i+2}\cdots a_mx_m=\label{eq:invalid_charac3}\\
                =x_0a_1x_1\cdots a_{i-1}x_{i-1}&(a_ix_i)x_{i+1}a_{i+2}\cdots a_mx_m=v_{[m]\setminus\{i+1\}}(\sigma).\nonumber
                \vspace*{-0.5em}
        \end{align}%
        \item $(4\Rightarrow 1):$ Suppose there is an index $j<i$, such that $v_{[m]\setminus\{i\}}(\sigma)=v_{[m]\setminus\{j\}}(\sigma)$. Thus, we get the equation         
         \begin{align}
            v_{[m]\setminus\{i\}}(\sigma)=  x_0a_1x_1 \cdots &x_{j-1}a_jx_j \cdots x_{i-1}x_i \cdots x_{m-1}a_mx_m\\
            = x_0a_1x_1 \cdots &x_{j-1}x_j \cdots x_{i-1}a_ix_i \cdots x_{m-1}a_mx_m=v_{[m]\setminus\{j\}}(\sigma)\nonumber
        \end{align}%
    which simplifies to $a_jx_j \cdots x_{i-1} = x_j \cdots x_{i-1}a_i$. Which is $a_jx = xa_i$, where $x = x_ja_{j+1} \cdots a_{i-1}x_i$. By Corollary \ref{cor:simple_word_equations}, we must have $a_j = a_i = a$ and $x = a^t$ for some $t \geq 0$. However, this implies that $a_{j+1} = a_{j+2} = \cdots = a_i = a$, and all the words $x_{j+1} = \cdots = x_{i-1} = 1$, since $\sigma$ is in canonical form. Then, by definition, $\sigma$ is invalid. \qedhere
        \end{itemize}
	\end{proof}

   \begin{lemma}\label{lemma:faces_ordering}
		Let $\tau\in\Ev_k$, $\sigma\in\Ev_m$ and $\Delta\in\Ev_n$ be valid blocks, then, the following hold:
		\begin{enumerate}
			\item $\sigma\leq\sigma$
			\item If $\tau\leq\sigma$ and $\sigma\leq\tau$, then $\sigma=\tau$. 
			\item If $\tau\leq\sigma$ and $\sigma\leq\Delta$, then $\tau\leq\Delta$.
		\end{enumerate}
    \end{lemma}
  
	\begin{proof}\leavevmode\vspace*{-0.8em}
		\begin{enumerate}
			\item Clearly $\sigma=\sigma(\emptyset,\emptyset)$, so $\sigma\leq\sigma$. 
			\item Since $\tau\leq\sigma$, then there are disjoint sets $I^+, I^-\subset [m]$ such that $\tau=\sigma(I^+, I^-)$, and $k=m-|I^+|-|I^-|$. Since $\sigma\leq\tau$ then also $m\leq k$, so $k=m$. Then $|I^+|=|I^-|=0$, and $I^+=I^-=\emptyset$. Therefore $\tau=\sigma(I^+,I^-)=\sigma(\emptyset,\emptyset)=\sigma$. 
			\item Note $\tau\leq\sigma$ and $\sigma\leq\Delta$ imply $k\leq m\leq n$. Suppose $\Delta=x_0(1,a_1)\cdots (1,a_n)x_n$, and let $I^+, I^-\subset [n]$ and $J^+, J^-\subset[m]$ sets with $I^+\cap I^-=\emptyset$ and $J^+\cap J^-=\emptyset$, such that $\sigma=\Delta(I^+,I^-)$ and $\tau=\sigma(J^+,J^-)$. Let $\bar{I}=[n]\setminus I^+\setminus I^-$ and write its elements as $\bar{I}=\{i_1<i_2<\dots < i_m\}$. Then we can write $\sigma=y_0(1,a_{i_1})y_1(1,a_{i_1})\cdots (1,a_{i_m})y_m$, and thus $\tau=\sigma(J^+,J^-)=y_0\lambda_1\cdots \lambda_my_m$ where $$y_j=\begin{cases}
				a_{i_j} &\text{ if } j\in J^+\\
				1 &\text{ if } j\in J^-\\
				(1,a_{i_j}) &\text{ if } j\not\in J^+\cup J^-.\\
			\end{cases}$$
	It is important to note that this assertion relies on Lemma \ref{lemma:equivalence_subblocks}, which allows us to use the index sets $ J^+ $ and $ J^- $ even if this description of $\tau=\sigma(J^+, J^-)$ is not yet in canonical form. Then, clearly $\tau=\Delta(K^+,K^-)$ where $K^\pm=I^\pm\cup\{i_j: j\in J^\pm\}$, so $\tau\leq\Delta$ as desired. \qedhere
	\end{enumerate}
	\end{proof}

\begin{theorem}[Restatement of Theorem \ref{thm:block_classification4}]\hypertarget{proof:block_classification4}
    Let $\sigma\in\Ev_4$ be a valid block of dimension 4. Then, $\sigma$ is isomorphic to exactly one of the following blocks:
    \begin{enumerate}
         \item $\bar{\sigma}_1=(1,a)(1,b)(1,a)(1,b)$
        \item $\bar{\sigma}_2=(1,a)(1,b)(1,a)b(1,a)$
        \item $\bar{\sigma}_3=(1,a)(1,b)a^2(1,b)(1,a)$
        \item $\bar{\sigma}_4=(1,a)(1,b)a(1,b)(1,a)$
        \item $\bar{\sigma}_5=(1,a)(1,b)ab(1,a)(1,b)$
        \item $\bar{\sigma}_6=
(1,a)b(1,a)b(1,a)b(1,a)$
    \end{enumerate}
\end{theorem}

\begin{proof}[Proof of Theorem \ref{thm:block_classification4}]
    Notice that for any $I \subset [m]$, we have $|v_I(\sigma)| = |v_\emptyset(\sigma)| + |I|$. Therefore, if $I$ and $J$ are subsets of $[m]$ such that $v_I(\sigma) = v_J(\sigma)$, it must follow that $|I| = |J|$. Note that if $|I| = |J| = m-1$, meaning $I = [m] \setminus \{i\}$ and $J = [m] \setminus \{j\}$, then, assuming without loss of generality that $i < j$, the equality $v_I(\sigma) = v_J(\sigma)$ reduces to the equation $xa_j = a_ix$, as in the proof of Theorem \ref{thm:isom_blocks_vertices}, where $x = x_ia_{i+1}x_{i+1} \cdots x_j$. By Corollary~\ref{cor:simple_word_equations}, this implies $a_i = a_j = a$ and $x = a^t$ for some $t \geq 0$, and so $x_i = a^r$ and $a_{i+1} = a$ for some $r \geq 0$. However, this is impossible because $\sigma$ is valid. Thus, for $I \subset [m]$ with $|I| = m-1$, the vertices $v_I(\sigma)$ are always distinct. Consequently, to determine the isomorphism type of a block, we only need to compare $v_I(\sigma)$ among subsets $I$ of the same size, for $1 \leq |I| \leq m-2$.

    In the specific case where $m=4$, we only need to check the vertices corresponding to subsets $I$ where $|I| = 1$ or $|I| = 2$. Therefore, we begin by computing all such vertices for the six blocks $\bar{\sigma}_1, \dots, \bar{\sigma}_6$, as shown in the following two tables. In each column, we use $(*)$ or $(**)$ to indicate equal vertices.   

\begin{table}
\begin{center}
\begin{tabular}{|c||c|c|c|c|c|c|} \hline 
     I &\phantom{x}$\bar{\sigma}_1$\phantom{x} &\phantom{x}$\bar{\sigma}_2\phantom{x}$ &\phantom{x}$\bar{\sigma}_3$\phantom{x} &\phantom{x}$\bar{\sigma}_4\phantom{x}$ &\phantom{x}$\bar{\sigma}_5\phantom{x}$ &\phantom{x}$\bar{\sigma}_6$\phantom{x}\\ \hline\hline
     $\{1\}$& $a$ (*)& $ab$(*)& $a^2$(*)& $a^3$(*)& $a^2b$&$ab^3$\\ \hline 
     $\{2\}$& $b$ (**)& $b^2$& $ba$& $ba^2$& $bab$&$bab^2$\\ \hline 
     $\{3\}$& $a$ (*)& $ab$(*)& $ab$& $a^2b$& $aba$&$b^2ab$\\ \hline 
     $\{4\}$& $b$ (**)& $ba$& $a^2$(*)& $a^3$(*)& $ab^2$&$b^3a$\\ \hline
\end{tabular}

\begin{tabular}{|c||c|c|c|c|c|c|} \hline 
     I &\phantom{x}$\bar{\sigma}_1$\phantom{x} & \phantom{x}$\bar{\sigma}_2\phantom{x}$ &\phantom{x}$\bar{\sigma}_3$\phantom{x} &\phantom{x}$\bar{\sigma}_4\phantom{x}$ &\phantom{x}$\bar{\sigma}_5\phantom{x}$ &\phantom{x}$\bar{\sigma}_6$\phantom{x}\\\hline
 $\{1,2\}$& $ab$(*)& $ab^2$& $aba^2$& $aba$(*)& $abab$(*)&$abab^2$\\\hline
 $\{1,3\}$& $a^2$& $a^2b$& $a^3b$& $a^2b$& $a^2ba$&$ab^2ba$\\\hline
 $\{1,4\}$& $ab$& $aba$(*)& $a^5$& $a^3$& $a^2b^2$&$ab^3a$\\\hline
 $\{2,3\}$& $ba$& $bab$& $ba^2b$& $bab$& $baba$&$babab$\\\hline
 $\{2,4\}$& $b^2$& $b^2a$& $ba^3$& $ba^2$& $bab^2$&$bab^2a$\\\hline
 $\{3,4\}$& $ab$(*)& $aba$(*)& $a^2ba$& $aba$(*)& $abab$(*)&$b^2aba$\\\hline
\end{tabular}
\end{center}
    \caption{Vertices $v_I(\bar{\sigma})$ when $|I|=1$ or 2, for blocks in Theorem \ref{thm:block_classification4}. }
    \label{table:vertices_4_blocks_bar}
\end{table}

By counting the coincidences of vertices $v_I(\bar{\sigma})$ in each of the different blocks, by Theorem \ref{thm:isom_blocks_vertices} we can already conclude that most of the blocks must be non-isomorphic, with the only potential exception being $\bar{\sigma}_2$ and $\bar{\sigma}_4$. However, suppose for contradiction that $\bar{\sigma}_2 \simeq \bar{\sigma}_4$. Then there must be a permutation $\pi: [4] \to [4]$ such that $v_{\pi(1)}(\bar{\sigma}_4) = v_{\pi(3)}(\bar{\sigma}_4)$ and $v_{\pi(\{1,4\})}(\bar{\sigma}_4) = v_{\pi(\{3,4\})}(\bar{\sigma}_4)$. This implies $\pi(\{1,3\}) = \{1,4\}$ and $\{\pi(\{1,4\}), \pi(\{3,4\})\} = \{\{1,2\}, \{3,4\}\}$.
Now, assuming $\pi(1) = 1$, we must have $\pi(\{1,4\}) = \{1,2\}$ and $\pi(\{3,4\}) = \{3,4\}$, but then $\pi(4) = 2$ and $\pi(4) = 3$, which is a contradiction. If instead we take $\pi(1) = 4$, we must get $\pi(\{1,4\}) = \{3,4\}$, but this leads to $\pi(4) = 4$ and $\pi(3) = 4$, which is also impossible. Therefore, $\bar{\sigma}_2 \not\simeq \bar{\sigma}_4$.

Let now $\sigma\in\Ev_4$ be a valid $4$-block. Removing affixes, we may assume $\sigma=(1,a)x(1,b)y(1,c)z(1,d)$ in canonical form, where $a,b,c,d\in\Sigma$ and $x,y,z\in\Sigma^*$. To find the isomorphism type of $\sigma$, we need to consider all possible coincidences of vertices amount the index sets $I\subset[4]$ of size 1 and 2. Removing the coincidences that produce invalid blocks as in Lemma \ref{lemma:charac_valid}, and up to reversing, we just need to consider the following four cases. 
\begin{enumerate}
    \item If $v_{\{1\}}(\sigma)=v_{\{3\}}(\sigma)$. That is $axyz=xycz$ so $axy=xyc$, and by Corollary \ref{cor:simple_word_equations}, $a=c$, $x=a^r$ and $y=a^t$ for some $r, t\geq 0$. Thus $\sigma=a^r(1,a)(1,b)a^t(1,a)z(1,d)\simeq (1,a)(1,b)a^t(1,a)z(1,d)$, with $a\neq b$ and $a\neq d$. Now need to consider two sub-cases:
    \begin{enumerate}
        \item If $v_{\{2\}}(\sigma)=v_{\{4\}}(\sigma)$. Then $ba^tz=a^tzd$, so by Corollary $\ref{cor:simple_word_equations}$, we have $b=d$,  $z=b^s$ for some $s\geq 0$, and $t=0$. Thus $\sigma\simeq (1,a)(1,b)(1,a)b^s(1,b)\simeq \bar{\sigma}_1$. 
        \item If $v_{\{2\}}(\sigma) \neq v_{\{4\}}(\sigma)$, we claim that $\sigma \simeq \bar{\sigma}_2$, with $\pi$ being the identity permutation. Table \ref{table:vertices_4_blocks} already shows that $\sigma$ shares the same vertex coincidences as $\bar{\sigma}_2$. Furthermore, these are the only coincidences, as we now demonstrate.

        For $|I|=1$, we are assuming $v_{\{2\}}(\sigma) \neq v_{\{4\}}(\sigma)$, so the only possible coincidence that would give a valid block is $v_{\{1\}}(\sigma) = v_{\{4\}}(\sigma)$. This would imply the equation $a^{t+1}z = a^tzd$, which forces $a = d$ and $z = a^s$ for some $s \geq 0$. However, this would make $\sigma$ invalid, so this coincidence can't occur.

        For $|I|=2$, since $a \neq b$, any possible coincidences would have to start either with $a$ or $b$. Starting with $b$, we would need $ba^{t+1}z = ba^tzd$, but canceling the common $b$ gives $a^{t+1}z = a^tzd$, which cannot hold since $v_{\{3\}}(\sigma) \neq v_{\{4\}}(\sigma)$. Coincidences starting with $a$ would be among $\{a^{t+2}z, aba^tz, a^{t+1}zd\}$, but canceling the common starting $a$ would mean a coincidence in $\{a^{t+1}z, ba^tz, a^tzd\} = \{v_{\{1\}}(\sigma), v_{\{2\}}(\sigma), v_{\{4\}}(\sigma)\}$, which we just proved cannot happen. Therefore, the only vertex coincidences occur for $v_{\{1\}}(\sigma) = v_{\{3\}}(\sigma)$ and $v_{\{1,4\}}(\sigma) = v_{\{3,4\}}(\sigma)$, which is the same for $\bar{\sigma}_2$, so by \ref{thm:isom_blocks_vertices},  $\sigma\simeq \bar{\sigma}_2$ as desired. 
    \end{enumerate}
    \item If $v_{\{1\}}(\sigma)=v_{\{4\}}(\sigma)$. That is $axyz=xyzd$, so by Corollary \ref{cor:simple_word_equations}, we have $a=d$, $x=a^r$, $y=a^s$, and $z=a^t$ for some $r,s,t\geq 0$. Thus $\sigma=(1,a)a^r(1,b)a^s(1,c)a^t(1,d)\simeq (1,a)(1,b)a^s(1,c)(1,d)$, where $a\neq b$ and $c\neq d$. Since having $s=0$ and $b=c$ simultaneously would result in an invalid block, then $v_{\{1\}}(\sigma)=v_{\{4\}}(\sigma)$ is the only possible coincidence of vertices $v_I(\sigma)$ with $|I|=1$. To analyze the coincidences when $|I|=2$, we need to consider two sub-cases.
    \begin{enumerate}
        \item If $s \geq 2$ or $b \neq c$, we claim that $\sigma \simeq \bar{\sigma}_3$, with $\pi$ as the identity permutation. Both blocks already share the same vertex table for $|I|=1$ (see Tables \ref{table:vertices_4_blocks_bar} and \ref{table:vertices_4_blocks}), and we just need to confirm that all vertices $v_I(\sigma)$ are distinct for $|I|=2$. This is straightforward if $b \neq c$. If instead $b = c$, the only possible coincidence would be $v_{\{1,2\}}(\sigma) = v_{\{3,4\}}(\sigma)$, meaning $aba^s = a^sba = a^sca$, but this is impossible since $s \geq 2$. Therefore, $\sigma \simeq \bar{\sigma}_3$, as desired.
        \item If $s=1$ and $b=c$, then the vertices $v_I(\sigma)$ for $|I|=2$ become $\{aba,a^2b,a^3, bab, ba^2\}$, with the only coincidence being $v_{\{1,2\}}(\sigma)=aba=v_{\{3,4\}}(\sigma)$. Thus, the coincidence of vertices for $\sigma$ are the exact same as those for $\bar{\sigma}_4$, and so by Theorem \ref{thm:isom_blocks_vertices}, we get $\sigma\simeq \bar{\sigma}_4$. 
    \end{enumerate}
    \item If all $v_I(\sigma)$ are distinct when $|I|=1$, but there are sets $|I_1|=|I_2|=2$, $I_1\neq I_2$ such that $v_{I_1}(\sigma)=v_{I_2}(\sigma)$. 
    \begin{enumerate}
        \item If $I_1\cap I_2\neq \emptyset$, say $I_1=\{i,j\}$ and $I_2=\{i,k\}$. In the case $i<j, k$, it is straightforward to see this implies $v_{\{j\}}(\sigma)=v_{\{k\}}(\sigma)$, which can't happen. The same holds if $j,k <i$. Thus, the only remainder case is, without loss of generality, $j<i<k$. If $i=2$, let $s=|x|$ and note that the first $s+1$ symbols of both vertices are $v_{\{1,2\}}(\sigma)[1,s+1]=ax=xb=v_{\{2,k\}}(\sigma)[1,s+1]$. By Corollary \ref{cor:simple_word_equations}, this would imply $\sigma$ is invalid. If instead $i=3$, we arise to a similar contradiction by considering the last $s+1$  symbols of $v_{\{j,3\}}(\sigma)=v_{\{3,4\}}(\sigma)$, with $s=|z|$ in this case. Therefore, this sub-case is impossible. 

        \item If $v_{\{1,i\}}(\sigma)=v_{\{2,j\}}(\sigma)$ where $\{i,j\}=\{3,4\}$. Let $s=|x|$ and consider the first $s+1$ symbols of the vertices, so $v_{\{1,i\}}(\sigma)[1,s+1]=ax=xb=v_{\{2,j\}}(\sigma)[1,s+1]$. Then, once again by Corollary~\ref{cor:simple_word_equations}, we would have $a=b$ and $x=a^s$, which is impossible since $\sigma$ is a valid block.

        \item If $v_{\{1,2\}}(\sigma)=v_{\{3,4\}}(\sigma)$, the previous two sub-cases imply that no other coincidences can occur when $|I|=2$. Since we have also assumed no coincidences occur when $|I|=1$, $\sigma$ has the exact same coincidences as $\bar{\sigma}_5$, as shown in Table \ref{table:vertices_4_blocks_bar}. Therefore, by Theorem \ref{thm:isom_blocks_vertices}, we have $\sigma\simeq\bar{\sigma}_5$. 
    \end{enumerate}

    \item If all $v_I(\sigma)$ are distinct, then since $\bar{\sigma}_6$ also has all distinct vertices, as shown in Table \ref{table:vertices_4_blocks_bar}, it follows from Theorem \ref{thm:isom_blocks_vertices} that $\sigma\simeq\bar{\sigma}_6$. \qedhere
\end{enumerate}

\begin{table}[]
\begin{center}
\begin{tabular}{|c||c|c|} \hline 
     I &$(1,a)(1,b)a^t(1,a)z(1,d)$&$(1,a)(1,b)a^s(1,c)(1,a)$\\ \hline\hline
     $\{1\}$&$a^{t+1}z$ (*)& $a^{s+1}$(*)\\ \hline 
     $\{2\}$&$ba^tz$& $ba^s$\\ \hline 
     $\{3\}$&$a^{t+1}z$ (*)& $a^sc$\\ \hline 
     $\{4\}$&$a^tzd$& $a^{s+1}$(*)\\ \hline
\end{tabular}
\begin{tabular}{|c||c|c|} \hline 
     I &$(1,a)(1,b)a^t(1,a)z(1,d)$& $(1,a)(1,b)a^s(1,c)(1,a)$\\\hline
 $\{1,2\}$&$aba^tz$& $aba^s$\\\hline
 $\{1,3\}$&$a^{t+2}z$& $a^{s+1}c$\\\hline
 $\{1,4\}$&$a^{t+1}zd$ (*)& $a^{s+2}$\\\hline
 $\{2,3\}$&$ba^{t+1}z$& $ba^sc$\\\hline
 $\{2,4\}$&$ba^tzd$& $ba^{s+1}$\\\hline
 $\{3,4\}$&$a^{t+1}zd$ (*)& $a^sca$\\\hline
\end{tabular}
\end{center}
    \caption{Vertices $v_I(\sigma)$ when $|I|=1$ or 2, for blocks in the proof of Theorem \ref{thm:block_classification4}.}
    \label{table:vertices_4_blocks}
\end{table}
\end{proof}

\subsection{Insertion Block and Insertion Chain Complexes}
\begin{theorem}[Restatement of Lemma \ref{lemma:partial_partial_zero}]\hypertarget{proof:partial_partial_zero}
Let $\sigma\in\Ev_m$ be a valid $m$-block, then
    $$\partial\partial\sigma =0. $$
\end{theorem}

\begin{proof}[Proof of Lemma \ref{lemma:partial_partial_zero}]
    If $m\leq 1$, the result is trivial since $\partial_0\equiv 0$. Hence we may assume $m\geq 2$ and write $\sigma=x_0(1,a_1)x_1\cdots (1,a_m)x_m$ in canonical form. Observe applying the linearity of $\partial$ and reorganizing the terms of the sums we obtain
    \begin{align*}
        \partial\partial\sigma&=\sum_{i=1}^m (-1)^{i+1}\partial\left[\bar{F}_i(\sigma)-\underline{F}_i(\sigma)\right]=\\
        &=\sum_{i=1}^m(-1)^{i+1}\left[\sum_{j=1}^{m-1} (-1)^{j+1}\left(\bar{F}_j(\bar{F}_i(\sigma)) -\underline{F}_j(\bar{F}_i(\sigma)) \right) - \sum_{j=1}^{m-1} (-1)^{j+1}\left(\bar{F}_j(\underline{F}_i(\sigma)) -\underline{F}_j(\underline{F}_i(\sigma)) \right)\right]=\\
        &= \sum_{i=1}^m\sum_{j=1}^{m-1}(-1)^{i+j} F(j,i) =\sum_{i=1}^m\sum_{j=1}^{i-1} (-1)^{i+j} F(j,i) + \sum_{i=1}^m\sum_{j=i}^{m-1}(-1)^{i+j}F(j,i)=\\
        &=\sum_{i=1}^m\sum_{j=1}^{i-1} (-1)^{i+j} F(j,i) + \sum_{i=1}^m\sum_{j=i+1}^{m}(-1)^{i+j-1}F(j-1,i)=\\
        &=\sum_{1\leq j<i\leq m}(-1)^{i+j} F(j,i) + \sum_{1\leq i<j\leq m}(-1)^{i+j-1} F(j-1,i)=\\
        &=\sum_{1\leq i<j\leq m}(-1)^{i+j} F(i,j) + \sum_{1\leq i<j\leq m}(-1)^{i+j-1} F(j-1,i)= \sum_{1\leq i<j\leq m} (-1)^{i+j}\left[  F(i,j)- F(j-1, i)\right]
    \end{align*}
    where $F(i,j)=\bar{F}_i(\bar{F}_j(\sigma))-\underline{F}_i(\bar{F}_j(\sigma))-\bar{F}_i(\underline{F}_j(\sigma))+\underline{F}_i(\underline{F}_j(\sigma))$. Combining like terms we get:

    \begin{align*}
        \partial\partial\sigma &= \overbrace{\sum_{1\leq i <j\leq m}(\text{-}1)^{i+j} \left(\bar{F}_i\bar{F}_j(\sigma)-\bar{F}_{j-1}\bar{F}_i(\sigma) \right)}^A-\overbrace{\sum_{1\leq i <j\leq m}(\text{-}1)^{i+j} \left(\underline{F}_i\bar{F}_j(\sigma)-\bar{F}_{j-1}\underline{F}_i(\sigma) \right)}^B+\\
        &-\underbrace{\sum_{1\leq i <j\leq m}(\text{-}1)^{i+j} \left(\bar{F}_{i}\underline{F}_{j}(\sigma)-\underline{F}_{j-1}\bar{F}_i(\sigma)\right)}_C + \underbrace{\sum_{1\leq i <j\leq m}(\text{-}1)^{i+j} \left(\underline{F}_i\underline{F}_j(\sigma)-\underline{F}_{j-1}\underline{F}_i(\sigma)\right)}_D\\
        &= A-B-C+D
    \end{align*}

    We prove now that $A=B=C=D=0$. 
    \begin{enumerate}
        \item Note that for each $1\leq i<j\leq m$,  $\bar{F}_i\bar{F}_j(\sigma)=\sigma(\{i,j\},\emptyset)= \bar{F}_{j-1}\bar{F}_i(\sigma)$, which is a well defined valid block, since, by Proposition \ref{prop:facets}, taking upper-facets always produces valid blocks. Thus $$\bar{F}_i\bar{F}_j(\sigma)- \bar{F}_{j-1}\bar{F}_i(\sigma)=\sigma(\{i,j\},\emptyset)-\sigma(\{i,j\},\emptyset)=0,$$
        and hence $A=0$. 

        \item For each $1\leq i<j\leq m$, note that $\underline{F}_i\bar{F}_j(\sigma) = \sigma(\{j\}, \{i\})$ is a valid block whenever $\sigma(\emptyset, \{i\})$ is valid; otherwise, $\underline{F}_i\bar{F}_j(\sigma) = 0$. Similarly, $\bar{F}_{j-1}\underline{F}_i(\sigma) = \sigma(\{j\}, \{i\})$, provided it is valid; otherwise, it is zero. It is straightforward to check that $\sigma(\{j\},\{i\})$ is valid if and only if $\sigma(\emptyset,\{i\})$ is valid. Therefore, both terms are either $\sigma(\{i\},\{j\})$ or zero simultaneously. Thus,  
        $$\underline{F}_i\bar{F}_j(\sigma)-\bar{F}_{j-1}\underline{F}_i(\sigma)=0.$$
        A similar argument shows that
        $$\underline{F}_{j-1}\bar{F}_i(\sigma)-\bar{F}_{i}\underline{F}_{j}(\sigma)=0.$$
        Hence $B=C=0$.
        
        \item Observe $\underline{F}_i\underline{F}_j(\sigma) =\sigma(\emptyset, \{i, j\})$ provided that both $\sigma(\emptyset, \{i, j\})$ and $\underline{F}_j(\sigma)$ are valid, and it is zero otherwise. Similarly, $\underline{F}_{j-1}\underline{F}_i(\sigma) = \sigma(\emptyset, \{j, i\})$ if both $\sigma(\emptyset, \{j, i\})$ and $\underline{F}_i(\sigma)$ are valid, and zero otherwise. Since $\underline{F}_i\underline{F}_j(\sigma) - \underline{F}_{j-1}\underline{F}_i(\sigma) = 0$ when both are valid or zero, we only need to consider the cases where one is invalid, but the other is not. 

        Let $I, J$ be the sets of indices where one term is a valid block and the other is zero, 
        \begin{align*}
            I&=\{(i,j): 1\leq i<j\leq m, \underline{F}_i\underline{F}_j(\sigma)\text{ is valid, but }\underline{F}_{j-1}\underline{F}_i(\sigma)=0\}\\
            J&=\{(i,j): 1\leq i<j\leq m, \underline{F}_{j-1}\underline{F}_i(\sigma)\text{ is valid, but }\underline{F}_{i}\underline{F}_j(\sigma)=0\}.
        \end{align*}

        Let $(i,j)\in I$, thus $\sigma(\emptyset, \{i, j\})$ and $\sigma(\emptyset, \{j\})$ are both valid, but $\sigma(\emptyset, \{i\})$ is invalid. Since the expression for $\sigma(\emptyset,\{i\})$ matches that of $\sigma$ except in between $(1,a_{i-1})$ and $(1,a_{i+1})$ where we get $$\sigma(\emptyset,\{i\})=\cdots (1,a_{i-1})x_{i-1}x_{i}(1,a_{i+1})\cdots,$$
        and $\sigma$ is valid, we must have $a_{i-1}=a_{i+1}=a$ and $x_i=a^s, x_{i-1}=a^r$ for some $s, r\geq 0$. However, the expression of $\sigma(\emptyset, \{i,j\})$ will include the same expression, making it invalid, unless $j=i+1$, in which case we will have 
        \begin{align*}
            \sigma(\emptyset,\{i,j\})=\sigma(\emptyset, \{i,i+1\})&=\cdots x_{i-2}(1,a_{i-1})x_{i-1}x_{i}x_{i+1}(1,a_{i+2})\cdots=\\
            &=\cdots x_{i-2}(1,a)a^{r+s}x_{i+1}(1,a_{i+2})\cdots
        \end{align*}
    
        which is valid provided that $a_{i+2}\neq a$, or $x_{i+1}\neq a^t$ for every $t\in\N$. Furthermore, we observe that the expression of $\sigma(\emptyset, \{i-1,i\})$ matches that of $\sigma$ except for the part in between $(1,a_{i-2})$ and $(1,a_{i+2})$, where we get
        \begin{align*}
            \sigma(\emptyset,\{i-1, i\})&=\cdots (1,a_{i-2})x_{i-2}x_{i-1}x_i(1,a_{i+1})x_{i+1}(1,a_{i+2})\cdots\\
            &=\cdots (1,a_{i-2})x_{i-2}a^{r+s}(1,a)x_{i+1}(1,a_{i+2})\cdots
        \end{align*}
        So $\sigma(\emptyset,\{i-1,i\})=\sigma(\emptyset,\{i,i+1\})$, and it is a valid block. Moreover, we can easily check that $\sigma(\emptyset, \{i-1\})$ will also be valid, so $\underline{F}_{i-1}\underline{F}_{i-1}(\sigma)=\sigma(\emptyset,\{i-1,i\})$ is a valid block, but $\underline{F}_{i-1}\underline{F}_{i}=0$, so $(i,i-1)\in J$. 

        Therefore, if $(i,j)\in I$ it must be that $(i,j)=(i,i+1)$, which implies $(i,i-1)\in J$, and satisfies 
        \begin{align}\label{eq:F_iF_j}
            \underline{F}_i\underline{F}_{i+1}(\sigma)=\sigma(\emptyset,\{i,i+1\})=\sigma(\emptyset,\{i,i-1\})=\underline{F}_{i-1}\underline{F}_{i-1}(\sigma).
        \end{align}
        A similar argument shows that if $(j,i)\in J$, then $(j,i)=(i-1,j)$, so $(i,i+1)\in I$, and the equation (\ref{eq:F_iF_j}) also holds. Thus, there exists a set $K\subset [m]$, such that $I=\{(i,i+1): i\in K\}$ and $J=\{(i-1,i): i\in K\}$. Therefore, we can compute $D$ as:
        \begin{align*}
            D&=\sum_{1\leq i <j\leq m}(\text{-}1)^{i+j} \left(\underline{F}_i\underline{F}_j(\sigma)-\underline{F}_{j-1}\underline{F}_i(\sigma)\right)\\
            &=\sum_{(i,i+1)\in I}(\text{-}1)^{i+j}\underline{F}_i\underline{F}_{i+1}(\sigma) -\sum_{(i-1,i)\in I}(\text{-}1)^{i+j}\underline{F}_{i-1}\underline{F}_{i-1}(\sigma)=\\
            &=\sum_{i\in K}(\text{-}1)^{i+j} \underline{F}_i\underline{F}_{i+1}(\sigma)-\underline{F}_{i-1}\underline{F}_{i-1}(\sigma) =0. 
        \end{align*}
        Therefore $D=0$ as desired, finishing the proof. 
    \end{enumerate}
\end{proof}

\subsection{Homology}

\hypertarget{proof:classification_cycles_1dim} Classifying all nontrivial 1-dimensional cycles of length 4 in Theorem \ref{thm:classification_cycles_1dim} involves translating the relationships between vertices into word equations. For each possible \textit{insertion pattern} that arises, the same set of word equations consistently appears, thus we present them as a separate lemma.

\begin{lemma}\label{lemma:system_word_equations}
	Let $a,b\in\Sigma$ be two, possibly equal, known symbols, and $x,x',y,y',z,z'\in\Sigma^*$ be variables representing strings. Then, up to a common prefix for $x$ and $x'$, or a common suffix for $y$ and $y'$, we have:
	\begin{enumerate}
		\item All the simultaneous solutions to $xaybz=x'ay'bz'\text{  and  } xyz=x'y'z'$ are\\
			
        $S_1\!:\!\Solutions{1}{a^r\gamma b^s}{1}{a^r}{\gamma}{b^s}$%
        $S_2\!:\!\Solutions{1}{a^r}{(ab)^tb^s}{a^r(ab)^t}{b^s}{1}$ and %
        $S_3\!:\!\Solutions{1}{a^r\gamma}{b^s}{a^r}{\gamma b^s}{1}$ 
        
        \item  All the simultaneous solutions to $xaybz=x'by'az'\text{  and  } xyz=x'y'z'$ are\\
        
        $S'_1\!:\!\Solutions{1}{a^r\gamma a^s}{1}{a^r}{\gamma}{a^s}$%
        $S'_2\!:\!\Solutions{1}{a^r}{(ab)^ta^{s+1}}{a^{r+1}(ba)^t}{a^s}{1}$ and %
        $S'_3\!:\!\Solutions{1}{a^r\gamma}{a^s}{a^r}{\gamma a^s}{1}$.
			
	\end{enumerate} 
		For non-negative integers $r,s,t$. 
\end{lemma} 

\begin{proof}
Let $w = xyz = x'y'z'$, then the equation $xaybz = x'cy'dz'$ with $\{c,d\} = \{a,b\}$ can be interpreted as stating that, starting from $w$, we can obtain the same resulting word by inserting $a$ and $b$ in two different ways. Let $w = c_1\cdots c_n$ for symbols $c_i \in \Sigma$. Visually, we can represent the word $w$ with a horizontal line $w = \vcenter{\hbox{\rule{1cm}{0.5pt}}}$, and inserting $a$ and $b$ at positions $i$ and $j$, respectively, can be represented as

$$c_1\cdots c_i\ a\ c_{i+1} \cdots c_j\ b\ c_{j+1} \cdots c_n = \vcenter{\hbox{\rule{0.3cm}{0.5pt}}} a \vcenter{\hbox{\rule{0.4cm}{0.5pt}}} b \vcenter{\hbox{\rule{0.3cm}{0.5pt}}}.$$

Then, to signify that $a$ and $b$ can be inserted in two different ways resulting in the same word, we stack both diagrams and interpret that the symbol corresponding to each vertical position must match in both lines. For instance, if there are indices $1 \leq i \leq k \leq j \leq n$ such that inserting $a, b$ at positions $0, j$ and $i, k$ produces the same word, we can depict this with the following diagram:

\noindent 
\begin{center}
    \begin{minipage}[c]{0.45\textwidth}
    \begin{flushright}
    \begin{align*}
        &\ a\ c_1c_2\cdots \cdots\cdots\cdots\cdots\cdots c_j\ b\ c_{j+1}\cdots c_n\\
        &c_1c_2\cdots c_i\ a\ c_{i+1}\cdots c_j\ b\ c_{k+1}\cdots\cdots\cdots c_n
    \end{align*}
    \end{flushright}
\end{minipage}
= 
\begin{minipage}[c]{0.45\textwidth}
    \begin{flushleft}
    \begin{align*}
        \lineword{a}{3}{b}{1}{1}{a}{1}{b}{2}=\ExampleDiagramI
    .\end{align*}
    \end{flushleft}
    \end{minipage}
\end{center}
Since the same symbols $c$ appear above and below, we can add vertical lines so that each parallelogram formed contains matching symbols on both sides. Every time we observe a rectangle, it indicates a factor shared identically between the two words. This allows us to break the equation into several smaller equations. In this particular example, the leftmost parallelogram involving $a$ gives us the equation $a\hat{x} = \hat{x}a$. The rectangle in the middle indicates a common factor for both equations, without any further conditions. The parallelogram involving $b$ results in the equation $\hat{y}b = b\hat{y}$. Finally, the last rectangle represents another common factor, but since we're interested in solutions up to common affixes, we can ignore rectangles at the beginning or end of the words. Just as in this example, for each of the cases we need to consider, the smaller equations that arise are those in Lemma \ref{prop:simple_equation_words}. We can then incorporate their solutions into the diagram, and thus we can determine the values for $x, y, z$ and $x', y', z'$ by considering their relative positions to $a$ and $b$.

We now apply this diagrammatic process  to solve the equations as follows. Consider the equation $xaybz = x'ay'bz'$. By removing common affixes, we can assume, without loss of generality, that $x = 1$. To represent this as a diagram, we need to analyze all possible relative positions of $a$ and $b$ in $x'ay'bz'$ compared to their positions in $xaybz$. Up to symmetries by reflection, this leads to the following three cases:

\begin{enumerate}
    \item $\CaseIb = \CaseIc =\CaseId$.
    
    Here Corollary \ref{cor:simple_word_equations} gives the solutions to $a\hat{x}=\hat{x}a$ and $\hat{y}b=b\hat{y}$, which we included in the diagram. $\gamma\in\Sigma^*$ can be any word. Then, we get $x, y,z$ and $x', y', z'$ by their relative positions to $a$ and $b$, giving the solution 
    $$S_1: \Solutions{1}{a^r\gamma b^s}{1}{a^r}{\gamma}{b^s}.$$

    \item $\CaseIIb = \CaseIIc = \CaseIId$. 
    
    Similarly to the previous case, we obtain 
    $S_2: \Solutions{1}{a^r}{(ab)^tb^s}{a^r(ab)^t}{b^s}{1}$.

    \item $\CaseIIIb =\CaseIIIc =\CaseIIId$. 
    
    Gives $S_3:\Solutions{1}{a^r\gamma}{b^s}{a^r}{\gamma b^s}{1}$.
\end{enumerate}

Similarly, the equation $xaybz=x'by'az'$ can be solved, up to reverse and common affixes, by considering three cases:

\begin{enumerate}
    \item $\CaseIVa = \CaseIVb =\CaseIVc$.
    
    Thus $a=b$ and we get
    $S_4:\Solutions{1}{a^r\gamma a^s}{1}{a^r}{\gamma}{a^s}.$
    
    \item
    $\begin{aligned}[t]
        \CaseVa &=\CaseVb =\CaseVc\\
        &=\CaseVd =\CaseVe.
    \end{aligned}$
    
    Where we used the fact that $aa^r=a^ra$ in order to reduce the parallelogram in the middle to equation $ab\hat{x}=\hat{x}ba$. Hence  $S_5:\Solutions{1}{a^r}{(ab)^ta^{s+1}}{a^{r+1}(ba)^t}{a^s}{1}$

    \item $\CaseVIa =\CaseVIb =\CaseVIc$. 
    
    Thus $a=b$ and we get $S_6: \Solutions{1}{a^r\gamma}{a^s}{a^r}{\gamma a^s}{1}.$
\end{enumerate}    
\end{proof}

\begin{theorem}[Restatement of Theorem \ref{thm:classification_cycles_1dim}]
    Let $\Sigma$ be a finite alphabet, and $W\subset\Sigma^*$, with $|W|=4$, be a set of 4 vertices. Then, if $H_1(\ChainBlock{W})\neq 0$, then $\ChainBlock{W}$ is isomorphic via reflection, symbol permutation or affixes to $\ChainBlock{V}$ where $V\subset \{a,b\}^*$ is one of the following:
    \begin{enumerate}
        \item $V_1=\left\{(ab)^t, a(ab)^t, (ab)^ta,(ab)^{t+1} \right\}$, for some $t\geq 1$. 
        \item $V_2=\left\{(ab)^t, a(ab)^t, (ab)^tb,(ab)^{t+1} \right\}$, for some $t\geq 2$. 
        \item $V_3=\left\{(ab)^t,(ab)^ta,b(ab)^t,(ab)^{t+1}\right\}$, for some $t\geq 1$. 
         \item $V_4=\left\{(ab)^ta,b(ab)^ta,(ab)^{t+1},b(ab)^t\right\}$, for some $t\geq 0$.
        \item $V_5=\left\{(ab)^{t+1},(ab)^{t+1}a,b(ab)^{t+1},b(ab)^ta\right\}$, for some $t\geq 0$.
         \item $V_6=\left\{(ab)^{t+1},a(ab)^{t+1},(ab)^{t+1}b,a(ab)^b\right\}$, for some $t\geq 1$.              
    \end{enumerate}
\end{theorem}

\begin{proof}[Proof of Theorem \ref{thm:classification_cycles_1dim}]
    Note that for an edge $v\sim w$, the lengths of the words $v$ and $w$ must satisfy $||v|-|w||=1.$ Hence, the lengths of the 4 vertices of a cycle, must follow one of two possible sequences $(s,s+1,s+1,s+2)$ or $(s,s,s+1,s+1)$, as in Figure \ref{fig:4-cycles1}. The arrows in the figure indicate the direction of symbol insertion, and $a,b,c,d\in\Sigma$ represent the inserted symbols. 

\begin{figure}
   \begin{center}
\begin{tikzpicture}
    \node at (0,0) (A) {$s$} ; 
    \node at (2,0) (B) {$s+1$} ; 
    \node at (0,2) (D) {$s+1$} ; 
    \node at (2,2) (C) {$s+2$} ; 
    
    \path[->] (A) edge node[auto=left] {$a$} (B);
    \path[->] (A) edge node[auto=left] {$b$} (D);
    \path[->] (B) edge node[auto=left] {$c$} (C);
    \path[->] (D) edge node[auto=left] {$d$} (C);
\end{tikzpicture}
\begin{tikzpicture}
    \node at (0,0) (A) {$s$} ; 
    \node at (2,0) (B) {$s+1$} ; 
    \node at (0,2) (D) {$s+1$} ; 
    \node at (2,2) (C) {$s$} ; 
    
    \path[->] (A) edge node[auto=left] {$a$} (B);
    \path[->] (A) edge node[auto=left] {$b$} (D);
    \path[<-] (B) edge node[auto=left] {$c$} (C);
    \path[<-] (D) edge node[auto=left] {$d$} (C);
\end{tikzpicture}%
\end{center}
    \caption{1-dimensional 4-cycles in proof of Theorem \ref{thm:classification_cycles_1dim}}
    \label{fig:4-cycles1}
\end{figure}
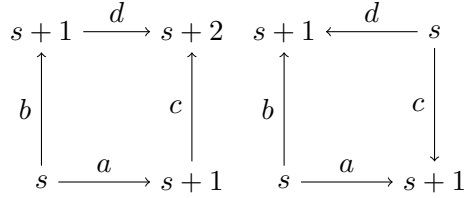

In the first case, let $u,v_1,v_2,w$ be the vertices, such that $|u|=s, |v_1|=|v_1|=s+1$ and $|w|=s+2$. Thus $w$ is obtained either from inserting $a$ and $c$ to $u$, or from inserting $b$ and $d$ to $u$, and hence $\{a,c\}=\{b,d\}$ as multisets. In the second case, name the vertices such that $|u|=|w|=s$ and $|v_1|=|v_1|=s+1$. Thus, we get $v_1$ from inserting $a$ to $u$, or inserting $c$ to $w$. Similarly, $v_2$ is $u$ with $b$ inserted, or $w$ with $d$ inserted. Then, for any symbol $e\in\Sigma$, denote by $|x|_e$  the number of times the symbol $e$ appears in $x$, and observe
\begin{align*}
    |v_1|_e=|u|_e+|a|_e=|w|_e+|c|_e &\Rightarrow |w|_e=|u|_e+|a|_e-|c|_e\\
    |v_2|_e=|u|_e+|b|_e=|w|_e+|d|_e &\Rightarrow |w|_e=|u|_e+|b|_e-|d|_e
\end{align*}
Thus $|a|_e+|d|_e=|b|_e+|c|_e$. 
By setting $e=a$, we get that $a=b$ or $a=c$. If $a\neq c$, then $a=b$, and setting $e=d$ produces that $c=d$. If $a\neq b$, then $a=c$, and setting $e=b$ produces that $b=d$. Thus, again, $\{a,c\}=\{b,d\}$ as multisets. 

Therefore, there are exactly 4 possible \textit{insertion patterns}, and we only need to care about two, possible equal, symbols $a$ and $b$. 
\begin{center}
A)\begin{tikzpicture}
    \node at (0,0) (A) {$u$} ; 
    \node at (2,0) (B) {$v_1$} ; 
    \node at (0,2) (D) {$v_2$} ; 
    \node at (2,2) (C) {$w$} ; 
    
    \path[->] (A) edge node[auto=left] {$a$} (B);
    \path[->] (A) edge node[auto=left] {$a$} (D);
    \path[->] (B) edge node[auto=left] {$b$} (C);
    \path[->] (D) edge node[auto=left] {$b$} (C);
\end{tikzpicture}
B)\begin{tikzpicture}
    \node at (0,0) (A) {$u$} ; 
    \node at (2,0) (B) {$v_1$} ; 
    \node at (0,2) (D) {$v_2$} ; 
    \node at (2,2) (C) {$w$} ; 
    
    \path[->] (A) edge node[auto=left] {$a$} (B);
    \path[->] (A) edge node[auto=left] {$b$} (D);
    \path[->] (B) edge node[auto=left] {$b$} (C);
    \path[->] (D) edge node[auto=left] {$a$} (C);
\end{tikzpicture}%
C)\begin{tikzpicture}
    \node at (0,0) (A) {$u$} ; 
    \node at (2,0) (B) {$v_1$} ; 
    \node at (0,2) (D) {$v_2$} ; 
    \node at (2,2) (C) {$w$} ; 
    
    \path[->] (A) edge node[auto=left] {$a$} (B);
    \path[->] (A) edge node[auto=left] {$a$} (D);
    \path[<-] (B) edge node[auto=left] {$b$} (C);
    \path[<-] (D) edge node[auto=left] {$b$} (C);
\end{tikzpicture}%
D)\begin{tikzpicture}
    \node at (0,0) (A) {$u$} ; 
    \node at (2,0) (B) {$v_1$} ; 
    \node at (0,2) (D) {$v_2$} ; 
    \node at (2,2) (C) {$w$} ; 
    
    \path[->] (A) edge node[auto=left] {$a$} (B);
    \path[->] (A) edge node[auto=left] {$b$} (D);
    \path[<-] (B) edge node[auto=left] {$a$} (C);
    \path[<-] (D) edge node[auto=left] {$b$} (C);
\end{tikzpicture}%
\end{center}
We now show that finding all possible words $\{u,v_1,v_2,w\}$ satisfying these insertion patterns, means solving the equations in Lemma \ref{lemma:system_word_equations}, and checking which solutions form a nontrivial cycle, for each of the patterns.  Since we only care about identifying solutions up to common affixes, permutation of symbols, or reverse, Lemma \ref{lemma:system_word_equations} will indeed give all the needed solutions. Moreover, since solutions $S'_1$ and $S'_3$ are precisely solutions $S_1$ and $S_3$ making $a=b$, respectively, we only need to check when the cycle is nontrivial for the solutions $S_1, S_2, S_3$ and $S'_2$. 

In the following analysis, $x,y,z, x',y',z'\in\Sigma^*$ are supposed to be variable words. 

\begin{enumerate}
    \item[A)] We get $w$ from $u$ by inserting symbols $a$ and $b$ in two different ways, thus, up to the isomorphisms already mentioned, we may write $u=xyz=x'y'z'$ and $w=xaybz=x'ay'bz'$ or $w=xaybz=x'by'az'$. Hence $v_1=xayz$ and $v_2=x'ay'z'$ or $v_2=x'y'az'$, respectively. Notice all possible solutions will be nontrivial cycles, as long as $v_1\neq v_2$. From Lemma \ref{lemma:system_word_equations} we get the solutions:
    \begin{itemize}
        \item Both solutions $S_1$ and $S_3$ result in  $$\SolVertices{a^r\gamma b^s}{a^{r+1}\gamma b^s}{a^{r+1}\gamma b^s}{a^{r+1}\gamma b^{s+1}},$$ where $v_1=v_2$, so they don't produce nontrivial cycles.
        \item Solution $S_2$ produces $$\SolVertices{a^r(ab)^tb^s}{a^{r+1}(ab)^tb^s}{a^r(ab)^tab^s}{a^r(ab)^tab^{s+1}},$$ which is a valid solution since $v_1\neq v_2$ given $t\geq1$. After removing common affixes we get: $$\SolVertices{(ab)^t}{a(ab)^t}{(ab)^ta}{(ab)^{t+1}}, t\geq1$$ 
        \item Solution $S'_2$ produces, after removing common affixes, the same nontrivial cycle $$\SolVertices{(ab)^t}{a(ab)^t}{(ab)^ta}{(ab)^{t+1}}, t\geq1$$ 
    \end{itemize}
    Thus, only the set $V_1$, for $t\geq 1$ produces nontrivial cycles with insertion pattern A. 
      
    \item[B)] Similar to pattern A), we get $w$ from $u$ by inserting $a,b$ in two different ways. So $u$ and $w$ have the same expressions as in pattern A), but here $v_1=xayz$ and $v_2=x'y'bz'$ or $v_2=x'by'z'$. All solutions from Lemma \ref{lemma:system_word_equations} will be valid as long as we don't get a 2-cell, so we must verify $v_2\neq xybz$. It is easy to see that solutions $S_1$ and $S_3$ are invalid since they produce $v_2=xybz$. Thus, it only remains to examine solutions $S_2$ and $S'_2$, which, after deleting common affixes, produce:
    \begin{align*}
        S_2&:\SolVertices{(ab)^t}{a(ab)^t}{(ab)^tb}{(ab)^{t+1}},\\
        S'_2&:\SolVertices{(ab)^t}{b(ab)^t}{(ab)^ta}{(ab)^{t+1}}.
    \end{align*}

    Where it is straightforward that the cycles are nontrivial, provided $t\geq 2$ for $S_2$ and $t\geq 1$ for $S_2'$. Therefore, pattern B) generates nontrivial cycles only for the sets $V_2$ and $V_3$ (for the latter case, we need to swap the roles of $a$ and $b$, which is acceptable as we are interested only in solutions up to symbol permutation).
    
    Notice that these are indeed two different solutions, even after any affixes, reversing, or symbol permutation, since in the first solution, given $t\geq 2$, $v_1$ has $aa$ and $v_2$ has a $bb$ in their expression. However, in the second solution, neither $v_1$ nor $v_2$ has any symbol twice in a row. And there is no way of getting $aa$ in $v_1$ and $bb$ in $v_2$ simultaneously by adding common affixes, reversing, or swapping the symbols. 
    
    \item[C)] Notice that $v_1$ is obtained from inserting $a$ in $u$ and also from inserting $b$ in $w$. Let $i$ be the position in $v_1$ of the instance of $a$ inserted to $u$, and similarly, let $j$ be the position of the $b$ inserted to $w$. Notice that since $u\neq w$, then $i\neq j$, even if we had $a=b$. Let $\hat{v}$ be the subword of $v_1$ obtained by deleting positions $i$ and $j$. 
    Then, we may write $\hat{v}=xyz$ and, without loss of generality, $u=xybz$ and $v_2=xaybz$, because we know $u$ is obtained from $\hat{v}$ by inserting one $b$, and $v_2$ is obtained from $u$ by inserting one $a$. At the same time, we can also write $\hat{v}=x'y'z'$, $w=x'ay'z'$ or $w=x'y'az'$ and $v_2=x'ay'bz'$ or $v_2=x'by'az'$, respectively. 
    Thus we need to solve $xaybz=x'ay'bz'$ or $xaybz=x'by'az'$, with $xyz=x'y'z'$. The solution produces $u, w$ and $v_2$, then, it gives a nontrivial cycle if a vertex $v_1\neq v_2$ that completes the cycle exists. Checking this for the solutions from Lemma \ref{lemma:system_word_equations}, only produces nontrivial cycles for $S_2$ and $S'_2$, for $t\geq0$, which, after deleting common affixes, both correspond to set $V_4$. 
    
    \item[D)] Here $v_1$ is obtained from $u$ or $w$ by inserting $a$. So $v_1[i]=v_1[j]=a$, where $i$ and $j$ are the positions where the insertion from $u$ and $v$ happens, respectively. Since $u\neq w$, then $i\neq j$. Let $\hat{v}$ be equal to $v_1$ without these two instances of $a$. Then, without loss of generality, we may assume $\hat{v}=xyz$, $u=xayz$ and $v_2=xaybz$. Then, we must be able to also write $\hat{v}=x'y'z'$ and  $w=x'ay'z'$ or $w=x'y'az'$, and $v_2=x'ay'bz'$ or $v_2=x'by'az'$, respectively. Thus we need to solve $xaybz=x'ay'bz'$ or $xaybz=x'by'az'$, with $xyz=x'y'z'$. These, as for pattern C, gives expressions for $u, w$ and $v_2$, then, it will give a nontrivial cycle, if a vertex $v_1$ fitting the pattern exists. Checking this for the solutions from Lemma \ref{lemma:system_word_equations}, only produces nontrivial cycles for $S_2$ with $t\geq0$, and $S'_2$ with $t\geq1$, which, after deleting common affixes, correspond to sets $V_5$ and $V_6$, respectively. Similar to the distinction between sets $V_1$ and $V_2$, we notice these two are distinct even if we swap symbols, add common affixes, or reverse the ordering, since $v_1$ and $v_2$ in $V_6$ have repeated symbols in a row, while for $V_5$ this is not the case.  

\end{enumerate}
\end{proof}

\subsection{Minimal Homological 2-Sphere}
We now conclude the proof of Theorem \ref{thm:min_cycles_d}. We present a computational approach demonstrating that $\mu(\S^2) \geq 8$. This result, along with the explicit set of words \( W = \{a, a^2, b, b^2, ab, ba, bab, aba\} \) with associated Block Complex shown in Figure \ref{fig:min_2_cycle}, finishes the proof. We start with a couple of key lemmas for our computations. 

\begin{lemma}\label{lemma:squares_sharing_edges}
    Let $\sigma_1, \sigma_2 \in \Ev_2$ be two distinct valid 2-dimensional blocks that share two edges. Then, there exist symbols $a, b \in \Sigma$ and words $x, y \in \Sigma^*$ such that the blocks are given by one of the following forms:
\begin{enumerate}
    \item $\sigma_1 = x(1, a)(1, b)y$ and $\sigma_2 = x(1, b)(1, a)y$; or
    \item $\sigma_1 = x(1, a)(1, b)ay$ and $\sigma_2 = xa(1, b)(1, a)y$, or vice versa.
\end{enumerate}
\end{lemma}

\begin{proof}
   Write $\sigma_1 = x(1, a)y(1, b)z$ and $\sigma_2 = \alpha(1, c)\beta(1, d)\gamma$ in canonical form, with $a, b, c, d \in \Sigma$ and $x, y, z, \alpha, \beta, \gamma \in \Sigma^*$. Denote the vertices of $\sigma_i$ by $u_i, v_i, v_i'$ and $w_i$, for $i = 1, 2$, ordered by their length such that $|u_i| < |v_i| = |v_i'| < |w_i|$. Since the two blocks are distinct, they cannot share all four vertices (by Theorem \ref{thm:unique_block_by_vertices}), therefore, the two edges they share must have one vertex in common. Up to analogous cases, this produces three main cases, each of which further subdivides into two subcases. Because the treatment of each main case is very similar, we provide the complete details for the first one only.

    \begin{enumerate}
        \item The common edges intersect at vertex $u_1=u_2$. So $v_1=v_2$ and $v_1'=v_2'$, or, $v_1=v_2'$ and $v_1'=v_2$. 
        \begin{enumerate}
            \item If $u_1=u_2$,  $v_1=v_2$ and $v_1'=v_2'$, then we get the edge equalitites $$x(1,a)yz=\alpha(1,c)\beta\gamma \text{ and } xy(1,b)z=\alpha\beta(1,d)\gamma.$$ 
            Clearly, this implies $a=c$ and $b=d$. Moreover, note the edges in the second equation are in canonical form, so $z=\gamma$. Thus, the first equation reduces to $x(1,a)y=\alpha(1,a)\beta$, which is in canonical form, and thus $x=\alpha$ and $y=\beta$. Therefore, $\sigma_1=\sigma_2$, which is a contradiction. 
           \item If $u_1=u_2$,  $v_1=v_2'$ and $v_1'=v_2$, then we get the edge equalitites $$x(1,a)yz=\alpha\beta(1,d)\gamma \text{ and } xy(1,b)z=\alpha(1,c)\beta\gamma.$$ Clearly this means $a=d$ and $b=c$. Suppose first $y\neq 1$, so the first equation above is in canonical form, and hence $x=\alpha\beta$ and $\gamma=yz$. Then $$xy(1,b)z=\alpha(1,b)\beta\gamma=\alpha(1,b)\beta yz\Rightarrow xy(1,b)=\alpha(1,b)\beta y$$
        \end{enumerate}
        But, by Proposition \ref{prop:simple_equation_words}, this implies $\beta y=b^t$ for some $t\geq 0$, and because $\sigma_2$ is in canonical form, it must be that $\beta=1$ and $y=b^t$. So $x=\alpha$ and $\gamma=b^tz$, thus $$\sigma_1=x(1,a)b^t(1,b)z\text{ and }\sigma_2=x(1,b)(1,a)b^tz.$$
        By redefining $x, y$, this solution can be written in the form 
        \begin{equation}\label{eq:solution_squares_case1}
            \sigma_1=x(1,a)(1,b)y\text{ and }\sigma_2=x(1,b)(1,a)y.
        \end{equation}
        The case $\beta\neq 1$ is analogous, and produces the same solution (\ref{eq:solution_squares_case1}). Suppose now $y=\beta=1$, by Proposition \ref{prop:simple_equation_words}, assuming without loss of generality that $|\alpha|>|x|$, we must have $\alpha=xa^t$ and $z=a^t\gamma$ for some $t\in\N$. But then, since the second edge equation reduces to $x(1,b)a^t\gamma=xa^t(1,b)\gamma$, we must have $t=0$, so $x=\alpha$ and $z=\gamma$, which again produces the same solution as in equation (\ref{eq:solution_squares_case1}). 
        
        \item The common edges intersect at vertex $w_1=w_2$. So $v_1=v_2$ and $v_1'=v_2'$, or, $v_1=v_2'$ and $v_1'=v_2$. 

        \begin{enumerate}
            \item If $w_1=w_2$, $v_1=v_2$ and $v_1'=v_2'$, then we get the edge equations: $$xay(1,b)z=\alpha c\beta (1,d)\gamma \text{ and } x(1,a)ybz=\alpha(1,c)\beta d\gamma.$$
            Solving these equations produces $\sigma_1=\sigma_2$, which is a contradiction. 
            \item If $w_1=w_2$, $v_1=v_2'$ and $v_1'=v_2$, then we get the edge equations: $$xay(1,b)z=\alpha(1,c)\beta \gamma \text{ and } x(1,a)ybz=\alpha c\beta(1, d)\gamma.$$
            Which imply $\alpha=\alpha a\beta ay$, also producing a contradiction.
        \end{enumerate}

        \item The common edges intersect at $v_1$ or $v_1'$. Assume, without loss of generality, that they intersect at $v_1$, so $v_1=v_2$ or $v_1=v_2'$. 

        \begin{enumerate}
            \item If $u_1=u_2$, $v_1=v_2$ and $w_1=w_2$, we get the edge equations $$x(1,a)yz=\alpha(1,c)\beta\gamma\text{ and } xay(1,b)z=\alpha c\beta (1,d)\gamma.$$
            Solving these equations produces $\sigma_1=\sigma_2$, which is a contradiction. 

            \item If $u_1=u_2$, $v_1=v_2'$ and $w_1=w_2$, we get the edge equations $$x(1,a)yz=\alpha\beta(1,d)\gamma\text{ and } xay(1,b)z=\alpha (1,c)\beta \gamma.$$
            Solving this equations produces the final solution $$\sigma_1=x(1,a)(1,b)a\gamma\text{ and }\sigma_2=xa(1,b)(1,a)\gamma.$$\qedhere
        \end{enumerate}
    \end{enumerate}
\end{proof}

We leverage Lemma \ref{lemma:squares_sharing_edges} to identify a short list of forbidden patterns. To describe them precisely, we adopt the following notation for a valid 2-block $\sigma$, using its vertices. We write $\sigma = (v_0, \{v_1, v_2\}, v_3)$ if  
\begin{itemize}
    \item $V_\emptyset(\sigma) = v_0$,  
    \item $V_{[2]}(\sigma) = v_3$, and  
    \item $\{V_{\{1\}}(\sigma), V_{\{2\}}(\sigma)\} = \{v_1, v_2\}$.
\end{itemize}

\begin{lemma}\label{lemma:forbidden_patterns}
There does not exist any set of words $W \subset \Sigma^*$ such that its insertion complex $\Cins{W}$ satisfies any of the following conditions:
\begin{enumerate}
    \item It contains two distinct 2-blocks $\sigma_1, \sigma_2$ such that $\sigma_1 = (v_0, \{v_1, v_2\}, v_3)$ and $\sigma_2 = (u, \{v_1, v_2\}, v_3)$ for some $u \ne v_0$.
    \item It contains three distinct 2-blocks given by \begin{align*}
        \sigma_1&=(v_0, \{v_1, v_2\}, v_3),\\
        \sigma_1&=(v_0, \{v_1, u\}, v_3), \text{ and}\\
        \sigma_1&=(v_0, \{u, v_2\}, v_3),
    \end{align*} 
    For some $u\neq v_1, v_2$. 
    \item It contains four distinct 2-blocks given by 
    \begin{align*}
        \sigma_1&=(v_0, \{v_1, v_2\}, v_3),\\
        \sigma_2&=(v_0, \{v_1, u\}, v_3),\\
        \sigma_3&=(v_0, \{w, v_2\}, v_3),\text{ and} \\
        \sigma_4&=(v_0, \{u,w\}, v_3). 
    \end{align*}
    For some $u,w, u_1, u_2$ distinct. 
\end{enumerate}
\end{lemma}

\begin{proof}
 Note that Case 1 follows directly from Lemma~\ref{lemma:squares_sharing_edges}, which states that no two distinct squares in the complex can share both of their upper facets (edges). 
 Cases 2 and 3 are similar in structure and can be analyzed using the same reasoning. We provide the proof for Case 2 and leave Case 3 to the reader.

Observe that each pair of squares shares exactly two consecutive edges (in the sense of the underlying digraph), allowing us to relabel the squares without loss of generality. Since all 2-blocks are of the form $x_0(1,a_1)x_1(1,a_2)x_2$, we may assume that the words involved do not share any nontrivial common prefix. After possibly renaming, we assume that $\sigma_1$ begins with $x_0 = 1$.

Applying Lemma~\ref{lemma:squares_sharing_edges} to $\sigma_1$ and $\sigma_2$, we get  
$$\sigma_1 = (1,a)(1,b)ay \quad \text{and} \quad \sigma_2 = a(1,b)(1,a)y,$$
for some distinct symbols $a \neq b$ and a word $y \in \Sigma^*$.  

Now, applying the lemma again to $\sigma_1$ and $\sigma_3$, we find that  
$$\sigma_3 = a(1,b)(1,a)y = \sigma_2,$$ 
which contradicts the assumption that $\sigma_2$ and $\sigma_3$ are distinct, completing the proof of Case 2.
\end{proof}

Our code is implemented using the open-source software SageMath \cite{sagemath} and is hosted on the GitHub repository \cite{minimial2spheregithub}. 

Our code examines all possible 2-dimensional complexes with at most 7 vertices, and checks various properties to determine whether they could exist as the Chain Complex of a set of words, discarding all but 5 possible patterns with $n=7$ vertices. Finally, a manual analysis of these 5 patterns, shows they are also impossible, proving $\mu(\S^2)\geq 8$. We outline these steps at a high level.

\begin{enumerate}
    \item We generate all possible 1-skeletons (graphs) with $n=5,6,7$ vertices using the function\\ \texttt{generate\_1\_skeletons(n)}. 
    \item We generate all possible 2-skeletons with the given list of possible 1-skeletons (\texttt{Possible\_Graphs}) that have nontrivial second homology, and that don't contain any of the forbidden patterns from Lemma \ref{lemma:forbidden_patterns}, using the function \texttt{generate\_2\_skeletons(Possible\_Graphs, forbidden\_patterns)}. This function only returns the corresponding sub-complexes with non-zero second homology, as pairs of the form \texttt{(Graph, squares)}, where \texttt{Graph} is the 1-skeleton, and \texttt{squares} is the list of 2-faces. 
    \item For the 2-complexes that still remain after the previous two steps, we count how many different symbols must be inserted in the squares using the function \texttt{count\_inserted\_symbols(G,squares)} and discard those that require only one single symbol inserted, since the desired complex cannot be obtained in this way. To see this, note first that the link of the longest word (and in fact, of any word) needs to be a cycle. Such cycle can only be of sizes 4 or 6, since the 1-skeleton is bipartite and $n\leq 7$. If it was of size 4, it would need to be as described in Theorem \ref{thm:classification_cycles_1dim}, but none of those cycles can be obtained with $a=b$. If it was instead of size 6, it would mean the longest word has 3 incident edges, meaning it gets an insertion of a symbol $a$ at three designated positions. It is straightforward to see then that, up to common affixes, the vertices would be of the form $\{axaya, axay,axya,xaya,axy,xay,xya\}$, for some words $x,y\in\Sigma^*$ with $x,y\not\in \{a\}^*$. But this is also a contradiction, since the second homology of the corresponding Insertion Complex is zero. 
    \item Finally, 5 possible 2-skeletons survive, but applying Lemma \ref{lemma:squares_sharing_edges} to the pairs of squares sharing 2 edges, ultimately shows that these are either impossible or that they have vanishing 2-homology. The complete details may be found in the Jupyter Notebook document. 
\end{enumerate}
	
\end{document}